\documentclass[12pt]{elsarticle}
\usepackage{a4wide}
\usepackage[color]{changebar}
\usepackage{exscale,amssymb}
\usepackage{graphicx}
\usepackage{url,amsmath}
\usepackage{theorem}
\usepackage{natbib}
\usepackage{enumerate}
\usepackage{mathrsfs}
\usepackage{tikz}
\usepackage{cases}
\allowdisplaybreaks[4]

\newtheorem{theorem}{Theorem}[section]
\newtheorem{definition}{Definition}[section]
\newtheorem{lemma}{Lemma}[section]
\newtheorem{proposition}{Proposition}[section]

\newtheorem{remark}{Remark}[section]
\newtheorem{corollary}{Corollary}[section]
\newtheorem{assumption}{Assumption}[section]
\newtheorem{method}{Method}[section]

\usepackage{multirow,booktabs,colortbl,tabularx}    
\title{
Block Coordinate Descent 
Only Converge to Minimizers
}

\author{Enbin Song$^*$,
Zhubin Shen$^*$ and
Qingjiang Shi$^\dag$\footnote{Email: e.b.song@163.com (E. B. Song);
qing.j.shi@gmail.com (Q. Shi)}
\\ \small \it  $^*$\hspace*{-0.5mm}Department of Mathematics, Sichuan
University, Chengdu, China\\
 \small \it  $^\dag$\hspace*{-0.5mm}College of EIE, Nanjing University of Aeronautics and Astronautics, Nanjing, China
}

\def\comment#1{}

\bibpunct{(}{)}{;}{a}{,}{,}
\cbcolor{red}  

\newcounter{mylineno}
\makeatletter
\let\oldtabcr\@tabcr

\def\mynewline{\refstepcounter{mylineno}%
                \llap{\footnotesize\arabic{mylineno}\hspace{5pt}}%
               }

\gdef\@tabcr{\@stopline \@ifstar{\penalty%
            \@M \@xtabcr}\@xtabcr\mynewline}

\makeatother

\makeatletter
\def\mymatrix#1{\null\,\vcenter{\normalbaselines\m@th
    \ialign{\hfil$##$\hfil&&\quad\hfil$##$\hfil\crcr
      \mathstrut\crcr\noalign{\kern-\baselineskip}
      #1\crcr\mathstrut\crcr\noalign{\kern-\baselineskip}}}\,}
\makeatother

\begin{document}
\begin{frontmatter}


\begin{abstract}
Given a non-convex twice continuously differentiable cost
function with Lipschitz continuous gradient, we prove that
all of block coordinate gradient descent, block mirror descent
and proximal block coordinate descent
converge to a local minimizer, almost surely with
random initialization. 
Furthermore, we show that these results also hold true even
for the cost functions with non-isolated critical points.
\end{abstract}

\end{frontmatter}

\section{Introduction}\label{Introduction}
A main source of difficulty for non-convex
optimization over continuous spaces is
the proliferation of saddle points. Actually,
it is easy to find some instances where
bad initialization of the gradient descent
converges to unfavorable saddle point
\citep[Section 1.2.3]{Nesterov2004Introductory}.
Although there exist such worst case instances
 in theory, many simple algorithms including
the  first order algorithms and their variants,
 perform extremely well in terms
of the quality of solutions of
continuous optimization.

\subsection{\bf Related work}
Recently, a
milestone result of the gradient descent was established by
 \cite{JLeeandMSimchowitz2016}. That is, the gradient descent method converges to a local
 minimizer, almost surely with random initialization,
 by resorting to a tool from topology of dynamical systems.
 \cite{JLeeandMSimchowitz2016} assume that a cost function satisfies the
 {\it strict saddle  property}. Equivalently, each critical
 point\footnote{$x$ is a critical point of $f$ if
 $\nabla f(x)={\bf 0}$.} $x$ of $f$ is either a local
 minimizer, or a ``strict saddle", i.e.,
 $\nabla^2 f(x)$ has at least
 one strictly negative eigenvalue.
They  demonstrated that if
  $f:\mathbb{R}^n\rightarrow \mathbb{R}$ is
  a twice continuously differentiable function
  whose gradient is Lipschitz continuous
  with constant $L$, then
  the gradient descent with a sufficiently
small constant step-size
  $\alpha$ (i.e.,
  $x_{k+1}=x_k-\alpha \nabla f(x_k)$
  and $0<\alpha<\frac{1}{L}$)
and a random initialization
converges to a local minimizer
or negative infinity almost surely.

There is a followup work given by
\cite{Panageas2016Gradient}, which  firstly
proven that
the results in \cite{JLeeandMSimchowitz2016}
do hold true even for cost function $f$
with non-isolated critical points. One key
tool they used is that
for every open cover there is a countable
subcover in $\mathbb{R}^n$. Moreover,
\cite{Panageas2016Gradient} have shown
the globally Lipschitz assumption can be
circumvented as long as the domain is convex and
forward invariant with respect to gradient descent.
In addition, they also provided an upper bound on the
allowable step-size (such that those results hold true).

There are some prior works showing that
first-order descent methods can indeed escape strict
saddle points with the assistance of near isotropic
noise. Specifically, \cite{Pemantle1990Nonconvergence} established
convergence of the Robbins-Monro stochastic approximation to local
minimizers for strict saddle functions
and \cite{Kleinberg2009Multiplicative} demonstrated that
the perturbed versions of multiplicative weights
algorithm can converge to local minima in generic potential
 games. In particular,
 \cite{Ge2015Escaping} quantified the convergence rate
 of noise-added stochastic gradient descent to local minima.
Note that the aforementioned methods requires the
assistance of isotropic noise, which can significantly slowdown the
convergence rate when the problem parameters dimension
is large. In contrast, our
setting is deterministic and corresponds to
simple implementations of block coordinate gradient descent, block mirror descent
and proximal block coordinate descent.

\subsection{\bf Our contribution}

In this paper, under the assumption that $f$ is twice
continuously differentiable function
with a Lipschitz continuous gradient
over $\mathbb{R}^n$,  we prove that
all the following block
coordinate descent methods with constant
step size,
  \begin{enumerate}[(1)]
\item block coordinate gradient descent;
\item block mirror descent; and
\item proximal block coordinate descent,
\end{enumerate}
 converge to a local minimizer, almost
surely with random initialization. This result even holds for
cost function with non-isolated critical points.

In other words, we not only affirmatively answer the
open questions whether mirror descent or block coordinate
descent does not converge to saddle points in
 \cite{JLeeandMSimchowitz2016}, but also show that the same results hold true for
 block mirror descent (including mirror descent as a special case) and proximal block coordinate descent.
In addition, these results also hold true even
for the cost functions with non-isolated critical points,
which generalizes the results in
\cite{Panageas2016Gradient} as well.




\subsection{\bf Outline of the proof}
Recall the main ideas in
\cite{JLeeandMSimchowitz2016}.
Suppose that $g:\mathbb{R}^n\rightarrow
\mathbb{R}^n$
is an iterative mapping of an optimization method
and the fixed point of $g$ is the critical point
of $f$ as well.
Then the following key properties of
$g$,
\begin{enumerate}[(i)]
\item $g$ is a diffeomorphism;
\item  $x^*$ is a fixed point of $g$, then
 there is at least one eigenvalue of the
Jacobian $Dg(x^*)$,
whose magnitude is strictly greater than one,
\end{enumerate}
imply the results in \cite{JLeeandMSimchowitz2016}
hold true.

Although the basic idea of the proof
presented in this paper comes from \cite{JLeeandMSimchowitz2016},
answering the underlying questions for block coordinate type algorithms is not easy and
the existing analysis methods are not applicable.
In particular, the eigenvalue analysis of
the Jacobian of the iterative mapping
needs a nontrivial argument.

As compared to Property (ii), Property (i) is easier to verify since we can decompose
the entire iterative mapping into multiple one-block updating case and then using the chain rule of diffeomorphism,
we prove the entire updating is a diffeomorphism. In contrast, Property (ii) is very challenging because the Jocabian $Dg(x^*)$
of block updating fashion at a saddle point is a non-symmetric matrix and a complex polynominal function of the original $\nabla f(x^*)$ with degree $p$ (the number of the blocks of the decision variables). We overcome the above difficulties by the following two steps. Precisely, the first step is to
transform the  original Jacobian
$Dg_\text{\tiny$\alpha f$}(x^*)$  into a simpler form which can be handled easily.
Next, based on the simple form of
$Dg_\text{\tiny$\alpha f$}(x^*)$, the second step is
to prove that $Dg(x^*)$ has at least
one eigenvalue with magnitude strictly
greater than one by resorting to Lemma
\ref{Zero Lemma} in appendices which follows
essentially from Rouche's Theorem in complex
analysis.
%
%

\subsection{\bf Notations and organization}
{\bf Notations.}
Denote complex number $z$ as $z=a + bi$, where $a$ and $b$ are real numbers
and $i$ is the imaginary unit with
$i^2=-1$. We also denote $a=\mbox{Re}(z)$
and  $b=\mbox{Im}(z)$ as the
real part and  the imaginary part of $z$, respectively.
 For a matrix $X$, we denote $\mbox{eig}(X)$ as
the set of eigenvalues of $X$, $X^T$ as the
transpose of $X$, $X^H$ as the  conjugate transpose
or Hermitian transpose of $X$,
$\rho(X)$ as the spectral radius of
$X$ (i.e., the maximum modulus of the eigenvalues of
$X$), and $\|X\|$ as the spectral norm of $X$. When $X$ is a real symmetric matrix,
let $\lambda_{\max}(X)$ and
$\lambda_{\min}(X)$ denote the maximum and minimum
eigenvalues of $X$, respectively.
Moreover,  for two real symmetric
matrices $X_1$ and $X_2$, $X_1\succ X_2$
(resp. $X_1 \succeq X_2$) means $X_1-X_2$ is positive
definite (resp. positive semi-definite). We use
$I_n$ to denote the identity matrix with dimension $n$,
and we will simply use $I$ when it is clear from
context what the dimension is. For square matrices
$X_s\in \mathbb{R}^{n_s\times n_s}$,
$s=1, 2, \ldots, p$, we denote $\mbox{Diag}
\left(X_1, X_2, \ldots, X_p\right)$
 as the block-diagonal matrix with $X_s$
being the $s$-th diagonal block.
For square matrices
$X_s\in \mathbb{R}^{n\times n}$,
$s=1, \ldots, p$, and $t, k\in\left\{1, 2, \ldots, p\right\}$,
 we use
$\prod\limits_{s=t}^{k}
X_s$
to denote the continued products
$X_t\cdot X_{t+1}\ldots X_{k-1}\cdot X_k$ if $t \leq k$
and
$X_t\cdot X_{t-1}\ldots X_{k+1}\cdot X_k$ if $t> k$,
respectively.
$\mathbb{P}_\nu$ denotes the probability
 with respect to a prior measure $\nu$, which is assumed to be
 absolutely continuous with respect to Lebesgue measure.\\

{\bf Organization.}
In section 2, we introduce the basic setting and
definitions used throughout the paper. Section 3
provides the main results for block coordinate gradient
descent method. The main results for block mirror descent and proximal block coordinate descent
 are given in Section 4 and Section 5, respectively.  Section 6 provides several lemmas. Finally, we conclude this paper in Section 7. The detailed proofs of some lemmas and propositions are presented in Section 9 .
\\

\section{Preliminaries}
\label{Preliminaries}
We consider optimization model
\begin{equation}
\label{Model}
\min\left\{f\left(x \right): x \in \mathbb{R}^n\right\},
\end{equation}
where we make the following blanket assumption:
\begin{assumption}\label{Assumption f}
$f$ is twice continuously differentiable
function whose gradient is Lipschitz  continuous over $\mathbb{R}^n$,
i.e.,
there exists a parameter $L>0$ such that
\begin{equation}\label{Lipschitz}
\left\|\nabla f(x)-\nabla f(y)\right\|
\leq L \left\|x-y\right\|
\hspace*{4mm}\mbox{for every} \hspace*{1mm}
x,
\, y\,\in \mathbb{R}^n.
\end{equation}
\end{assumption}
Throughout this paper, in order to introduce
the block coordinate descent  method,
we assume the vector of decision variables
 $x$ has the following partition:
\begin{equation}\label{x Partion}
x=\left(\begin{array}{c}
x\text{\scriptsize$(1)$}\\
x\text{\scriptsize$(2)$}\\
\vdots\\
x\text{\scriptsize$(p)$}
\end{array}
\right),
\end{equation}
where
$x\text{\scriptsize$(s)$}\in \mathbb{R}^{n_s}$,  $n_1$, $n_2$, $\ldots$,
$n_p$ are $p$
positive integer numbers satisfying
$\sum\limits_{s=1}^{p}n_s=n$.

Moreover, we use the notations in
\cite{Nesterov2012Effici} and define matrices
$U_s\in \mathbb{R}^{n \times n_s}$,
 the $s$-th block-column of $I_n$,
$s=1, \ldots, p$,  such that
\begin{equation}\label{DefUi}
\left(\begin{array}{cc}
U_1, U_2, \ldots, U_p
\end{array}
\right)=I_n.
\end{equation}
Clearly, according to our notations, we have
$x\text{\scriptsize$(s)$}=U^T_sx$ for every $x\in \mathbb{R}^n$,
$s=1, \ldots, p$.
Consequently,
$x=\sum\limits_{s=1}^p U_sx\text{\scriptsize$(s)$}$ and the derivative
corresponding variables in the vector $x\text{\scriptsize$(s)$}$
can be expressed as
$$\nabla_s f(x)\equiv
U_s^T\nabla f(x),
\hspace*{4mm}s=1, \ldots, p.$$

Below we  give some necessary
definitions as appeared in
\cite{JLeeandMSimchowitz2016}
and \cite{Panageas2016Gradient}.
\begin{definition}\mbox{}\par
\begin{enumerate}[1.]
\item A point $x^*$ is a critical point of $f$
if $\nabla f(x^*)={\bf 0}$. We denote
$C=\left\{x: \nabla f(x)={\bf 0}\right\}$ as
the set of critical points
(can be uncountably many).
\item A critical point $x^{\ast}$ is isolated if
there is a neighborhood $U$ around $x^{\ast},$
and $x^{\ast}$ is the only
critical point in $U$\footnote{If the critical points are
isolated then they are countably many or finite.}.
Otherwise is called non-isolated.
\item A critical point is a local minimum
if there is a neighborhood $U$ around $x^{\ast}$
such that $f(x^{\ast}) \leq f(x)$ for all $x\in U$,
and a local maximum if $f(x^{\ast}) \geq f(x)$.
\item  A critical point is a saddle point if for all
neighborhoods $U$ around $x^{\ast}$,
there are $x,y\in U$
such that $f(x)\leq f(x^{\ast}) \leq f(y)$.
\end{enumerate}
\end{definition}
\begin{definition}[Strict Saddle] A critical point $x^{\ast}$ of $f$ is a strict
saddle if\\ $\lambda_{\text{min}}(\nabla^{2}f(x^{\ast}))<0$.
\end{definition}
\begin{definition}[Global Stable Set]\label{Global Stable Set} The global stable
set $W^{s}(x^{\ast})$ of a critical point $x^{\ast}$
is the set of initial conditions of
an iterative mapping $g$ of an optimization
method  that converge
to $x^{\ast}$:
$$W^{s}(x^{\ast})=\{x:\lim_{k} g^{k}(x)=x^{\ast}\}.$$
\end{definition}

\section{The BCGD  method}\label{BCGD Sec}
In this section, we will prove that the BCGD method does not
converge to saddle  points under appropriate
choice of step size, almost surely with random initialization. This lays the ground for the analysis of the
BMD and
PBCD methods.

\subsection{\bf The BCGD method description}
For ease of later reference and also for the sake
of clarity, based on notations in Section \ref{Preliminaries},
we present a  detailed description of the BCGD
method \citep{Beck2013On} for problem \eqref{Model} below.
\medskip

\noindent \fbox{\begin{minipage}{14cm}
\begin{method}[BCGD]
\label{BCGD 3.1}
\end{method}
{\bf  Input:} $\alpha<\frac{1}{L}$.\\
{\bf  Initialization:}
$x_0\in \mathbb{R}^{n}.$\\
{\bf General Step} $\left(k=0, 1, \ldots\right)$:
Set $x_k^0=x_k$ and define recursively \\
\begin{equation*}\label{Algorithm eq1}
x_k^s=x_k^{s-1}-\alpha U_s \nabla_s f(x_k^{s-1}),
\hspace*{3mm} s=1, \ldots, p.
\end{equation*}
Set $x_{k+1}=x_k^p$.
\end{minipage}}
\\

In what follows, for a given step size $\alpha>0$,
we use  $g_{\text{\tiny$\alpha f$}}^s$ to denote
the corresponding gradient mapping with respect to
 $x\text{\scriptsize$(s)$}$,
i.e.,
\begin{equation}\label{Degi}
g_{\text{\tiny$\alpha f$}}^s(x)
\triangleq x-\alpha U_s \nabla_s f(x),
 \hspace*{3mm} s=1, \ldots, p.
\end{equation}
 It is clear that, given $x_k$, the above BCGD method generates
 $x_{k+1}$ in the following manner,
  \begin{equation}
x_{k+1}=g_{\text{\tiny$\alpha f$}}\left(x_{k}\right),
\end{equation}
where composite mapping
 \begin{equation}\label{Define g}
g_{\text{\tiny$\alpha f$}}(x)
\triangleq g_{\text{\tiny$\alpha f$}}^p
\circ
g_{\text{\tiny$\alpha f$}}^{p-1}
\circ
\cdots
\circ
g_{\text{\tiny$\alpha f$}}^2
\circ
g_{\text{\tiny$\alpha f$}}^1
\left(x\right).
\end{equation}
By simple computation, the Jacobian of
 $g_{\alpha f}^s$ is given by
 \begin{equation}\label{Jacob giG}
 Dg_{\alpha f}^s(x)=I_n-\alpha \nabla^2 f(x)U_{s}U_s^T,
 \end{equation}
 where $
 \nabla^2 f(x)=
 \left(
 \frac{\partial^2 f(x)}{\partial x\text{\scriptsize$(s)$}
\partial x\text{\scriptsize$(t)$}}
\right)_{1\leq s,\, t\leq p}.
$\\
By using the chain rule, we obtain the following Jacobian
 of the mapping $g$, i.e.,
  \begin{equation}\label{Chain Rule}
Dg_{\text{\tiny$\alpha f$}}\left(x\right)
=
Dg_{\text{\tiny$\alpha f$}}^{1}\left(y_1\right)
\times
Dg_{\text{\tiny$\alpha f$}}^{2}\left(y_2\right)
\times \cdots \cdot \times
Dg_{\text{\tiny$\alpha f$}}^{p-1}\left(y_{p-1}\right)
\times
Dg_{\text{\tiny$\alpha f$}}^{p}\left(y_p\right),
\end{equation}
 where $y_1=x$, and $y_s=g_{\text{\tiny$\alpha f$}}^{s-1}(y_{s-1})$,
 $s=2, \ldots, p$.

Given the above basic notations, as mentioned in
Section \ref{Introduction}, it is sufficient
for us to prove that the iterative mapping
$g_{\text{\tiny$\alpha f$}}$
admits the following two key properties:
\begin{enumerate}[(i)]
\item $g_{\text{\tiny$\alpha f$}}$ is a diffeomorphism;
\item  $x^*$ is a fixed point of $g$, then
 there is at least one eigenvalue of the
Jacobian $Dg_{\text{\tiny$\alpha f$}}(x^*)$,
whose magnitude is strictly greater than one.
\end{enumerate}
In what follows,  we will mainly concern with
proving the above two properties of $g_{\text{\tiny$\alpha f$}}$.
Specifically, Subsections \ref{SubSecDiff}
and \ref{BCGDSubSecJacobia} present their proofs,
respectively.

\subsection{{\bf The iterative mapping
 $g_\text{\tiny$\alpha f$}$  of BCGD method is a
 diffeomorphism}}\label{SubSecDiff}
In this subsection, we first present the following
Lemma \ref{g12Diff} which shows that
$g_{\text{\tiny$\alpha f$}}^s, s=1, \ldots, p,$
are diffeomorphisms. Based on this lemma, then we
further prove that $g_\text{\tiny$\alpha f$}$  is
a diffeomorphism as well.

\begin{lemma}\label{g12Diff}
Under Assumption \ref{Assumption f},
if step size $\alpha<\frac{1}{L}$,
then the iterative mappings
$g_{\text{\tiny$\alpha f$}}^s$ defined
by \eqref{Degi}, $s=1, \ldots, p,$ are diffeomorphisms.
\end{lemma}
\begin{proof}
The proof of Lemma \ref{g12Diff} is lengthy and has
been relegated to the Appendix.
\hfill\end{proof}

Given the above lemma, we immediately obtain the
following result.
\begin{proposition}\label{gDiff}
Under Assumption \ref{Assumption f}, the mapping
$g_{\text{\tiny$\alpha f$}}
 $ defined by \eqref{Define g}
with step size
$\alpha<\frac{1}{L}$ is a diffeomorphism.
\end{proposition}
\begin{proof}
It follows from Proposition 2.15 in \cite{Lee2013} that the
composition of two diffeomorphisms is also a diffeomorphism.
Using this fact and the previous Lemma \ref{g12Diff},
the proof is completed.
\hfill\end{proof}

\begin{remark}
Note that, in order to guarantee  invertibility of
$Dg_{\text{\tiny$\alpha f$}}(x)$,
Eq. \eqref{rangleLL} in the the appendix must hold, which clearly implies
that $\alpha<\frac{1}{L}$ is  necessary.

\end{remark}
\subsection{{\bf Eigenvalue analysis of the Jacobian of
$g_\text{\tiny$\alpha f$}$ at a strict saddle point}}
\label{BCGDSubSecJacobia}
In this subsection, we will analyze the eigenvalues
 of the Jacobian of $g_\text{\tiny$\alpha f$}$
at a strict saddle point and show that it has at least
one eigenvalue with magnitude greater than one, which is a crucial part in our
entire proof. The proof mainly involves
two steps.

More specifically, the first step is to
transform the  original Jacobian
$Dg_\text{\tiny$\alpha f$}(x^*)$  into a simpler form which can be dealt with more easily.
Furthermore, based on the simple form of
$Dg_\text{\tiny$\alpha f$}(x^*)$, the second step is
to prove that $Dg(x^*)$ has at least
one eigenvalue with magnitude strictly
greater than one by resorting to Lemma
\ref{Zero Lemma} in Appendix which follows
essentially from Rouche's Theorem in complex
analysis.

In what follows, we assume that
$x^*\in\mathbb{R}^n$ is a strict saddle point.
Hence, $g_\text{\tiny$\alpha f$}^s(x^*)=x^*$, $s=1, 2, \ldots, p,$.
By  chain rule
\eqref{Chain Rule}, we have
\begin{equation}\label{Dgx*}
Dg_\text{\tiny$\alpha f$}(x^*)
=Dg_\text{\tiny$\alpha f$}^1(x^*)
\times
Dg_\text{\tiny$\alpha f$}^{2}\left(x^*\right)
\times \cdots  \times
Dg_\text{\tiny$\alpha f$}^{p-1}(x^*)\times
Dg_\text{\tiny$\alpha f$}^{p}
\left(x^*\right).
\end{equation}

Since $\mbox{eig}\left(Dg_\text{\tiny$\alpha f$}(x^*)\right)=
\mbox{eig}\left(\left(Dg_\text{\tiny$\alpha f$}(x^*)\right)^T\right)$,
 it is sufficient for us to
analyze the eigenvalues of matrix
$\left(Dg_\text{\tiny$\alpha f$}(x^*)\right)^T$.
In addition,  Eq. \eqref{Dgx*} leads to
\begin{equation}\label{Dgx*T}
\begin{split}
\left(Dg_\text{\tiny$\alpha f$}(x^*)\right)^T
&=\left(Dg_\text{\tiny$\alpha f$}^p(x^*)\right)^T
\times
\left(Dg_\text{\tiny$\alpha f$}^{p-1}(x^*)\right)^T
\times \cdots \times
\left(Dg_\text{\tiny$\alpha f$}^2(x^*)\right)^T
\times
\left(Dg_\text{\tiny$\alpha f$}^1(x^*)\right)^T
\\
&=
\left(I_n-\alpha U_pU_p^T \nabla^2 f(x^*)\right)
\times
\left(I_n-\alpha U_{p-1}U_{p-1}^T \nabla^2 f(x^*)\right)
\\
&\hspace*{4mm}\times \cdots\times
\left(I_n-\alpha U_2U_2^T \nabla^2 f(x^*)\right)
\times
\left(I_n-\alpha U_1U_1^T \nabla^2 f(x^*)\right),
\end{split}
\end{equation}
where the second equality is due to
\eqref{Jacob giG}.
For the sake of simplicity, we
denote matrix $\nabla^2 f(x^*)$ as matrix
$A\in\mathbb{R}^{n\times n}$ with $p\times p$ blocks.
Specifically,
\begin{equation}\label{Def A}
A\triangleq\left(A_{st}\right)_{ 1\leq s,\, t\leq p},
\end{equation}
and its $(s, t)$-th block is defined as
\begin{equation}\label{Def A_ij}
A_{st}\triangleq\frac{\partial^2 f(x^*)}
{\partial x^*\text{\scriptsize$(s)$}
\partial x^*\text{\scriptsize$(t)$}},
\hspace*{3mm} 1\leq s,\, t\leq p.
\end{equation}
Furthermore, we denote the $s$-th block-row of
$A$ as
\begin{equation}\label{Def A_i}
A_s\triangleq\left(A_{st}\right)_{ 1\leq t\leq p},
\hspace*{3mm}s=1, \ldots,p.
\end{equation}
Based on the above notations, we have
\begin{equation}\label{UiA}
U_s^T\nabla^2 f(x^*)=A_s,\hspace*{3mm}
 1\leq s\leq p,
\end{equation}
and
\begin{equation}\label{AiUj}
A_sU_t=A_{st},\hspace*{3mm}
 1\leq s, \,t\leq p.
\end{equation}
Combining \eqref{Dgx*T} and \eqref{UiA}, we have
 \begin{equation}\label{Dgx*T1}
\begin{split}
\left(Dg_\text{\tiny$\alpha f$}(x^*)\right)^T
&=
\left(I_n-\alpha U_pA_p\right)
\times
\left(I_n-\alpha U_{p-1}A_{p-1}\right)
\\
&\hspace*{4mm}\times \cdots\times
\left(I_n-\alpha U_2A_2\right)
\times
\left(I_n-\alpha U_1A_1\right).
\end{split}
\end{equation}
In order to analyze the eigenvalues of the above matrix,
we furthermore define a matrix below:
 \begin{equation}\label{Def G}
G\triangleq\frac{1}{\alpha}
\left[
I_n-
\left(Dg_\text{\tiny$\alpha f$}(x^*)\right)^T
\right],
\end{equation}
or equivalently,
 \begin{equation}\label{DGT G}
 \left(Dg_\text{\tiny$\alpha f$}(x^*)\right)^T=
 I_n-\alpha G.
\end{equation}
The above relation \eqref{DGT G} clearly means that
\begin{equation}\label{Eigrela}
\lambda
 \in\mbox{eig}
 \left(G\right)\Leftrightarrow
 1-\alpha \lambda \in
  \mbox{eig}\left(\left(Dg_\text{\tiny$\alpha f$}(x^*)\right)^T\right).
\end{equation}

In the rest of this subsection,  our focus is on
analyzing the eigenvalues of $G$. We first
introduce the following Lemma \ref{Property G} and Lemma \ref{Property G1} which
assert a particular relation
between $G$ and $A$.
\begin{lemma}\label{Property G}
Assume that
$x^*\in\mathbb{R}^n$ is a strict saddle point,
$G$ is given by \eqref{Def G},
$U_s$ and $A_s$ are defined by \eqref{DefUi} and
\eqref{Def A_i}, respectively, $s=1, \ldots,p$.
Then,
\begin{equation}\label{Porperty UiG}
U_s^TG
= A_s
-\alpha\sum\limits_{t=1}^{s-1}A_{st}U_t^TG,
\hspace*{3mm}s=1, \ldots,p.
\end{equation}
\end{lemma}
\begin{proof}
The proof of Lemma \ref{Property G} is lengthy and has
been relegated to the Appendix.
\hfill\end{proof}

In order to give a clear and simple expression
of $G$, we further define the strictly block
lower triangular  matrix based on $A$ below:
 \begin{equation}\label{Def bar A}
 \check{A}\triangleq
 \left(\check{A}_{st}\right)_{1\leq s,\,t\leq p}
 \end{equation}
 with $p\times p$ blocks and its $(s, t)$-th
 block is given by
 \begin{equation}\label{LsigmabarA}
\check{A}_{st}
=\begin{cases}
\begin{array}{lr}
    A_{st},  & s>t,\\
    \mathbf{0},& s\leq t.
  \end{array}
  \end{cases}
\end{equation}
Similarly, we denote the $s$-th block-row of
$\check{A}$ as
\begin{equation}\label{Def check A_i}
\check{A}_s\triangleq
\left(\check{A}_{st}\right)_{ 1\leq t\leq p},
\hspace*{3mm}s=1, \ldots,p.
\end{equation}

The following lemma plays an important role in this subsection
because it gives a simple expression  of $G$ in terms
of $A$ and $\check{A}$, which allows us to analyze
the eigenvalues of $G$ more easily.
\begin{lemma}\label{Property G1}
Let
$x^*\in\mathbb{R}^n$ be a strict saddle point. Assume that $G$, $A$ and
$\check{A}$ are defined by \eqref{Def G}, \eqref{Def A} and \eqref{Def bar A},
respectively. Then,
\begin{equation}\label{G Form}
G=\left(I_n+ \alpha \check{A}\right)^{-1}A.
\end{equation}
\end{lemma}
\begin{proof}
We assume that $G$ has the following partition:
\begin{equation}\nonumber
G=\left(\begin{array}{c}
G_1\\
G_2\\
\vdots\\
G_p
\end{array}
\right),
\end{equation}
where  $G_s\in \mathbb{R}^{n_s\times n}$
is the $s$-th block-row of $G$, $s=1, 2, \ldots, p.$
Consequently, we have
\begin{equation}\nonumber
\begin{split}
G_s&=U_s^TG
= A_s
-\alpha\sum\limits_{t=1}^{s-1}A_{st}U_t^TG
= A_s
-\alpha\sum\limits_{t=1}^{s-1}A_{st}G_t
= A_s
-\alpha \check{A}_sG, \hspace*{1mm}s=1, 2, \ldots, p,
\end{split}
\end{equation}
where in the second equality we use  Lemma \ref{Property G} and
the last equality is due to the definition of
$\check{A}_s$ in
\eqref{Def check A_i}.
Since the above equality holds for any $s$,
it immediately follows that
\begin{equation}\nonumber
G=A-\alpha \check{A}G,
\end{equation}
which is equivalent to
\begin{equation}\nonumber
\left(I_n+\alpha \check{A}\right)G=A.
\end{equation}
Note that $I_n+\alpha \check{A}$ is an invertible matrix
because $\check{A}$ is a block strictly lower triangular
matrix.
Premultiplying both sides of the above equality
by $\left(I_n+\alpha \check{A}\right)^{-1}$,
we arrive at \eqref{G Form}. Therefore, the lemma
is proved.
\hfill\end{proof}


The following proposition plays a central role in
 this subsection
because it provides a sufficiently exact description
of the distribution of the eigenvalues of $G$.
More importantly,
it leads  immediately to the subsequent Proposition
\ref{Jaco g neg eig} which asserts that,
 there is at least one eigenvalue of Jacobian of
iterative mapping $g_\text{\tiny$\alpha f$}$
at a strict saddle point,
whose magnitude is strictly greater than one.

\begin{proposition}\label{Eig G in Omega}
Assume that $x^*\in\mathbb{R}^n$ is a strict saddle point,
and $G$ is defined by \eqref{Def G} with
$\alpha\in\left(0, \frac{1}{L}\right)$, where $L$
is determined by \eqref{Lipschitz}. Then,
there exists at least one eigenvalue of
$G$ which
 lies in
closed left half complex plane excluding origin, i.e.,
\begin{equation}\label{Key0}
\forall\, \beta\in
\left(0, \frac{1}{L}\right)
 \Rightarrow
\exists \, \lambda\in \left[\mbox{eig}
\left(G\right)
\text{\scriptsize$\bigcap$}\, \Omega\right],
\end{equation}
where
\begin{equation}\label{Omega0}
\Omega\triangleq\left\{a+bi\big{|}a, b\in
\mathbb{R}, a\leq0,~(a,b)\neq(0,0),
i=\sqrt{-1} \right\}.
\end{equation}
\end{proposition}
\begin{proof}
It follows from Lemma \ref{Property G1} that  $G$ has the
following expression:
\begin{equation}\label{Re G Form}
G=\left(I_n+ \alpha \check{A}\right)^{-1}A,
\end{equation}
where $A$ and $\check{A}$ are defined by \eqref{Def A}
and \eqref{Def bar A}, respectively.
Since $A=\nabla^2 f\left(x^*\right)$ and $x^*$ is a strictly
saddle point, $A$ has at least one negative eigenvalue.
Furthermore, by applying Lemma \ref{Key Lemma BCGD} in Section
\ref{Several Lemmas} 
with identifications $A\sim B$, $\check{A}\sim\check{B}$,
$\alpha\sim \beta$ and $\rho(A)\sim \rho(B)$, we have
\begin{equation}\label{Key0 proof}
\forall\, \beta\in
\left(0, \frac{1}{\rho(A)}\right)
 \Rightarrow
\exists \, \lambda\in \left[\mbox{eig}
\left(G\right)
\text{\scriptsize$\bigcap$}\, \Omega\right],
\end{equation}
where $\Omega$ is defined by \eqref{Omega0}.

In addition, Lemma 7 in \cite{Panageas2016Gradient} implies
that the gradient Lipschitz  continuous
constant $L\geq \rho(A)=\rho\left(\nabla f (x^*)\right)$, which amounts to
$\alpha \in \left(0, \frac{1}{L}\right)
\subseteq\left(0, \frac{1}{\rho(A)}\right)$.
Hence, the above
\eqref{Key0 proof} leads immediately to \eqref{Key0}.
\hfill
\end{proof}

Based on the above Proposition \ref{Eig G in Omega},
we immediately obtain the key proposition in this
subsection.
\begin{proposition}\label{Jaco g neg eig}
Assume  Assumption \ref{Assumption f} holds.
Suppose that BCGD iterative mapping $g_\text{\tiny$\alpha f$}$ is defined
by \eqref{Define g} with
$\alpha\in \left(0,\,\frac{1}{L}\right)$,
 $L$ is determined by \eqref{Lipschitz}, and
$x^*\in\mathbb{R}^n$ is a strict saddle point. Then
$Dg_\text{\tiny$\alpha f$}(x^*)$ has at least one eigenvalue whose magnitude is
strictly greater than one.
\end{proposition}
\begin{proof}
Recalling Eq. \eqref{Eigrela}, we have
\begin{equation}\label{Recal Eigrela}
\lambda
 \in\mbox{eig}
 \left(G\right)\Leftrightarrow
 1-\alpha \lambda \in
  \mbox{eig}\left(\left(Dg_\text{\tiny$\alpha f$}
  (x^*)\right)^T\right),
\end{equation}
where 
$G=\frac{1}{\alpha}
\left[
I_n-
\left(Dg_\text{\tiny$\alpha f$}(x^*)\right)^T
\right]
$ is defined by \eqref{Def G}.

Combined with \eqref{Recal Eigrela},
Proposition  \ref{Eig G in Omega} implies that
there exists at least one eigenvalue of
$\left(Dg_\text{\tiny$\alpha f$}(x^*)\right)^T$ which can be
expressed as
 \begin{equation}
1-\alpha \left(a+bi\right),
\end{equation}
 where $a+bi$ belongs to $\Omega$ defined by \eqref{Omega}, or
 equivalently,
 \begin{equation}
 a\leq0 \hspace*{2mm}\text{and}\hspace*{2mm} \left(a, b\right)
\neq \left(0, 0\right).
 \end{equation}
 Consequently, its magnitude is
\begin{equation}\nonumber
\begin{split}
&\left|1-\alpha \left(a+bi\right)\right|
=\sqrt{1-2\alpha a+\alpha^2 a^2+\alpha^2 b^2}
\geq\sqrt{1+\alpha^2\left(a^2+b^2\right)}
>1,
\end{split}
\end{equation}
where the first inequality is due to $a\leq0$ and $\alpha>0$;
and the second inequality thanks to $\left(a, b\right)
\neq \left(0, 0\right)$. Note that the eigenvalues
of $Dg_\text{\tiny$\alpha f$}(x^*)$
are the same as those of
$\left(Dg_\text{\tiny$\alpha f$}(x^*)\right)^T$. Thus, the proof is finished.
\hfill\end{proof}
\subsection{\bf Main results of BCGD}\label{Main Subsec of BCGD}

We first introduce the following proposition, which asserts that
the limit point of the sequence generated by the
BCGD method \ref{BCGD 3.1} is a critical point of $f$.
\begin{proposition}\label{limit=critical point BCGD}
Under Assumption \ref{Assumption f}, if
  $\left\{x_k\right\}_{k\geq 0}$
   is generated by the BCGD method \ref{BCGD 3.1}
with $0<\alpha<\frac{1}{L}$, $\lim\limits_{k}{x_k}$
exists and denote it as $x^*$, then $x^*$ is a
critical point of $f$, i.e., $\nabla f(x^*)
={\bf 0}$.
\end{proposition}
\begin{proof}
Since $\left\{x_k\right\}_{k\geq 0}$ is generated
by the BCGD method \ref{BCGD 3.1},
 $x_k=g^k_\text{\tiny$\alpha f$}(x_0)
\footnote
{$g^k_\text{\tiny$\alpha f$}$
denotes the composition of
$g_\text{\tiny$\alpha f$}$
with itself $k$ times.}
$,
where $g_{\text{\tiny$\alpha f$}}$ is defined by
\eqref{Degi}. Hence, $\lim\limits_kx_k=\lim\limits_k
g^k_\text{\tiny$\alpha f$}(x_0)=x^*$. Since
$g_{\text{\tiny$\alpha f$}}$ is a  diffeomorphism,
we immediately know that $x^*$ is a
fixed point of $g_{\text{\tiny$\alpha f$}}$.
It follows easily from the definition \eqref{Define g} of
$g_{\text{\tiny$\alpha f$}}$ that
$\nabla f(x^*)={\bf 0}$. Thus the proof is finished. \hfill
\end{proof}
Armed with the results established in  previous
subsections and the above Proposition
\ref{limit=critical point BCGD}, we now state and prove
our main theorems of BCGD. Specifically, similar to
the proof of Theorem 4
in \cite{JLeeandMSimchowitz2016},
 the center-stable manifold theorem in
 \cite{Smale1967Differentiable,Michael1987,
 hirsch1977invariant} is a primary tool
 because it gives a local characterization of the stable
 set. Hence, we first rewrite it as follows.
\begin{theorem}\cite[Theorem III. 7]{Michael1987}
\label{Sabl Mani Theorem}
Let 0 be a fixed point for the $C^r$ local
differomorphism $\phi:U\longrightarrow E$ where
$U$ is a neighborhood of zero in the Banach space
$E$ and
$\infty >r\geq 1$. Let $E_{sc}\bigoplus  E_u$
be the invariant splitting of $R^n$ into
the generalized eigenspaces of $Df(0)$
corresponding to eigenvalues of absolute
value less than or equal to one, and greater
than one. Then there is a local $\phi$ invariant
$C^r$ embedded disc $W^{sc}_{loc}$ tangent
 to $E_{sc}$ at $0$ and a ball B around zero
 in an adapted norm such that
$\phi(W^{sc}_{loc})\bigcap B\subset W^{sc}_{loc}$,
and $\bigcap\phi^{-k}(B)\subset W^{sc}_{loc}$.
\end{theorem}
The following theorem is similar to
 Theorem 4 in \cite{JLeeandMSimchowitz2016}.

\begin{theorem}\label{Main Theorem BCGD}
Let $f$ be a $C^2$ function and $x^*$ be a strict
saddle. If  $\left\{x_k\right\}$ is generated
by the BCGD method \ref{BCGD 3.1} with $0<\alpha<\frac{1}{L}$, then
$$\mathbb{P}_\nu\left[\lim_{k}{x_k}=x^*\right]=0.$$
\end{theorem}
\begin{proof}
Proposition \ref{limit=critical point BCGD} implies that,
if $\lim\limits_{k}{x_k}$ exists then it must be a critical
point. Hence, we consider calculating the Lebesgue measure
(or probability with respect to the prior measure $\nu$)
of the set $\left[\lim\limits_{k}{x_k}=x^*\right]
 =W^s\left(x^*\right)$ (see Definition
 \ref{Global Stable Set}).

Furthermore, since Proposition \ref{gDiff} means
 the BCGD iterative mapping $g_\text{\tiny$\alpha f$}$  is
 a diffeomorphism, we replace $\phi$ and fixed point
 by $g_\text{\tiny$\alpha f$}$ and the strict
 point $x^*$ in the above Stable Manifold  Theorem
\ref{Sabl Mani Theorem}, respectively.
Then  the manifold $W_{loc}^{s}(x^*)$
has strictly positive codimension because of
Proposition \ref{Jaco g neg eig} and $x^*$ being a strict
saddle point. Hence, $W_{loc}^{s}(x^*)$ has measure zero.

In what follows,
we are able to apply the same arguments
as in \cite{JLeeandMSimchowitz2016}
to finish the proof of the theorem. Since
the proof follows a similar pattern,
it is omitted.
\hfill\end{proof}

Given the above Theorem \ref{Main Theorem BCGD},
we immediately  obtain the following
Theorem \ref{key theorem BCGD} and its
corollary by the same arguments as in the proofs
of Theorem 2 and Corollary 12 in
 \cite{Panageas2016Gradient}.
Therefore, we omit their proofs.
\begin{theorem}[Non-isolated]\label{key theorem BCGD}
Let $f: \mathbb{R}^n\rightarrow\mathbb{R}$ be a twice
continuously differentiable function
and $\sup\limits_{x\in\mathbb{R}^n}\left\|
\nabla f(x)\right\|\leq L<\infty$. The set of initial
points $x\in \mathbb{R}^n$, for each of which the BCGD
method \ref{BCGD 3.1}
with step size $0<\alpha<\frac{1}{L}$ converges
to a strict saddle point, is of (Lebesgue) measure zero, without assumption that critical points are isolated.
\end{theorem}
A straightforward corollary of Theorem
\ref{key theorem BCGD} is given below:
\begin{corollary}\label{key corolary BCGD}
Assume that the conditions of Theorem \ref{key theorem BCGD}
 are satisfied and all saddle points of $f$
are strict. Additionally, 
assume
$\lim\limits_kg^k_\text{\tiny$\alpha f$}(x)$
exists
for all $x$
in $\mathbb{R}^n$. Then
 $$
\mathbb{P}_\nu
\left[\lim\limits_kg^k_\text{\tiny$\alpha f$}(x)
=x^*\right]=1,$$
where $g_{\text{\tiny$\alpha f$}}$ is defined by
\eqref{Degi} and $x^*$ is a local minimum .
\end{corollary}

\section{The BMD method}\label{BMD sec}
In this section, we will extend the above results
to the BMD method in \cite{Beck2003Mirror} and
\cite{Juditsky2014Deterministic}. In other words, the BMD method, based on Bregman's divergence, converges to minimizers as well, almost surely with random initialization.
\subsection{{\bf The BMD method description}}
For clarity of notation,
recall  the vector of decision variables
 $x$ has been assumed to have the following
 partition (see \eqref{x Partion}):
\begin{equation}
x=\left(\begin{array}{c}
x\text{\scriptsize$(1)$}\\
x\text{\scriptsize$(2)$}\\
\vdots\\
x\text{\scriptsize$(p)$}
\end{array}
\right),
\end{equation}
where
$x\text{\scriptsize$(t)$}\in \mathbb{R}^{n_t}$, and $n_1$, $n_2$, $\ldots$,
$n_p$ are $p$
positive integer numbers satisfying
$\sum\limits_{t=1}^{p}n_t=n$.

Correspondingly,  we assume a set of variables $x_k^s$ have the following partition as well:
\begin{equation}\label{genexksBMDA}
\begin{split}
x_k^s&\triangleq
\left(\begin{array}{c}
x_{k}^s\text{\scriptsize$(1)$}\\
x_{k}^s\text{\scriptsize$(2)$}\\
\vdots\\
x_{k}^s\text{\scriptsize$(p)$}
\end{array}
\right),  \hspace*{3mm} s=1, \ldots, p;
\hspace*{1mm}k=0, 1, \ldots,
\end{split}
\end{equation}
where $x_k^s\text{\scriptsize$(t)$}\in
\mathbb{R}^{n_t}$;
 $t, s=1, \ldots, p; \,k=0, 1, \ldots$.

In order to introduce the BMD method,
we assume there are $p$ strictly convex
 and continuously differentiable functions
$\varphi_t:
\,\mathbb{R}^{n_t}\rightarrow \mathbb{R}$, $t=1,2, \ldots, p$.
Furthermore, we make the following assumption.
\begin{assumption}\label{Assumption phis}
$\varphi_t$ is a strongly convex and
twice continuously differentiable function
with parameter $\mu_t>0$, i.e.,  for any
$y\text{\scriptsize$(t)$}$ and $
 x\text{\scriptsize$(t)$}\in\mathbb{R}^{n_t}$,
\begin{equation}\label{Strongly Convex}
\varphi_t\left(y\text{\scriptsize$(t)$}\right)
\geq \varphi\left(x\text{\scriptsize$(t)$}\right)
+\left\langle
\nabla \varphi_t\left(x\text{\scriptsize$(t)$}\right),
 y\text{\scriptsize$(t)$}
 -
 x\text{\scriptsize$(t)$}
  \right\rangle
+\frac{\mu_t}{2}\left\|y\text{\scriptsize$(t)$}
 -
 x\text{\scriptsize$(t)$}\right\|^2,
 \hspace*{3mm} t=1,2, \ldots, p.
\end{equation}
\end{assumption}
The Bregman divergences of the above
strongly convex functions, $B_{\varphi_t}: \mathbb{R}^{n_t}
\times \mathbb{R}^{n_t}\to \mathbb R^{+}$,
are defined as
\begin{equation}\label{Def Bvarphis}
\begin{split}
B_{\varphi_t}\left(x\text{\scriptsize$(t)$},
y\text{\scriptsize$(t)$}\right)
\triangleq \varphi_t
\left(x\text{\scriptsize$(t)$}\right)-
\varphi_t\left(
y\text{\scriptsize$(t)$}
\right)-
\left\langle x\text{\scriptsize$(t)$}
-y\text{\scriptsize$(t)$},
 \nabla \varphi_t\left(
 y\text{\scriptsize$(t)$}
 \right)\right\rangle,
\hspace*{3mm} t=1,2, \ldots, p.
\end{split}
\end{equation}
The Bregman divergence was initially
studied by
\cite{Bregman1967The} and later by many others
(see
\cite{Auslender2006Interior,
Bauschke2006Bregman,
Teboulle1997Convergence}
and references therein).
\begin{remark}
We should mention that Bregman divergence is
generally defined for a continuously differentiable and
strongly convex function which is not necessarily twice
differentiable. Here, the twice differentiability of
$\varphi_s$  seems a necessary and reasonable condition  for
our subsequent analysis because the existence of the
Jacobian of  the iterative mapping
depends directly  on $\nabla^2\varphi_t$
(see \eqref{D MBDA psis}).
\end{remark}
Let
\begin{equation}\label{Def mu}
\mu = \min\left\{\mu_1, \mu_2,\ldots, \mu_p
\right\}.
\end{equation}

Given the above notations,
the detailed description of the block mirror descent
algorithm for problem \eqref{Model} is given below.
\\
\medskip

\noindent \fbox{\begin{minipage}{14.3cm}
\begin{method}[BMD]\label{BMD 4.1}
\end{method}
{\bf  Input:} $\alpha<\frac{\mu}{L}$.\\
{\bf  Initialization:}
$x_0\in \mathbb{R}^{n}.$\\
{\bf General Step} $\left(k=0, 1, \ldots\right)$:
Set $x_k^0=x_k$ and define recursively
for $s=1$, $2$, $\ldots$, $p:$\\
$t=1, 2, \ldots, p,$\\
{\bf If} $t=s$,
\begin{equation}\label{xksBMD}
x_k^s\text{\scriptsize$(t)$}=
\arg\min\limits_{x\text{\scriptsize$(t)$}}
\left\langle x\text{\scriptsize$(t)$},
\nabla_t f\left(x_k^{s-1}\right)
\right\rangle
+
\frac{1}{\alpha}B_{\varphi_t}\left(x\text{\scriptsize$(t)$},
x_k^{s-1}\text{\scriptsize$(t)$}\right).
\end{equation}
\bf{Else}
\begin{equation}\label{xktneqsBMD}
x_{k}^s\text{\scriptsize$(t)$}
=
x_{k}^{s-1}\text{\scriptsize$(t)$}.
\end{equation}
{\bf End}\\
Set $x_{k+1}=x_k^p$.
\end{minipage}}
\\
\\

Note that  $\varphi_s$
is a strongly
convex function. Then it is easily seen from
\eqref{Def Bvarphis} that
$ B_{\varphi_s}\left(x\text{\scriptsize$(s)$},
y\text{\scriptsize$(s)$}\right)$  is a strongly
convex function with respect to
$x\text{\scriptsize$(s)$}$
if $y\text{\scriptsize$(s)$}$
is fixed. Hence, let
$x^s_k\text{\scriptsize$(s)$}$
be the unique solution of problem \eqref{xksBMD}.
The KKT condition implies that
\begin{equation}
{\bf 0}=
\nabla_s f\left(x_k^{s-1}\right)
+\frac{1}{\alpha}
\left(\nabla \varphi_s
\left(x^s_k\text{\scriptsize$(s)$}\right)
-
\nabla \varphi_s
\left(x_k^{s-1}\text{\scriptsize$(s)$}\right)
\right),
\end{equation}
which is equivalent to
 \begin{equation}
 \nabla \varphi_s
\left(x_k^s\text{\scriptsize$(s)$}\right)
=
\nabla \varphi_s
\left(x_k^{s-1}\text{\scriptsize$(s)$}\right)
-\alpha
\nabla_s f\left(x_k^{s-1}\right).
\end{equation}

Combined with the assumption about $\varphi_s$,
 Lemma \ref{Phi Diff} in Appendix leads
immediately to the existence of the inverse
mapping of $\nabla \varphi_s$.
Let $\left[\nabla \varphi_s\right]^{-1}$ denote its inverse.
 Then
$x_k^s\text{\scriptsize$(s)$}$
can be expressed in terms of
 $x_k^{s-1}$ as
\begin{equation}\label{xks update BMD}
x_k^s\text{\scriptsize$(s)$}
=\left[\nabla \varphi_s\right]^{-1}
\left(
\nabla \varphi_s
\left(x_k^{s-1}\text{\scriptsize$(s)$}\right)
-\alpha
\nabla_s f\left(x_k^{s-1}\right)
\right).
\end{equation}
It follows from Eqs. \eqref{genexksBMDA},
\eqref{xksBMD}, and \eqref{xktneqsBMD} that
\begin{equation}\label{MBDA psi}
x_k^s=\left(I_n-U_sU_s^T\right)x_k^{s-1}
+U_s\left[\nabla\varphi_s\right]^{-1}
\left(\nabla\varphi_s
\left(x_k^{s-1}\text{\scriptsize$(s)$}\right)
-\alpha\nabla_sf\left(x_k^{s-1}\right)\right),
\hspace*{2mm}s=1, 2, \ldots, p,
\end{equation}
where $U_s$ is defined by \eqref{DefUi}.

In what follows, we define $\psi_s: \mathbb{R}^{n_s}
\to \mathbb{R}^{n_s}$ as
\begin{equation}\label{MBDA psi1}
\psi_s\left(x\right)\triangleq
\left(I_n-U_sU_s^T\right)x
+U_s\left[\nabla\varphi_s\right]^{-1}
\left(\nabla\varphi_s\left(x\text{\scriptsize$(s)$}\right)
-\alpha\nabla_sf\left(x\right)\right),
\hspace*{2mm}s=1, 2, \ldots, p.
\end{equation}
It is clear that, given $x_k$, the above BMD method generates
 $x_{k+1}$ in the following manner,
  \begin{equation}
x_{k+1}=\psi\left(x_{k}\right),
\end{equation}
where the composite mapping
 \begin{equation}\label{Define psi}
\psi(x)\triangleq \psi_p\circ \psi_{p-1}\circ \cdots\circ \psi_{2}\circ \psi_{1}\left(x\right).
\end{equation}

Additionally, it follows from \eqref{MBDA psi1}
that,  
for each $s=1, 2, \ldots, p,$
the Jacobian of $\psi_s$ is given below:
\begin{equation}\label{D MBDA psis}
\begin{split}
&D\psi_s\left(x\right)
\\
&=\left(I_n-U_sU_s^T\right)
+D\left\{\left[\nabla\varphi_s\right]^{-1}
\left(\nabla\varphi_s\left(x\text{\scriptsize$(s)$}\right)
-\alpha\nabla_sf\left(x\right)\right)\right\}U_s^T
\\
&=\left(I_n-U_sU_s^T\right)
\\
&\hspace*{4mm}+D\left\{\nabla\varphi_s\left(x\text{\scriptsize$(s)$}\right)
-\alpha\nabla_sf\left(x\right)\right\}
\times\left\{
\nabla^2\varphi_s
\left\{\left[\nabla \varphi_s\right]^{-1}
\left(\nabla\varphi_s\left(x\text{\scriptsize$(s)$}\right)
-\alpha\nabla_sf\left(x\right)\right)\right\}
\right\}^{-1}U_s^T
\\
&=\left(I_n-U_sU_s^T\right)
+\\
&\hspace*{4mm}\left\{U_s\nabla^2\varphi_s\left(x\text{\scriptsize$(s)$}\right)
-\alpha\nabla^2f\left(x\right)U_s\right\}
\left\{
\nabla^2\varphi_s
\left\{\left[\nabla \varphi_s\right]^{-1}
\left(\nabla\varphi_s\left(x\text{\scriptsize$(s)$}\right)
-\alpha\nabla_sf\left(x\right)\right)\right\}
\right\}^{-1}U_s^T,\end{split}
\end{equation}
where the first equality is due to chain rule; the second
equality holds because of chain rule and  inverse
function theorem in \cite{Spivak1965Calculus};
the last equality follows from the definition
of $U_s$ and
$
 \nabla^2 f(x)=
 \left(
 \frac{\partial^2 f(x)}{\partial x\text{\scriptsize$(s)$}
\partial x\text{\scriptsize$(t)$}}
\right)_{1\leq s,\, t\leq p}.
$

By using the chain rule, we obtain the following Jacobian
 of the mapping $\psi$, i.e.,
  \begin{equation}\label{Chain Rule psi}
D\psi\left(x\right)
=
D\psi_1\left(y_1\right)
\times
Dg_2\left(y_2\right)
\times \cdots \cdot \times
Dg_{p-1}\left(y_{p-1}\right)
\times
Dg_{p}\left(y_p\right),
\end{equation}
 where $y_1=x$, and $y_s=\psi_{s-1}(y_{s-1})$,
 $s=2, \ldots, p$.
\subsection{\bf The iterative mapping
$\psi$ of BMD is a diffeomorphism}
In this subsection, we first present the following
Lemma \ref{LpsiDiff} which shows that
$\psi_s, s=1, \ldots, p,$
are diffeomorphisms. Based on this lemma, then we
further prove that $\psi$  is
a diffeomorphism as well.

 \begin{lemma}\label{LpsiDiff}
If the step size $\alpha<\frac{\mu}{L}$,  then the
mappings $\psi_s$ defined
by \eqref{MBDA psi1}, $s=1, \ldots, p,$
 are diffeomorphisms.
\end{lemma}
\begin{proof}
The proof is lengthy and has been relegated to the Appendix.
\hfill\end{proof}
\begin{proposition}\label{psiDiff}
The mapping
  $\psi$ defined by \eqref{Define psi}
with step size
$\alpha<\frac{\mu}{L}$ is a diffeomorphism.
\end{proposition}
\begin{proof}
It follows from Proposition 2.15 in \cite{Lee2013} that the
composition of two diffeomorphisms is also a diffeomorphism.
Using this fact and previous Lemma \ref{LpsiDiff},
the proof is completed.
\hfill\end{proof}
\subsection{{\bf Eigenvalue analysis of the Jacobian of
$\psi$ at a strict saddle point}}
\label{SubSecJacobia}

In this subsection, we will analyze the eigenvalues of the Jacobian of $\psi$ at a
strict saddle point and show it has at least one eigenvalue with magnitude greater
than one, which is a crucial part.

Suppose that $x^*\in\mathbb{R}^n$ is a strict saddle point.
Then $\nabla f(x^*)={\bf 0}$ and
$\psi_s(x^*)=x^*$, $s=1, 2, \ldots, p$.
Hence,
\begin{equation}\label{D MBDA psis 1}
\begin{split}
&D\psi_s\left(x^*\right)\\
&=\left(I_n-U_sU_s^T\right)
+
\\
&\left\{U_s\nabla^2\varphi_s\left(x^*\text{\scriptsize$(s)$}\right)
-\alpha\nabla^2f\left(x^*\right)U_s\right\}
\left\{
\nabla^2\varphi_s
\left\{\left[\nabla \varphi_s\right]^{-1}
\left(\nabla\varphi_s\left(x^*\text{\scriptsize$(s)$}\right)
-\alpha\nabla_sf\left(x^*\right)\right)\right\}
\right\}^{-1}U_s^T\\
&=\left(I_n-U_sU_s^T\right)
\\
&\hspace*{4mm}+\left\{U_s\nabla^2\varphi_s\left(x^*\text{\scriptsize$(s)$}\right)
-\alpha\nabla^2f\left(x^*\right)U_s\right\}
\left\{
\nabla^2\varphi_s
\left\{\left[\nabla \varphi_s\right]^{-1}
\left(\nabla\varphi_s\left(x^*\text{\scriptsize$(s)$}\right)
\right)\right\}
\right\}^{-1}U_s^T\\
&=\left(I_n-U_sU_s^T\right)
+
\left\{U_s\nabla^2\varphi_s\left(x^*\text{\scriptsize$(s)$}\right)
-\alpha\nabla^2f\left(x^*\right)U_s\right\}
\left\{
\nabla^2\varphi_s
\left(x^*\text{\scriptsize$(s)$}\right)
\right\}^{-1}U_s^T\\
&=I_n
-\alpha\nabla^2f\left(x^*\right)U_s
\left\{
\nabla^2\varphi_s
\left(x^*\text{\scriptsize$(s)$}\right)
\right\}^{-1}U_s^T,
\end{split}
\end{equation}
where the second equality is due to
$\nabla f(x^*)={\bf 0}$; the third equality thanks to
$\left[\nabla \varphi_s\right]^{-1}
\left(\nabla\varphi_s\left(x^*\text{\scriptsize$(s)$}\right)
\right)=x^*\text{\scriptsize$(s)$}$.

Furthermore, since the eigenvalues of $D\psi\left(x^*\right)$
are the same as those of
$\left\{D\psi\left(x^*\right)\right\}^T$,
it suffices to analyze the eigenvalues of
$\left\{D\psi\left(x^*\right)\right\}^T$.
It follows from
\eqref{Chain Rule psi} that
\begin{equation}\label{BMDpsix*T}
\begin{split}
&\left\{D\psi\left(x^*\right)\right\}^T
\\
&=\left\{D\psi_1(x^*)
\times
D\psi_{2}\left(x^*\right)
\times \cdots  \times
D\psi_{p-1}(x^*)\times
D\psi_{p}
\left(x^*\right)
\right\}^T\\
&=\left\{D\psi_p(x^*)\right\}^T
\times
\left\{D\psi_{p-1}\left(x^*\right)\right\}^T
\times \cdots  \times
\left\{D\psi_{2}(x^*)\right\}^T\times
\left\{D\psi_{1}
\left(x^*\right)
\right\}^T\\
&=\prod\limits_{s=p}^1
\left\{I_n
-\alpha\nabla^2f\left(x^*\right)U_s
\left\{
\nabla^2\varphi_s
\left(x^*\text{\scriptsize$(s)$}\right)
\right\}^{-1}U_s^T\right\}^T\\
&=
\prod\limits_{s=p}^1
\left\{I_n
-\alpha U_s
\left\{
\nabla^2\varphi_s
\left(x^*\text{\scriptsize$(s)$}\right)
\right\}^{-1}U_s^T\nabla^2f\left(x^*\right)
\right\}\\
&=
\prod\limits_{s=p}^1
\left\{I_n
-\alpha U_s
\left\{
\nabla^2\varphi_s
\left(x^*\text{\scriptsize$(s)$}\right)
\right\}^{-1}A_s
\right\},
\end{split}
\end{equation}
where the third equality uses \eqref{D MBDA psis 1}
and the last equality is due to definition
\eqref{Def A_i} of $A_s$.

Similar to the case of BCGD, we furthermore define
 \begin{equation}\label{Def WTG}
\widetilde{G}\triangleq\frac{1}{\alpha}
\left[
I_n-
\left\{D\psi(x^*)\right\}^T
\right],
\end{equation}
or equivalently,
 \begin{equation}\label{DGT WTG}
 \left(D\psi(x^*)\right)^T=
 I_n-\alpha \widetilde{G}.
\end{equation}
The above relation \eqref{DGT WTG} clearly means that
\begin{equation}\label{Eigrela WTG}
\lambda
 \in\mbox{eig}
 \left(\widetilde{G}\right)\Leftrightarrow
 1-\alpha \lambda \in
  \mbox{eig}\left(\left(D\psi(x^*)\right)^T\right).
\end{equation}
In order to achieve a simple form of  $\widetilde{G}$,
we first introduce the following notations.
Define
\begin{equation}\label{Def Psi}
\Psi\triangleq\mbox{Diag}
\left(
\nabla^2\varphi_1
\left(x^*\text{\scriptsize$(1)$}\right)
,
\nabla^2\varphi_2
\left(x^*\text{\scriptsize$(2)$}\right),
 \ldots,
\nabla^2\varphi_p
\left(x^*\text{\scriptsize$(p)$}\right)
\right).
\end{equation}
Then
\begin{equation}\label{Def Psiinv}
\Psi^{-1}=\mbox{Diag}
\left(\left\{
\nabla^2\varphi_1
\left(x^*\text{\scriptsize$(1)$}\right)
\right\}^{-1},
\left\{
\nabla^2\varphi_2
\left(x^*\text{\scriptsize$(2)$}\right)
\right\}^{-1},
 \ldots,
 \left\{
\nabla^2\varphi_p
\left(x^*\text{\scriptsize$(p)$}\right)
\right\}^{-1}
\right).
\end{equation}
Note that  $A$, $A_{st}$ and
$A_s$ have been defined by \eqref{Def A}, \eqref{Def A_ij}
and \eqref{Def A_i}, respectively. Then
\begin{equation}
A=\left(A_{st}\right)_{1\leq s, t \leq p}=
\left(\begin{array}{c}
A_1\\
A_2\\
\vdots\\
A_p
\end{array}\right),
\end{equation}
where
$A_{st}\in\mathbb{R}^{n_s\times n_t}$
and $A_s\in\mathbb{R}^{n_s\times n}$
 denote  the $(s, t)$-th block and
 the $s$-th block-row of $A$, respectively.
We further define
\begin{equation}\label{Def T}
\begin{split}
T
&\triangleq
\Psi^{-1} A=
\left(\left\{
\nabla^2\varphi_1
\left(x^*\text{\scriptsize$(s)$}\right)
\right\}^{-1} A_{st}\right)_{1\leq s, t\leq p}
=\left(\begin{array}{c}
\left\{
\nabla^2\varphi_1
\left(x^*\text{\scriptsize$(1)$}\right)
\right\}^{-1} A_1\\
\left\{
\nabla^2\varphi_2
\left(x^*\text{\scriptsize$(2)$}\right)
\right\}^{-1} A_2\\
\vdots\\
\left\{
\nabla^2\varphi_p
\left(x^*\text{\scriptsize$(p)$}\right)
\right\}^{-1} A_p
\end{array}\right)=
\left(\begin{array}{c}
T_1\\
T_2\\
\vdots\\
T_p
\end{array}\right)
\end{split},
\end{equation}
where
$\left\{
\nabla^2\varphi_s
\left(x^*\text{\scriptsize$(s)$}\right)
\right\}^{-1}A_{st}\in\mathbb{R}^{n_s\times n_t}$
and $\left\{
\nabla^2\varphi_s
\left(x^*\text{\scriptsize$(s)$}\right)
\right\}^{-1}A_s\in\mathbb{R}^{n_s\times n}$
 denote  the $(s, t)$-th block and
the $s$-th block-row of $T$, respectively.
Given the above notations,
we further denote the strictly block lower triangular
matrix based on $T$ as
 \begin{equation}\label{Def check T}
 \check{T}
 \triangleq
 \left(\check{T}_{st}\right)_{1\leq s,\,t\leq p}
 \end{equation}
 with $p\times p$ blocks and its $(s, t)$-th
 block is given by
 \begin{equation}\label{CheckTst}
\begin{split}
\check{T}_{st}
&=
\begin{cases}
\begin{array}{lr}
    T_{st},  & s>t,\\
    \mathbf{0},& s\leq t
  \end{array}
  \end{cases}
  \\
  &=\begin{cases}
\begin{array}{lr}
    \left\{
\nabla^2\varphi_s
\left(x^*\text{\scriptsize$(s)$}\right)
\right\}^{-1}
A_{st},  & s>t,\\
    \mathbf{0},& s\leq t.
  \end{array}
  \end{cases}
  \end{split}
\end{equation}
Substituting $T_s$ (see its definition \eqref{Def T})
into \eqref{BMDpsix*T}, we deduce that
\begin{equation}\label{Dpsix*T}
\begin{split}
\left\{D\psi\left(x^*\right)\right\}^T
=\prod\limits_{s=p}^1
\left\{I_n
-\alpha U_s
T_s
\right\}.
\end{split}
\end{equation}

The following Lemma shows that $\widetilde{G}$
still has a form similar to $G$ defined by \eqref{Def G}
. In fact, it just replaces
$A$ and $\check{A}$ in \eqref{Def G}
by $T$ and $\check{T}$, respectively.
\begin{lemma}\label{Property WG}
Let $x^*\in\mathbb{R}^n$ be a strict saddle
point. Assume that $\widetilde{G}$, $T$ and $\check{T}$ are
defined by \eqref{Def WTG}, \eqref{Def T} and
\eqref{Def check T}, respectively. Then
\begin{equation}\label{WideG Form}
\widetilde{G}=\left(I_n+ \alpha
\check{T}\right)^{-1}T.
\end{equation}
\end{lemma}
\begin{proof}
With the identifications \eqref{Def WTG} $\sim$
\eqref{Def G},  \eqref{Dpsix*T} $\sim$
\eqref{Dgx*T1}, \eqref{Def T} $\sim$ \eqref{Def A} and
\eqref{Def check T} $\sim$ \eqref{Def bar A}, we
are able to apply the same arguments as in the
proof of Lemma \ref{Property G1} to
obtain \eqref{WideG Form}. Since the proof follows
a similar pattern, it is therefore omitted.
\hfill\end{proof}

Based on the above Lemma \ref{Property WG},
along with definitions \eqref{Def Psiinv}, \eqref{Def T}, \eqref{Def check T},
\eqref{Def A} and \eqref{Def bar A}, we further
have
\begin{equation}\label{WideG Trans}
\begin{split}
\widetilde{G}&=\left(I_n+ \alpha
\Psi^{-1}\check{A}\right)^{-1}\Psi^{-1} A\\
&=\left(\Psi+ \alpha
\check{A}\right)^{-1} A\\
&=\left[\Psi^{\frac{1}{2}}
\left(I_n+ \alpha
\Psi^{-\frac{1}{2}}
\check{A}
\Psi^{-\frac{1}{2}}
\right)
\Psi^{\frac{1}{2}}
\right]^{-1} A\\
&=\Psi^{-\frac{1}{2}}
\left(I_n+ \alpha
\Psi^{-\frac{1}{2}}
\check{A}
\Psi^{-\frac{1}{2}}
\right)
^{-1}\Psi^{-\frac{1}{2}} A\\
&=\left(\Psi^{-\frac{1}{2}}\right)^T
\left[I_n+ \alpha
\Psi^{-\frac{1}{2}}
\check{A}
\left(\Psi^{-\frac{1}{2}}\right)^T
\right]
^{-1}\Psi^{-\frac{1}{2}} A,
\end{split}
\end{equation}
where the second equality holds because $\Psi^{-\frac{1}{2}}$
denotes the unique symmetric, positive definite square
root matrix of the symmetric, positive definite matrix
$\Psi^{-1}$.

Switching the order of the products \eqref{WideG Trans}
by moving the first component to the last, we get a new
matrix
\begin{equation}\label{Def overline G}
\overline{G}\triangleq
\left[
I_n+\alpha\Psi^{-\frac{1}{2}}
\check{A}
\left(\Psi^{-\frac{1}{2}}\right)^T
\right]^{-1}
\Psi^{-\frac{1}{2}}
 A
 \left(\Psi^{-\frac{1}{2}}\right)^T.
\end{equation}
Note that $\mbox{eig}\left(XY\right)
=\mbox{eig}\left(YX\right)$
for any two square matrices, thus
\begin{equation}\label{Eig WideG}
\mbox{eig}\left(\widetilde{G}\right)=
\mbox{eig}\left(\overline{G}\right),
\end{equation}
which shows that it suffices to analyze  eigenvalues of
$\overline{G}$.

Before presenting Lemma \ref{Eig wideG in Omega}, the grand result of this subsection,
 we first provide
two special properties of $\Psi^{-\frac{1}{2}}A
\left(\Psi^{-\frac{1}{2}}\right)^T$ which will be
used in the subsequent analysis.
\begin{lemma}\label{PsicheckA}
Let $x^*\in\mathbb{R}^n$ be a strict
saddle point. Assume that $A$
and $\Psi$ are define by \eqref{Def A}
and \eqref{Def Psi}, respectively. Then,
\begin{enumerate}[(i)]
 \item $\Psi^{-\frac{1}{2}}A
 \left(\Psi^{-\frac{1}{2}}\right)^T$
 has at least one negative eigenvalue.
 \item The spectral radius of the symmetric matrix
 $\Psi^{-\frac{1}{2}}
 A\left(\Psi^{-\frac{1}{2}}\right)^T
 $ is upper bounded by $\frac{\mu}{L}$.
 \end{enumerate}
\end{lemma}
\begin{proof}
(i) Since $\Psi^{-\frac{1}{2}}A
 \left(\Psi^{-\frac{1}{2}}\right)^T$
 is a  congruent transformation of $A$,
 they have same index of inertia.
In addition, $x^*$ is a strict saddle point implies that
$A=\nabla^2 f(x^*)$ (see its definition \eqref{Def A})
 has at least one negative eigenvalue.
Hence, $\Psi^{-\frac{1}{2}}A
 \left(\Psi^{-\frac{1}{2}}\right)^T$
 has at least one negative eigenvalue as well.
\\
(ii) In what follows, we will prove that
 \begin{equation}
 \rho\left(\Psi^{-\frac{1}{2}}
 A
 \left(\Psi^{-\frac{1}{2}}\right)^T\right)
 \leq \frac{L}{\mu}.
 \end{equation}
Obviously, it suffices to prove that
$-\frac{L}{\mu} I_n\preceq
\Psi^{-\frac{1}{2}}
 A
 \left(\Psi^{-\frac{1}{2}}\right)^T
 \preceq \frac{L}{\mu}I_n $
holds. It  follows easily from \eqref{Lipschitz}
and  Lemma 7 in \cite{Panageas2016Gradient} that
\begin{equation}\label{-LAL}
-L I_n\preceq \nabla^2 f(x^*)=A \preceq L I_n.
\end{equation}
In addition, Eqs. \eqref{Strongly Convex} means that
\begin{equation}
\nabla^2\varphi_s
\left(x^*\text{\scriptsize$(s)$}\right)
\succeq \mu_s I_{n_s}\succeq \mu I_{n_s}, \hspace*{3mm} s=1 ,2, \ldots, p,
\end{equation}
where the last inequality is due to the definition
of $\mu$ (see it definition in \eqref{Def mu}).
Consequently, combined with the definition
\eqref{Def Psi}, the above inequalities imply that
\begin{equation}
\frac{1}{\mu}\Psi\succeq I_n,
\end{equation}
which, combined with \eqref{-LAL} further implies
\begin{equation}\label{-LmuALmu}
-\frac{L}{\mu} \Psi I_n\preceq
 A
 \preceq \frac{L}{\mu} \Psi I_n.
\end{equation}
Multiplying the left-hand side of the
above inequalities  by $\Psi^{-\frac{1}{2}}$
 and the right-hand side by
 $\left(\Psi^{-\frac{1}{2}}\right)^T$, respectively,
we arrive at
\begin{equation}
-\frac{L}{\mu} I_n\preceq
\Psi^{-\frac{1}{2}}
 A
 \left(\Psi^{-\frac{1}{2}}\right)^T
 \preceq \frac{L}{\mu}I_n.
\end{equation}
Thus, the proof is finished.
\hfill\end{proof}

Based on the above Lemma \ref{PsicheckA} and
Lemma \ref{Key Lemma BCGD} in Section
\ref{Several Lemmas}, we will prove that there exists
at least one eigenvalue of $\overline{G}$
defined by \eqref{Def overline G} which belongs
to $\Omega$ defined by \eqref{Omega}.
\begin{lemma}\label{Eig wideG in Omega}
Suppose  that $x^*\in\mathbb{R}^n$ is a strict
saddle point, and
\begin{equation}
\overline{G}=
\left[
I_n+\alpha\Psi^{-\frac{1}{2}}
\check{A}
\left(\Psi^{-\frac{1}{2}}\right)^T
\right]^{-1}
\Psi^{-\frac{1}{2}}
 A
 \left(\Psi^{-\frac{1}{2}}\right)^T
\end{equation}
is given  by \eqref{Def overline G} with $\alpha\in
\left(0, \frac{\mu}{L}\right)$, where $L$ is determined by
\eqref{Lipschitz}; $\mu$, $A$ and $\check{A}$ are given
by \eqref{Def mu}, \eqref{Def A}  and \eqref{Def bar A},
respectively; and $\Psi^{-\frac{1}{2}}$ denotes the unique
 symmetric, positive definite square root matrix of the symmetric,
 positive definite matrix $\Psi^{-1}$ defined by
 \eqref{Def Psiinv}. Then, there exists at  least one
 eigenvalue of the $\overline{G}$ which belongs
 to $\Omega$ defined by \eqref{Omega0}.
\end{lemma}
\begin{proof}
Since $x^*$ is a strict saddle point and
$A=\nabla^2 f\left(x^*\right)$ is given by
\eqref{Def A}, Lemma \ref{PsicheckA} clearly
implies that $\Psi^{-\frac{1}{2}}A
 \left(\Psi^{-\frac{1}{2}}\right)^T$
 has at least one negative eigenvalue
 and the spectral radius of the symmetric matrix
 $\Psi^{-\frac{1}{2}}
 A\left(\Psi^{-\frac{1}{2}}\right)^T
 $ is upper bounded by $\frac{\mu}{L}$.
\\
Further,  note that $\overline{G}=
\left[
I_n+\alpha\Psi^{-\frac{1}{2}}
\check{A}
\left(\Psi^{-\frac{1}{2}}\right)^T
\right]^{-1}
\Psi^{-\frac{1}{2}}
 A
 \left(\Psi^{-\frac{1}{2}}\right)^T
$
and $\Psi^{-\frac{1}{2}}\check{A}
\left(\Psi^{-\frac{1}{2}}\right)^T$
is a strictly block lower triangle matrix based on
$\Psi^{-\frac{1}{2}}A
\left(\Psi^{-\frac{1}{2}}\right)^T$,
where $\check{A}$ is defined by \eqref{Def bar A} .
Therefore, by applying Lemma
\ref{Key Lemma BCGD}
in Section \ref{Several Lemmas} with
identifications $\Psi^{-\frac{1}{2}}
 A
 \left(\Psi^{-\frac{1}{2}}\right)^T\sim B$,
 $\Psi^{-\frac{1}{2}}
 \check{A}
 \left(\Psi^{-\frac{1}{2}}\right)^T\sim\check{B}$,
$\alpha\sim \beta$ and
$\rho\left(\Psi^{-\frac{1}{2}}A
 \left(\Psi^{-\frac{1}{2}}\right)^T\right)$
 $\sim$ $\rho(B)$, we know, for any
 $\alpha$ $\in \left(0, \frac{\mu}{L}\right)$, there exists at
least one eigenvalue of $\bar{G}$ belonging
to $\Omega$ defined by \eqref{Omega0}.

In addition, Lemma \ref{PsicheckA} implies that
gradient Lipschitz  continuous constant
$\frac{L}{\mu}\geq$ $\rho\left(\Psi^{-\frac{1}{2}}A
 \left(\Psi^{-\frac{1}{2}}\right)^T\right)$, which
  leads immediately  to
$\alpha \in \left(0, \frac{\mu}{L}\right)
\subseteq\left(0, \frac{1}{\rho\left(\Psi^{-\frac{1}{2}}A
 \left(\Psi^{-\frac{1}{2}}\right)^T\right)}\right)$.
Then, the above arguments clearly imply
that the proposition holds true.
\hfill\end{proof}

The following proposition shows that the
Jacobian of the BMD iterative mapping $\psi$ defined
by \eqref{Define psi} at a strict saddle point
admits at least one eigenvalue whose magnitude is strictly
greater than one.

\begin{proposition}\label{Jaco psi neg eig}
Assume that BMD iterative mapping $\psi$ is defined
by \eqref{Define psi} with
$\alpha\in \left(0,\,\frac{\mu}{L}\right)$, and
$x^*\in\mathbb{R}^n$ is a strict saddle point. Then
$D\psi\left(x^*\right)$ has at least one eigenvalue whose magnitude is
strictly greater than one.
\end{proposition}
\begin{proof}
Note that \eqref{Eig WideG} and \eqref{Def overline G}
imply that $\widetilde{G}$ and $\overline{G}$ have
same eigenvalues. Consequently, Lemma \ref{Eig wideG in Omega} means
there exists at least one eigenvalue of the
$\widetilde{G}$ lying in $\Omega$ defined by \eqref{Omega}.

In what follows, with the identifications
$D\psi\left(x^*\right)$ $\sim$
$Dg_\text{\tiny$\alpha f$}(x^*)$,
$\widetilde{G}$ $\sim$ $G$ and $\frac{L}{\mu}$ $\sim$ $L$,
we are able to apply the same arguments  as in the proof
of Proposition \ref{Jaco g neg eig} in Subsection
\ref{BCGDSubSecJacobia} to finish the proof. Since the proof
follows a similar pattern, it is therefore omitted.
\hfill\end{proof}
\subsection{\bf Main results of BMD method}
We first introduce the following proposition, which asserts that
the limit point of the sequence generated by the
BMD method \ref{BMD 4.1} is a critical point of $f$.
\begin{proposition}\label{limit=critical point BMD}
Under Assumption \ref{Assumption f}, if
  $\left\{x_k\right\}_{k\geq 0}$
   is generated by the BMD method \ref{BMD 4.1}
with $0<\alpha<\frac{\mu}{L}$, $\lim\limits_{k}{x_k}$
exists and denote it as $x^*$, then $x^*$ is a
critical point of $f$, i.e., $\nabla f(x^*)
={\bf 0}$.
\end{proposition}
\begin{proof}
First, since $\left\{x_k\right\}_{k\geq 0}$ is generated
by the BMD method \ref{BMD 4.1},
then $x_k=\psi^k(x_0)
\footnote
{$\psi^k$
denotes the composition of
$\psi$
with itself $k$ times.}
$,
where the BMD iterative mapping $\psi$ is defined by
\eqref{Def Psi}. Hence, $\lim\limits_kx_k=\lim\limits_k
\psi^k(x_0)=x^*$. Since $\psi$ is a  diffeomorphism,
we immediately know that $x^*$ is a fixed point of $\psi$.

Second, notice that $\psi$ and $\psi_s$ are defined by
\eqref{Define psi} and \eqref{MBDA psi1}, respectively.
If $x^*$ is a fixed point of $\psi$, then
\begin{equation}
\psi(x^*)=x^*,
\end{equation}
which implies that
\begin{equation}\label{MBDA psi1 x*}
\psi_s\left(x^*\right)=
\left(I_n-U_sU_s^T\right)x^*
+U_s\left[\nabla\varphi_s\right]^{-1}
\left(\nabla\varphi_s\left(x^*\text{\scriptsize$(s)$}\right)
-\alpha\nabla_sf\left(x^*\right)\right),
\hspace*{2mm}s=1, 2, \ldots, p.
\end{equation}
Consequently,
\begin{equation}\label{xks update BMD x*}
x^*\text{\scriptsize$(s)$}
=\left[\nabla \varphi_s\right]^{-1}
\left(
\nabla \varphi_s
\left(x^*\text{\scriptsize$(s)$}\right)
-\alpha
\nabla_s f\left(x^*\right)
\right),
\hspace*{2mm}s=1, 2, \ldots, p.
\end{equation}
Since Lemma \ref{Phi Diff} in Appendix asserts
that $\nabla\varphi_s$  is a diffeomorphism,
then $\left[\nabla\varphi_s\right]^{-1}$ is a
diffeomorphism as well. Thus it follows from
\eqref{xks update BMD x*} that
\begin{equation}\label{xks update BMD x* Phi}
\nabla \varphi_s
\left(x^*\text{\scriptsize$(s)$}\right)
=
\nabla \varphi_s
\left(x^*\text{\scriptsize$(s)$}\right)
-\alpha
\nabla_s f\left(x^*\right),
\hspace*{2mm}s=1, 2, \ldots, p.
\end{equation}
We arrive at
\begin{equation}\label{nablax*=0}
\nabla_s f\left(x^*\right)
={\bf 0},
\hspace*{2mm}s=1, 2, \ldots, p,
\end{equation}
or equivalently, $x^*$ is a stationary point of $f$.
Thus the proof is finished. \hfill
\end{proof}

Armed with the results established in  previous
subsections, we now state and prove our main
theorem of the BMD method, whose proof is similar to
that of Theorem \ref{Main Theorem BCGD} in
Subsection \ref{Main Subsec of BCGD}. However,
its proof is still given as follows in detail
for the sake of completeness.
\begin{theorem}\label{Main Theorem BMD}
Let $f$ be a $C^2$ function and $x^*$ be a strict
saddle point. If  $\left\{x_k\right\}$ is generated
by the BMD method \ref{BMD 4.1} with $0<\alpha<\frac{\mu}{L}$,  then $$\mathbb{P}_\nu\left[\lim_{k}{x_k}=x^*\right]=0.$$
\end{theorem}
\begin{proof}
First, Proposition \ref{limit=critical point BMD} implies that, if
$\lim\limits_{k}{x_k}$ exists then
it must be a critical point. Hence, we consider calculating the Lebesgue measure
(or probability with respect to the prior measure $\nu$)
of the set $\left[\lim\limits_{k}{x_k}=x^*\right]
 =W^s\left(x^*\right)$ (see Definition
 \ref{Global Stable Set}).

 Second, since Proposition \ref{psiDiff} means the BMD iterative
 mapping $\psi$ is a diffeomorphism, we replace
 $\phi$ and fixed point by $\psi$ and the strict saddle
 point $x^*$ in the above Stable Manifold
 Theorem \ref{Sabl Mani Theorem} in
 Subsection \ref{Main Subsec of BCGD}, respectively.
 Then  the manifold $W_{loc}^{s}(x^*)$
 has strictly positive codimension because of
Proposition \ref{Jaco psi neg eig} and $x^*$ being a strict
saddle point. Hence, $W_{loc}^{s}(x^*)$ has measure zero.

In what follows,
we are able to apply the same arguments
as in \cite{JLeeandMSimchowitz2016}
to finish the proof of the theorem. Since
the proof follows a similar pattern,
it is therefore omitted.
\hfill\end{proof}

Given the above Theorem \ref{Main Theorem BMD},
we immediately  obtain the following
Theorem \ref{key theorem BMD} and its
corollary by the same arguments as in the proofs
of Theorem 2 and Corollary 12 in
 \cite{Panageas2016Gradient}.
Therefore, we omit their proofs.
\begin{theorem}[Non-isolated]\label{key theorem BMD}
Let $f: \mathbb{R}^n\rightarrow\mathbb{R}$ be a twice
continuously differentiable function
and $\sup\limits_{x\in\mathbb{R}^n}\left\|
\nabla f(x)\right\|_2\leq L<\infty$. The set of initial
conditions $x\in \mathbb{R}^n$ so that the
BMD method \ref{BMD 4.1} with step size
$0<\alpha<\frac{\mu}{L}$ converges to a strict
saddle point is of (Lebesgue) measure zero,
without assumption that critical points are
isolated.
\end{theorem}
A straightforward corollary of Theorem
\ref{key theorem BMD} is given below:
\begin{corollary}\label{key corolary BMD}
Assume that the conditions of Theorem
\ref{key theorem BMD}
 are satisfied and all saddle points of $f$
are strict. Additionally, 
assume
$\lim\limits_k\psi^k(x)$
exists for all $x$
in $\mathbb{R}^n$. Then
 $$
\mathbb{P}_\nu
\left[\lim\limits_k\psi^k(x)
=x^*\right]=1,$$
where $\psi$ is defined by \eqref{Def Psi} and
$x^*$ is a local minimum.
\end{corollary}

\section{The PBCD method}\label{PBCD sec}
In this section, we will prove that the PBCD method in \cite{Hong2017Iteration,Fercoq2013Accelerated,Hong2015A} converges to minimizers as well, almost surely with random initialization.

\subsection{\bf The PBCD  method
description}

For clarity of notation, recall the vector of decision
variables $x$ has been assumed to have the following
partition (see \eqref{x Partion}):
\begin{equation}
x=\left(\begin{array}{c}
x\text{\scriptsize$(1)$}\\
x\text{\scriptsize$(2)$}\\
\vdots\\
x\text{\scriptsize$(p)$}
\end{array}
\right),
\end{equation}
where
$x\text{\scriptsize$(t)$}\in \mathbb{R}^{n_t}$,
$n_1$, $n_2$, $\ldots$,
$n_p$ are $p$
positive integer numbers satisfying
$\sum\limits_{s=1}^{t}n_t=n$.
Correspondingly,  we assume a set of variables $x_k^s$ has the following partition as well:
\begin{equation}\label{genexks}
\begin{split}
x_k^s&\triangleq
\left(\begin{array}{c}
x_{k}^s\text{\scriptsize$(1)$}\\
x_{k}^s\text{\scriptsize$(2)$}\\
\vdots\\
x_{k}^s\text{\scriptsize$(p)$}
\end{array}
\right),  \hspace*{3mm} s=1, \ldots, p;
\hspace*{1mm}k=0, 1, \ldots,
\end{split}
\end{equation}
where $x_k^s\text{\scriptsize$(t)$}\in \mathbb{R}^{n_t}$;
 $t, s=1, \ldots, p; \,k=0, 1, \ldots$.

Given the above notations, the detailed description of the
PBCD method for problem \eqref{Model} is given below.
\medskip

\noindent \fbox{\begin{minipage}{14.3cm}
\begin{method}[PBCD]\label{PBCD 5.1}
\end{method}
{\bf  Input:} $\alpha<\frac{1}{L}$.\\
{\bf  Initialization:}
$x_0\in \mathbb{R}^{n}.$\\
{\bf General Step} $\left(k=0, 1, \ldots\right)$:
Set $x_k^0=x_k$ and define recursively
for $s=1$, $2$, $\ldots$, $p:$\\
$t=1, 2, \ldots, p,$\\
{\bf If} $t=s$,
\begin{equation}\label{xks}
\begin{split}
x_k^s\text{\scriptsize$(t)$}=
\arg\min\limits_{x\text{\scriptsize$(t)$}}
&\Big{\{}f\left(x_k^{s-1}\text{\scriptsize$(1)$},\ldots,
x_k^{s-1}\text{\scriptsize$(t-1)$},
x\text{\scriptsize$(t)$},
x_k^{s-1}\text{\scriptsize$(t+1)$},
x_k^{s-1}\text{\scriptsize$(p)$}\right)
\\
&\hspace*{2mm}+\frac{1}{2\alpha}
\left\|
x\text{\scriptsize$(t)$}
-
x_{k}^{s-1}\text{\scriptsize$(t)$}
\right\|_2^2
\Big{\}}.
\end{split}
\end{equation}
\bf{Else}
\begin{equation}\label{xktneqs}
x_{k}^s\text{\scriptsize$(t)$}
=
x_{k}^{s-1}\text{\scriptsize$(t)$}.
\end{equation}
{\bf End}\\
Set $x_{k+1}=x_k^p$.
\end{minipage}}
\\
\\

It follows from $\alpha<\frac{1}{L}$ that
$$f\left(x_k^{s-1}\text{\scriptsize$(1)$},\ldots,
x_k^{s-1}\text{\scriptsize$(s-1)$},
x\text{\scriptsize$(s)$},
x_k^{s-1}\text{\scriptsize$(s+1)$},
x_k^{s-1}\text{\scriptsize$(p)$}\right)
+\frac{1}{2\alpha}
\left\|
x\text{\scriptsize$(s)$}
-
x_{k}^{s-1}\text{\scriptsize$(s)$}
\right\|_2^2$$ is a strongly convex function with respect
to variable $x\text{\scriptsize$(s)$}$. Hence, let
$x_{k}^{s}\text{\scriptsize$(s)$}$ be
the unique minimizer of problem \eqref{xks}.
Then by the KKT condition,
\begin{equation*}
\begin{split}
{\bf 0}
=\nabla_s f\left(x_k^{s-1}\text{\scriptsize$(1)$},\ldots,
x_k^{s-1}\text{\scriptsize$(s-1)$},
x_k^s\text{\scriptsize$(s)$},
x_k^{s-1}\text{\scriptsize$(s+1)$},
x_k^{s-1}\text{\scriptsize$(p)$}\right)
+\frac{1}{\alpha}\left(
x_k^s\text{\scriptsize$(s)$}
-
x_{k}^{s-1}\text{\scriptsize$(s)$}\right),
\end{split}
\end{equation*}
which is equivalent to
\begin{equation}\label{KKTxks1}
\begin{split}
x_{k}^{s-1}\text{\scriptsize$(s)$}
=x_k^s\text{\scriptsize$(s)$}
+
\alpha \nabla_s f\left(x_k^{s-1}\text{\scriptsize$(1)$},\ldots,
x_k^{s-1}\text{\scriptsize$(s-1)$},
x_k^s\text{\scriptsize$(s)$},
x_k^{s-1}\text{\scriptsize$(s+1)$},
x_k^{s-1}\text{\scriptsize$(p)$}\right).
\end{split}
\end{equation}
Combining Eqs. \eqref{genexks}, \eqref{xktneqs} and
\eqref{KKTxks1}, we have the following
relationships between $x_k^{s-1}$ and $x_k^{s}$,
\begin{equation}\label{xks-1}
x_k^{s-1}=x_k^{s}+\alpha U_s \nabla_s f(x_k^{s}),
 \hspace*{3mm} s=1, \ldots, p.
\end{equation}
Recall that, for a given step size $\alpha>0$,
we have used
$g_{\text{\tiny$\alpha f$}}^s$ (see its
definition in \eqref{Degi}) to denote
the gradient mapping of function $f$ with respect
to variable
$x\text{\scriptsize$(s)$}$,
i.e.,
\begin{equation}\label{RDegi}
g_{\text{\tiny$\alpha f$}}^s(x)
=x-\alpha U_s \nabla_s f(x),
 \hspace*{3mm} s=1, \ldots, p.
\end{equation}
By using same notations, for a given step  size
$\alpha>0$, the gradient mapping of function  $-f$ with
respect to variable $x\text{\scriptsize$(s)$}$  is
\begin{equation}\label{RDegi}
\begin{split}
g_{\text{\tiny$\alpha(-f)$}}^s(x)
=x-\alpha U_s \nabla_s\left(-f(x)\right)
=x-\alpha U_s \left(- \nabla_s f(x)\right)
=x+\alpha U_s \nabla_s f(x).
\end{split}
\end{equation}
It is obvious that Eqs. \eqref{xks-1} and \eqref{RDegi}
imply that
\begin{equation}\label{xks-1Degi}
x_k^{s-1}
=g_{\text{\tiny$\alpha(-f)$}}^s(x_k^{s}),
 \hspace*{3mm} s=1, \ldots, p.
\end{equation}
Note that Lemma \ref{Pg12Diff} in Subsection
\ref{PSubSecDiff}  implies that
$g_{\text{\tiny$\alpha(-f)$}}^s$ is a diffeomorphism
with $\alpha\in\left(0, \frac{1}{L}\right)$.
We denote its inverse as
$\left[g_{\text{\tiny$\alpha(-f)$}}^s\right]^{-1}$ which is
a diffeomorphism as well.
Then, \eqref{xks-1Degi} is further equivalent to
\begin{equation}\label{xks=giinv}
x_k^{s}=
\left[g_{\text{\tiny$\alpha(-f)$}}^s\right]^{-1}(x_k^{s-1}),
 \hspace*{3mm} s=1, \ldots, p.
\end{equation}
\begin{lemma}
Assume that $g_{\text{\tiny$\alpha(-f)$}}^s$ is determined
by \eqref{RDegi}, $s=1, \ldots, p.$
Given $x_k$, the above PBCD method
generates $x_{k+1}$ in the following manner,
  \begin{equation}\label{Def g-f inv}
x_{k+1}=\left[g_{\text{\tiny$\alpha(-f)$}}\right]^{-1}\left(x_{k}\right),
\end{equation}
where the composite iterative mapping $\left[g_{\text{\tiny$\alpha(-f)$}}\right]^{-1}$
denotes the inverse mapping of the following mapping
 \begin{equation}\label{Define g-f}
 \begin{split}
g_{\text{\tiny$\alpha(-f)$}}
&\triangleq
g_{\text{\tiny$\alpha(-f)$}}^1
\circ
g_{\text{\tiny$\alpha(-f)$}}^2
\circ
\cdots
\circ
g_{\text{\tiny$\alpha(-f)$}}^{p-1}
\circ
g_{\text{\tiny$\alpha(-f)$}}^p.
\end{split}
\end{equation}
\end{lemma}
\begin{proof}
According to the PBCD method and \eqref{xks=giinv},
we have $x_k^0=x_k$
and
\begin{equation}\label{PDefine g}
 \begin{split}
x_{k+1}
&=
\left[g_{\text{\tiny$\alpha(-f)$}}^p\right]^{-1}
\circ
\left[g_{\text{\tiny$\alpha(-f)$}}^{p-1}\right]^{-1}
\circ
\cdots
\circ
\left[g_{\text{\tiny$\alpha(-f)$}}^2\right]^{-1}
\circ
\left[g_{\text{\tiny$\alpha(-f)$}}^1\right]^{-1}
\left(x_k\right)\\
&=
\left[g_{\text{\tiny$\alpha(-f)$}}^1
\circ
g_{\text{\tiny$\alpha(-f)$}}^2
\circ
\cdots
\circ
g_{\text{\tiny$\alpha(-f)$}}^{p-1}
\circ
g_{\text{\tiny$\alpha(-f)$}}^p\right]^{-1}
\left(x_k\right).
\end{split}
\end{equation}
Thus the proof is finished.\hfill
\hfill\end{proof}

By simple computation, the Jacobian of
 $g_{\alpha (-f)}^s$ is given by
 \begin{equation}\label{Jacob giP}
 Dg_{\text{\tiny$\alpha(-f)$}}^s(x)=I_n+\alpha \nabla^2 f(x)U_{s}U_s^T,
 \end{equation}
 where \begin{equation}
 \nabla^2 f(x)=
 \left(
 \frac{\partial^2 f(x)}{\partial x\text{\scriptsize$(s)$}
\partial x\text{\scriptsize$(t)$}}
\right)_{1\leq s,\, t\leq p}.
\end{equation}
Since $g_{\text{\tiny$\alpha(-f)$}}(x)$ is defined by
\eqref{Define g-f},  the chain rule implies that
the Jacobian of the mapping $g$ is
 \begin{equation}\label{-fChain Rule}
Dg_{\text{\tiny$\alpha(-f)$}}(x)
=
Dg_{\text{\tiny$\alpha(-f)$}}^p\left(y_p\right)
\times
Dg_{\text{\tiny$\alpha(-f)$}}^{p-1}\left(y_{p-1}\right)
\times
\cdots
\times
Dg_{\text{\tiny$\alpha(-f)$}}^2\left(y_2\right)
\times
Dg_{\text{\tiny$\alpha(-f)$}}^1\left(y_1\right),
\end{equation}
where $y_1=x$, and $y_s=g_{\text{\tiny$\alpha(-f)$}}^{s-1}(y_{s-1})$,
 $s=2, \ldots, p$.

 \subsection{{\bf The iterative mapping
 $\left[g_\text{\tiny$\alpha (-f)$}\right]^{-1}$  of the PBCD method is a
 diffeomorphism}}\label{PSubSecDiff}
 In this subsection, we first present the following
Lemma \ref{Pg12Diff} which shows that
$g_{\text{\tiny$\alpha (-f)$}}^s, s=1, \ldots, p,$
are diffeomorphisms. Based on this lemma, then we
further prove that $\left[g_\text{\tiny$\alpha (-f)$}\right]^{-1}$  is
a diffeomorphism as well.

\begin{lemma}\label{Pg12Diff}
If step size $\alpha<\frac{1}{L}$,  then the
mappings $g_{\text{\tiny$\alpha(-f)$}}^s$ defined
by \eqref{Degi},
$s=1,$ $\ldots, p,$   are diffeomorphisms.
\end{lemma}
\begin{proof}
Note that $L$ is also the Lipschitz constant
of gradient  of $-f$.
By replacing $-f$ with $f$, Lemma \ref{g12Diff} in
Subsection \ref{SubSecDiff} implies that
$g_{\text{\tiny$\alpha(-f)$}}^s$ is a diffeomorphism
with $\alpha\in\left(0, \frac{1}{L}\right)$, the proof is completed.
\hfill\end{proof}
\begin{proposition}\label{PgDiff}
The mapping
$\left[g_{\text{\tiny$\alpha (-f)$}}\right]^{-1}$
determined by \eqref{Define g-f} with step size
$\alpha<\frac{1}{L}$ is a diffeomorphism.
\end{proposition}
\begin{proof}
Note that Lemma \ref{Pg12Diff} implies that
$g_{\text{\tiny$\alpha(-f)$}}^s, s=1, \ldots, p,$  defined
by \eqref{Degi} are diffeomorphisms and
 $g_{\text{\tiny$\alpha (-f)$}}$ is defined
 \eqref{Define g-f} with $0<\alpha<\frac{1}{L}$.
By a similar argument as in the proof of Proposition
\ref{gDiff} in Subsection \ref{SubSecDiff}, we know
$g_{\text{\tiny$\alpha (-f)$}}$ is a diffeomorphism.
Thus its inverse is a diffeomorphism as well,
the proof is completed.
\hfill\end{proof}
\subsection{{\bf Eigenvalue analysis of the Jacobian of
$\left[g_\text{\tiny$\alpha (-f)$}\right]^{-1}$ at a strict saddle point}}
\label{SubSecJacobia P}
In this subsection, we consider the eigenvalues
 of the Jacobian of
 $\left[g_\text{\tiny$\alpha (-f)$}\right]^{-1}$
at a strict saddle point,  which is a crucial part in our
entire proof.\\

If  $x^*\in\mathbb{R}^n$ is a strict saddle point,
then $\nabla f(x^*)={\bf 0}$. Consequently,
\begin{equation}\label{x*Degi}
x^*=
\left[g_{\text{\tiny$\alpha(-f)$}}^s\right]^{-1}(x^*),
 \hspace*{3mm} s=1, \ldots, p,
\end{equation}
which implies that
\begin{equation}\label{xks=gi}
x^*=g_{\text{\tiny$\alpha(-f)$}}(x^*).
\end{equation}
More importantly, the above Eq. \eqref{xks=gi} implies that
the Jacobian of
$\left[g_{\text{\tiny$\alpha(-f)$}}\right]^{-1}$
at $x^*$ can be expressed as
\begin{equation}\label{xks=Degi}
D\left[g_{\text{\tiny$\alpha(-f)$}}\right]^{-1}(x^*)
=\left(Dg_{\text{\tiny$\alpha(-f)$}}(x^*)\right)^{-1},
\end{equation}
 which is due to inverse function theorem in
\cite{Spivak1965Calculus}.
Hence, in order to show that there is at least one eigenvalue
of $D\left[g_\text{\tiny$\alpha(-f)$}\right]^{-1}(x^*)$
whose magnitude is strictly greater than one, we first argue that $Dg_\text{\tiny$\alpha(-f)$}(x^*)$ still has a similar structure as that of $Dg_\text{\tiny$\alpha f$}(x^*)$ in Section \ref{BCGD Sec}.
%
%

Specifically, chain rule \eqref{-fChain Rule}
and Eq. \eqref{xks=gi} imply
\begin{equation}\label{DPgx*}
Dg_\text{\tiny$\alpha (-f)$}(x^*)
=Dg_\text{\tiny$\alpha (-f)$}^p(x^*)
\times
Dg_\text{\tiny$\alpha (-f)$}^{p-1}\left(x^*\right)
\times \cdots  \times
Dg_\text{\tiny$\alpha (-f)$}^{2}(x^*)\times
Dg_\text{\tiny$\alpha (-f)$}^1
\left(x^*\right).
\end{equation}
Moreover,  the eigenvalues of
$Dg_{\text{\tiny$\alpha(-f)$}}(x^*)$
are the same as those of its transpose. Hence, we compute
\begin{equation}\label{Dg-fx*T}
\begin{split}
&\left(Dg_\text{\tiny$\alpha (-f)$}(x^*)\right)^T
\\
&=\left(Dg_\text{\tiny$\alpha (-f)$}^1(x^*)\right)^T
\times
\left(Dg_\text{\tiny$\alpha (-f)$}^{2}(x^*)\right)^T
\times \cdots \times
\left(Dg_\text{\tiny$\alpha (-f)$}^{p-1}(x^*)\right)^T
\times
\left(Dg_\text{\tiny$\alpha (-f)$}^p(x^*)\right)^T
\\
&=
\left(I_n+\alpha U_1U_1^T \nabla^2 f(x^*)\right)
\times
\left(I_n+\alpha U_2U_2^T \nabla^2 f(x^*)\right)
\\
&\hspace*{4mm}\times \cdots\times
\left(I_n+\alpha U_{p-1}U_{p-1}^T \nabla^2 f(x^*)\right)
\times
\left(I_n+\alpha U_pU_p^T \nabla^2 f(x^*)\right),
\end{split}
\end{equation}
where the second equality is due to
\eqref{Jacob giP}.


Now we define
 \begin{equation}\label{Def H}
H\triangleq\frac{1}{\alpha}
\left[
I_n-
\left(Dg_\text{\tiny$\alpha (-f)$}(x^*)\right)^T
\right],
\end{equation}
or equivalently,
 \begin{equation}\label{DGT H1}
 \left(Dg_\text{\tiny$\alpha (-f)$}(x^*)\right)^T=
 I_n-\alpha H.
\end{equation}
The above relation \eqref{DGT H1} clearly means that
\begin{equation}\label{Eigrela P}
\lambda
 \in\mbox{eig}
 \left(H\right)\Leftrightarrow
 1-\alpha \lambda \in
  \mbox{eig}\left(\left(Dg_\text{\tiny$\alpha (-f)$}(x^*)\right)^T\right).
\end{equation}
For the sake of clarification,
we rewrite $A$ defined by \eqref{Def A}
 below:
\begin{equation}\label{Def A r}
A=\left(A_{st}\right)_{ 1\leq s,\, t\leq p},
\end{equation}
and  its $(s, t)$-th block is given by
\begin{equation}\label{Def A_ij r}
A_{st}=\frac{\partial^2 f(x^*)}
{\partial x^*\text{\scriptsize$(s)$}
\partial x^*\text{\scriptsize$(t)$}},
\hspace*{3mm} 1\leq s,\, t\leq p.
\end{equation}
Similarly, we denote the strictly block upper triangular matrix based
on $A$ as
 \begin{equation}\label{Def hat A}
 \hat{A}\triangleq
 \left(\hat{A}_{st}\right)_{1\leq s,\,t\leq p}
 \end{equation}
 with $p\times p$ blocks and its $(s, t)$-th
 block is given by
 \begin{equation}\label{Upper hat Ast}
\hat{A}_{st}
=\left\{
\begin{array}{lr}
    A_{st},  & s<t,\\
    \mathbf{0},& s\geq t.
  \end{array}
  \right.
\end{equation}

Based on the above notations, we are able to
obtain the following Lemma \ref{Property H}
which is similar to Lemma \ref{Property G1}
in Subsection \ref{BCGDSubSecJacobia}. It gives a simple expression
of $H$ in terms of $A$ and $\hat{A}$.
\begin{lemma}\label{Property H}
Let
$x^*\in\mathbb{R}^n$ be a strict saddle point. Assume that
$H$,
$A$ and $\hat{A}$ are defined by
\eqref{Def H}, \eqref{Def A r} and
\eqref{Def hat A}, respectively.
Then
\begin{equation}\label{H Form}
H=-\left(I_n- \alpha \hat{A}\right)^{-1}A.
\end{equation}
\end{lemma}
\begin{proof}
The proof is similar to that of
Lemma \ref{Property G1} in Subsection \ref{BCGDSubSecJacobia}.
\hfill\end{proof}

In what follows, we proceed now in a
way similar to the case of BCGD, although some
specific technical difficulties will occur in
the analysis.

The following proposition shows that the
Jacobian of the PBCD iterative mapping $\left[g_\text{\tiny$\alpha (-f)$}\right]^{-1}$ determined
by \eqref{Define g-f} at a strict saddle point
admits at least one eigenvalue whose magnitude is strictly
greater than one, which plays a  key role
in this subection.
\begin{proposition}\label{Jaco g-f eig}
Assume that PBCD iterative mapping
$\left[g_\text{\tiny$\alpha (-f)$}\right]^{-1}$ is determined
by \eqref{Define g-f} and
$x^*\in\mathbb{R}^n$ is a strict saddle point. Then
$D\left[g_{\text{\tiny$\alpha(-f)$}}\right]^{-1}(x^*)$ has at least one eigenvalue whose magnitude is
strictly greater than one.
\end{proposition}
\begin{proof}
First, recall the Jacobian of
$\left[g_{\text{\tiny$\alpha(-f)$}}\right]^{-1}$
at $x^*$ can be expressed as
\begin{equation}\label{R xks=Degi}
D\left[g_{\text{\tiny$\alpha(-f)$}}\right]^{-1}(x^*)
=\left(Dg_{\text{\tiny$\alpha(-f)$}}(x^*)\right)^{-1}.
\end{equation}
Second, recall Eqs. \eqref{Def H}-\eqref{Eigrela P} as follows.
 \begin{equation}\label{Rcal Def H}
H\triangleq\frac{1}{\alpha}
\left[
I_n-
\left(Dg_\text{\tiny$\alpha (-f)$}(x^*)\right)^T
\right],
\end{equation}
or equivalently,
 \begin{equation}\label{Recal DGT G1}
 \left(Dg_\text{\tiny$\alpha (-f)$}(x^*)\right)^T=
 I_n-\alpha H.
\end{equation}
It follows from Lemma \ref{Property H} that  $H$ has the
following expression:
\begin{equation}\label{Re H Form}
\begin{split}
H&=-\left(I_n- \alpha \hat{A}\right)^{-1}A
=\left(I_n+\alpha\left(-\hat{A}\right)\right)^{-1}\left(-A\right),
\end{split}
\end{equation}
where $A$ and $\hat{A}$ are defined by
 \eqref{Def A r}
and \eqref{Def hat A}.
Combining Eqs. \eqref{R xks=Degi} and \eqref{Recal DGT G1}, we have
\begin{equation}\label{Recal Eigrela P}
\begin{split}
&\mbox{eig}\left(
D\left[g_{\text{\tiny$\alpha(-f)$}}
\right]^{-1}(x^*)
\right)
\\
&=\mbox{eig}\left(
\left\{
D\left[g_{\text{\tiny$\alpha(-f)$}}
\right]^{-1}(x^*)
\right\}^T
\right)
\\
&=\mbox{eig}\left(
\left\{
\left(Dg_{\text{\tiny$\alpha(-f)$}}(x^*)
\right)^T\right\}^{-1}
\right)
\\
&=\mbox{eig}\left(
\left(
I_n-\alpha H\right)^{-1}
\right)
\\
&=\mbox{eig}\left(
\left\{
I_n-\alpha
\left[I_n+\alpha\left(-\hat{A}\right)\right]^{-1}
\left(-A\right)
\right\}^{-1}
\right),
\end{split}
\end{equation}
where the last equality is due to \eqref{Re H Form}.

Since $A=\nabla^2 f\left(x^*\right)$ and $x^*$ is a strictly
saddle point, $A$ has at least one negative eigenvalue. Hence,
$-A$ has at least one positive eigenvalue.
Consequently, by applying Lemma \ref{Key Theorem BCGD P Mod}
with identifications $-A\sim B$, $-\hat{A}\sim\hat{B}$,
$\alpha\sim \beta$ and  $\rho(A)\sim \rho(B)$, we know
, for any $\alpha \in \left(0, \frac{1}{\rho(A)}\right)$, there exists at least one eigenvalue $\lambda$
of $\left\{I_n-\alpha
\left[I_n+\alpha\left(-\hat{A}\right)\right]^{-1}\left(-A\right)
\right\}^{-1}$,
whose magnitude is strictly greater than one.

Furthermore, Lemma 7 in \cite{Panageas2016Gradient} implies
that gradient Lipschitz  continuous constant
 $L\geq \rho(A)=\rho\left(\nabla f (x^*)\right)$, which leads to
$\alpha \in \left(0, \frac{1}{L}\right)
\subseteq\left(0, \frac{1}{\rho(A)}\right)$.
Then, the above arguments and  \eqref{Recal Eigrela P} imply
that the proposition holds true.
\hfill\end{proof}

\subsection{\bf Main results of PBCD}

We first introduce the following proposition, which asserts that
the limit point of the sequence generated by the
PBCD method \ref{PBCD 5.1} is a critical point of $f$.
\begin{proposition}\label{limit=critical point PBCD}
Under Assumption \ref{Assumption f}, if
  $\left\{x_k\right\}_{k\geq 0}$
   is generated by the PBCD method \ref{PBCD 5.1}
with $0<\alpha<\frac{1}{L}$, $\lim\limits_{k}{x_k}$
exists and denote it as $x^*$, then $x^*$ is a
critical point of $f$, i.e., $\nabla f(x^*)
={\bf 0}$.
\end{proposition}
\begin{proof}
Notice that $\left\{x_k\right\}_{k\geq 0}$ is generated
by the PBCD method \ref{PBCD 5.1}.
We clearly have $x_k=\left\{
\left[g_{\text{\tiny$\alpha(-f)$}}\right]^{-1}
\right\}
^k(x_0)
\footnote
{$\left\{
\left[g_{\text{\tiny$\alpha(-f)$}}\right]^{-1}
\right\}
^k$
denotes the composition of
$
\left[g_{\text{\tiny$\alpha(-f)$}}\right]^{-1}
$
with itself $k$ times.}
$,
where $\left[g_{\text{\tiny$\alpha(-f)$}}\right]^{-1}$
denotes the inverse
mapping of $g_{\text{\tiny$\alpha(-f)$}}$
  defined by \eqref{Define g-f}. Hence,
  $\lim\limits_kx_k=\lim\limits_k
\left\{
\left[g_{\text{\tiny$\alpha(-f)$}}\right]^{-1}
\right\}
^k(x_0)=x^*$. Since
$\left[g_{\text{\tiny$\alpha(-f)$}}\right]^{-1}$
is a  diffeomorphism,
we immediately know that $x^*$ is a
fixed point of
$\left[g_{\text{\tiny$\alpha(-f)$}}\right]^{-1}$.
Equivalently, it is a fixed point of
$g_{\text{\tiny$\alpha(-f)$}}$.
It follows easily from
the definition \eqref{Define g-f} of
$g_{\text{\tiny$\alpha(-f)$}}$ that
$\nabla f(x^*)={\bf 0}$. Thus the proof is finished. \hfill
\end{proof}

Armed with the results established in  previous
subsections and the above Proposition \ref{limit=critical point PBCD}, we now state and prove our
main theorem of PBCD, whose proof is similar to that of
Theorem \ref{Main Theorem BCGD} in Subsection
\ref{Main Subsec of BCGD}. However, its proof is still
given as follows in detail for the sake of completeness.
\begin{theorem}\label{Main Theorem PBCD}
Let $f$ be a $C^2$ function and $x^*$ be a strict
saddle point. If  $\left\{x_k\right\}_{k\geq0}$ is generated
by the PBCD method \ref{PBCD 5.1} with $0<\alpha<\frac{1}{L}$, then
$$\mathbb{P}_\nu\left[\lim_{k}{x_k}=x^*\right]=0.$$
\end{theorem}
\begin{proof}
First, Proposition \ref{limit=critical point PBCD} implies that, if
$\lim\limits_{k}{x_k}$ exists then
it must be a critical point. Hence, we consider calculating the Lebesgue measure
(or probability with respect to the prior measure $\nu$)
of the set $\left[\lim\limits_{k}{x_k}=x^*\right]
 =W^s\left(x^*\right)$ (see Definition
 \ref{Global Stable Set}).

 Second, since Proposition \ref{PgDiff} means
 PBCD iterative mapping
 $\left[g_\text{\tiny$\alpha (-f)$}\right]^{-1}$  is
 a diffeomorphism, we replace $\phi$ and fixed point
 by $\left[g_\text{\tiny$\alpha (-f)$}\right]^{-1}$
 and the strict saddle point $x^*$ in
 the above Stable Manifold  Theorem
\ref{Sabl Mani Theorem} in Subsection \ref{Main Subsec of BCGD}, respectively.
Then  the manifold $W_{loc}^{s}(x^*)$
has strictly positive codimension because of
Proposition \ref{Jaco g-f eig} and $x^*$ being a strict
saddle point. Hence, $W_{loc}^{s}(x^*)$ has measure zero.

In what follows,
we are able to apply the same arguments
as in \cite{JLeeandMSimchowitz2016}
to finish the proof of the theorem. Since
the proof follows a similar pattern,
it is therefore omitted.
\hfill\end{proof}

Given the above Theorem \ref{Main Theorem PBCD},
we immediately  obtain the following
Theorem \ref{key theorem PBCD} and its
corollary by the same arguments as in the proofs
of Theorem 2 and Corollary 12 in
 \cite{Panageas2016Gradient}.
Therefore, we omit their proofs.

\begin{theorem}[Non-isolated]\label{key theorem PBCD}
Let $f: \mathbb{R}^n\rightarrow\mathbb{R}$ be a twice
continuously differentiable function
and $\sup\limits_{x\in\mathbb{R}^n}\left\|
\nabla f(x)\right\|^2<\infty$. The set of initial
conditions $x\in \mathbb{R}^n$ so that the PBCD
method \ref{PBCD 5.1} with step size
$0<\alpha<\frac{1}{L}$ converges to a strict saddle
point is of (Lebesgue) measure zero,
without assumption that critical points are isolated.
\end{theorem}

A straightforward corollary of Theorem
\ref{key theorem PBCD} is given below:
\begin{corollary}\label{key corolary PBCD}
Assume that the conditions of Theorem
\ref{key theorem PBCD}
 are satisfied and all saddle points of $f$
are strict. Additionally, 
the composite iterative mapping
$\left[g_{\text{\tiny$\alpha(-f)$}}\right]^{-1}$
denotes the inverse mapping of $g_{\text{\tiny$\alpha(-f)$}}$,
assume
$\lim\limits_k
\left\{
\left[g_{\text{\tiny$\alpha(-f)$}}\right]^{-1}
\right\}
^k
(x)$
exists for all $x$ in $\mathbb{R}^n$. Then
 $$
\mathbb{P}_\nu
\left[\lim\limits_k
\left\{
\left[g_{\text{\tiny$\alpha(-f)$}}\right]^{-1}
\right\}
^k
(x)=x^*\right]=1,$$
where $\left[g_{\text{\tiny$\alpha(-f)$}}\right]^{-1}$
denotes the inverse
mapping of $g_{\text{\tiny$\alpha(-f)$}}$
  defined by \eqref{Define g-f} and
$x^*$ is a local minimum.
\end{corollary}
\section{Several Technical Lemmas}\label{Several Lemmas}
In what follows, we will provide several technical
lemmas (Lemmas \ref{Bound etaecheckBeta}--
\ref{Key Theorem BCGD P Mod}), which provide the basis for proving that Jacobian of iterative mappings including $g_\text{\tiny$\alpha f$}$ defined by \eqref{Define g}
  in Section \ref{BCGD Sec}, $\psi$ defined by \eqref{Def Psi} in Section \ref{BMD sec} and $\left[g_{\text{\tiny$\alpha(-f)$}}\right]^{-1}$ defined by \eqref{Def g-f inv} in Section \ref{PBCD sec}
  at a strict saddle point has at least one eigenvalue with magnitude strictly
greater than one.
In particular, Lemma \ref{Key Lemma BCGD} gives a sufficiently exact
description of the distribution of the eigenvalues of matrix with the kind of
structure as $G$ which is related to $g_\text{\tiny$\alpha f$}$ and $\psi$,
while Lemmas \ref{Key Theorem BCGD P} and \ref{Key Theorem BCGD P Mod} give the similar description
 of that of the eigenvalues of matrix with the kind of
structure as $H$ which is related to $g_\text{\tiny$\alpha (-f)$}$.

Before presenting the grand result of this subsection, we first
introduce two notations which will be used in what follows.
Assume
$B\in\mathbb{R}^{n\times n}$ is symmetric matrix with $p\times p$ blocks.
Specifically,
\begin{equation}\label{Def B}
B\triangleq\left(B_{st}\right)_{1\leq s, \,t\leq p},
\end{equation}
and  its $(s, t)$-th block
\begin{equation}\label{Def B_ij}
B_{st}\in \mathbb{R}^{n_s\times n_t},\hspace*{3mm} 1\leq s,\, t\leq p,
\end{equation}
where $n_1$, $n_2$, $\ldots$, $n_p$ are $p$ positive
integer numbers satisfying
$\sum\limits_{s=1}^{p}n_s=n$.
In addition, we denote the strictly block lower triangular matrix
based on $B$ as
 \begin{equation}\label{Def bar B}
 \check{B}\triangleq
 \left(\check{B}_{st}\right)_{1\leq s,\,t\leq p}
 \end{equation}
 with $p\times p$ blocks and its $(s, t)$-th block is given by
 \begin{equation}\label{LsigmabarB}
\check{B}_{st}
=\begin{cases}
\begin{array}{lr}
    B_{st},      & s>t,\\
    \mathbf{0},  & s\leq t.
  \end{array}
  \end{cases}
\end{equation}
Similarly, we denote the strictly block upper triangular matrix
based on $B$ as
 \begin{equation}\label{hat B}
 \hat{B}\triangleq
 \left(\hat{B}_{st}\right)_{1\leq s,\,t\leq p}
 \end{equation}
 with $p\times p$ blocks and its $(s, t)$-th block is given by
 \begin{equation}\label{hat Bst}
\hat{B}_{st}
=\left\{
\begin{array}{lr}
    B_{st},      & s<t,\\
    \mathbf{0},  & s\geq t.
  \end{array}
  \right.
\end{equation}

Given the above notations, then we have the following
Lemma \ref{Bound etaecheckBeta}, which asserts that, for
any unit vector $\eta\in \mathbb{C}^n$,
the real parts of $\eta^H\check{B}\eta$
and $\eta^H\hat{B}\eta$ are both bounded
by the spectral radius of $B$.

\begin{lemma}\label{Bound etaecheckBeta}
Assume that $B$, $\check{B}$ and $\hat{B}$ are defined by
\eqref{Def B}, \eqref{Def bar B} and \eqref{hat B}, respectively.
Then for an arbitrary $n$ dimensional vector
$\eta\in \mathbb{C}^n$ with $\|\eta\|=1$, we have
\begin{equation}\label{leq Re ch B geq}
 -\rho(B) \leq
 \mbox{Re}\left(\eta^H\check{B}\eta\right)
 \leq \rho(B),
\end{equation}
and
\begin{equation}\label{leq Re ch D geq}
 -\rho(B)\leq
 \mbox{Re}\left(\eta^H\hat{B}\eta\right)
 \leq \rho(B).
\end{equation}
\end{lemma}
\begin{proof}
We first define a block diagonal matrix based
on $B$  below:
\begin{equation}\label{DiagB}
\tilde{B}\triangleq
\mbox{Diag}
\left(B_{11},B_{22},\ldots, B_{pp}\right),
\end{equation}
whose main diagonal blocks are the same as those
of $B$. Therefore, $B$ has the following decomposition, i.e.,
\begin{equation}\label{Decomp B}
B=\check{B}+\tilde{B}+\check{B}^T.
\end{equation}
In addition, it is obvious that
\begin{equation}\label{Bound B}
-\rho(B) I_n\preceq B \preceq \rho(B) I_n,
\end{equation}
which, combined with Theorem 4.3.15 in \cite{Horn:1985:MA:5509},
means that
\begin{equation}\label{Bound checkB}
-\rho(B) I_n\preceq \tilde{B} \preceq \rho(B) I_n.
\end{equation}
Assume that $\eta\in \mathbb{C}^n$ and $\|\eta\|=1$.
On the one hand,
\begin{equation}\label{Re etaH cB eta geq}
\begin{split}
2\rho(B)
&
\geq \eta^H \left(\rho(B) I_n+B\right)\eta\\
&=\eta^H
\left[\rho(B) I_n+
\left(\check{B}+\tilde{B}+\check{B}^T\right)
\right]\eta\\
&=\eta^H
\left(\rho(B) I_n+\tilde{B}\right)\eta+
\eta^H\left(\check{B}+\check{B}^T\right)\eta\\
&=\eta^H
\left(\rho(B) I_n+\tilde{B}\right)\eta+
\eta^H\left(\check{B}+\check{B}^H\right)\eta\\
&=\eta^H
\left(\rho(B) I_n+\tilde{B}\right)\eta+2\mbox{Re}\left(
\eta^H\check{B}\eta\right)\\
&\geq2\mbox{Re}\left(
\eta^H\check{B}\eta\right),
\end{split}
\end{equation}
where the first inequality is due to
\eqref{Bound B}; the first equality holds because
of \eqref{Decomp B};  the third equality follows from
the fact that $\check{B}$ is a real matrix;
and \eqref{Bound checkB} amounts to
the last inequality.\\
On the other hand, we also have
\begin{equation}
\label{Re etaH cB eta leq}
\begin{split}
2\rho(B)
&
\geq \eta^H \left(\rho(B) I_n-B\right)\eta\\
&=\eta^H
\left[\rho(B) I_n-
\left(\check{B}+\tilde{B}+\check{B}^T\right)
\right]\eta\\
&=\eta^H
\left(\rho(B) I_n-\tilde{B}\right)\eta-
\eta^H\left(\check{B}+\check{B}^T\right)\eta\\
&=\eta^H
\left(\rho(B) I_n-\tilde{B}\right)\eta-
\eta^H\left(\check{B}+\check{B}^H\right)\eta\\
&=\eta^H
\left(\rho(B) I_n-\tilde{B}\right)\eta-2\mbox{Re}\left(
\eta^H\check{B}\eta\right)\\
&\geq-2\mbox{Re}\left(
\eta^H\check{B}\eta\right),
\end{split}
\end{equation}
where the first inequality is due to
\eqref{Bound B}; the first equality holds because
of \eqref{Decomp B};  the third equality follows from
the fact that $\check{B}$ is a real matrix;
and \eqref{Bound checkB} amounts to
the last inequality.

Clearly, Eqs. \eqref{Re etaH cB eta geq}
and \eqref{Re etaH cB eta leq} lead to
\eqref{leq Re ch B geq}.
By using the same argument, we know
\eqref{leq Re ch D geq} holds true.
\hfill\end{proof}

The following lemma states that,
if $B$ is invertible, the real part of the eigenvalues
of $B^{-1}\left(I_n+\beta \check{B} \right)$ does not equal
zero for any $\beta\in
\left(0, \frac{1}{\rho\left(B\right)}
\right)$.
\begin{lemma}\label{Re neq 0 Lem}
Assume that $B$ and $\check{B}$ are defined by
\eqref{Def B} and \eqref{Def bar B}, respectively.
Moreover, suppose $B$ is invertible.
For any $\beta\in
\left(0, \frac{1}{\rho\left(B\right)}\right)$
and $t\in[0, 1]$,
if $\lambda$ is an eigenvalue of
$B^{-1}\left(I_n+t\beta \check{B} \right)$,
then $\mbox{Re}\left(\lambda\right)\neq 0$.
\end{lemma}
\begin{proof}
According to assumptions, $B$ is invertible.
Combining with $I_n+t\beta \check{B}$ being invertible with
any $t\in[0, 1]$,
we know $B^{-1}\left(I_n+t\beta \check{B}\right)$ is an
invertible matrix. Let $\lambda$ be
an eigenvalue of $B^{-1}\left(I+t\beta \check{B}\right)$
and $\xi$ be the corresponding eigenvector of
unit length. Then $\lambda\neq0$ and
\begin{equation}\label{Eigenvaluevector0}
B^{-1}\left(I_n+t\beta \check{B}\right)\xi=\lambda\xi,
\end{equation}
which is clearly equivalent to equation:
\begin{equation}\label{Eigenvaluevector1}
\left(I_n+t\beta \check{B}\right)\xi=\lambda B\xi.
\end{equation}
Premultiplying both sides of the above
equality by $\xi^H$, we arrive at
\begin{equation}\label{RealneqReal}
1+t\beta\xi^H\check{B}\xi=\lambda \xi^HB\xi.
\end{equation}
Note that $0<\beta
<\frac{1}{\rho\left(B\right)}$ and $t\in[0, 1]$.
Lemma \ref{Bound etaecheckBeta} implies that
$0<\mbox{Re}\left(1+t\beta\xi^H\check{B}\xi\right)<2$, i.e., the real
part of $1+t\beta\xi^H\check{B}\xi$
is a positive real number. We denote
$1+t\beta\xi^H\check{B}\xi$
as $a+bi$ where $0<a<2$.
If $\lambda=ci$ ($c\neq 0$)
is a purely imaginary number,
then
$\left(\xi^HB\xi\right)ci$ is also a purely imaginary
number because of $B$ being a real symmetric matrix.
Hence,
\eqref{RealneqReal} becomes
\begin{equation}
1+t\beta\xi^H\check B\xi
=a+bi=
\left(\xi^HB\xi\right)ci,
\end{equation}
which is a contradiction.
Hence, $\mbox{Re}(\lambda)\neq0$ is proved.
\hfill\end{proof}

The following Lemma states that,
if $B$ is invertible, the real part of the eigenvalues
of $\left(\beta B\right)^{-1}
\left(I_n+t\beta \hat{B} \right)$
is strictly larger than $\frac{1}{2}$ for any $\beta\in
\left(0, \frac{1}{\rho\left(B\right)}
\right)$ and $t\in[0, 1]$.
\begin{lemma}\label{Re neq 0 Lem P}
Assume that $B$ and $\hat{B}$ are defined by
\eqref{Def B} and \eqref{hat B}, respectively.
Moreover, suppose $B$ is invertible.
For any $\beta\in
\left(0, \frac{1}{\rho\left(B\right)}\right)$
and $t\in\left[0,\, 1\right]$,
if $\lambda$ is an eigenvalue of
$\left(\beta B\right)^{-1}
\left(I_n+t\beta \hat{B}
\right)$
and $\mbox{Re}\left(\lambda\right)>0$,
then $\mbox{Re}\left(\lambda\right)
\geq \frac{1}{2}+\frac{1-\beta\rho(B)}{\beta\rho(B)}
>\frac{1}{2}$.
\end{lemma}
\begin{proof}
The proof is lengthy and has been relegated to the
Appendix.
\hfill\end{proof}

The following Lemma plays a key role because it provides a sufficiently exact description
of the distribution of the eigenvalues of
$\left(I_n+\beta \check{B} \right)^{-1}B$, which
has the the same structure as $G$ defined by
\eqref{Def G}.

\begin{lemma}\label{Key Lemma BCGD}
Assume that $B$ and $\check{B}$ are defined by
\eqref{Def B} and \eqref{Def bar B}, respectively.
Furthermore if $\lambda_{\min}(B)<0$,
then, for an arbitrary
$\beta\in\left(0, \frac{1}{\rho\left(B\right)}\right)$,
there is at least
one eigenvalue $\lambda$ of
$\left(I_n+\beta \check{B} \right)^{-1}B$
which lies in
closed left half complex plane excluding origin, i.e.,
\begin{equation}\label{Key}
\forall\, \beta\in
\left(0, \frac{1}{\rho\left(B\right)}\right)
 \Rightarrow
\exists \, \lambda\in \left[\mbox{eig}
\left(\left(I_n+\beta \check{B} \right)^{-1}B\right)
\text{\scriptsize$\bigcap$}\, \Omega\right],
\end{equation}
where
\begin{equation}\label{Omega}
\Omega\triangleq\left\{a+bi\big{|}a, b\in
\mathbb{R}, a\leq0,~(a,b)\neq(0,0),
i=\sqrt{-1} \right\}.
\end{equation}
\end{lemma}
\begin{proof}
The proof is  lengthy and has
been relegated to the Appendix.
\hfill\end{proof}



Similar to Lemma \ref{Key Lemma BCGD},
the following lemma
plays a key role in this case
because it provides a sufficiently
exact description of the distribution of the eigenvalues of
$\beta\left(I_n+\beta \hat{B} \right)^{-1}B$, which
has the the same structure as $H$
defined by \eqref{Def H}.
\begin{lemma}\label{Key Theorem BCGD P}
Assume that $B$ and $\hat{B}$ are defined by
\eqref{Def B} and \eqref{hat B}, respectively.
Furthermore if $\lambda_{\max}(B)>0$,
then, for an arbitrary
$\beta\in\left(0, \frac{1}{\rho\left(B\right)}\right)$,
there is at least
one nonzero eigenvalue $\lambda$ of
$\beta \left(I_n+\beta \hat{B} \right)^{-1}B$
such that
\begin{equation}\label{Key P}
\frac{1}{\lambda}\in\Xi(\beta, B),
\end{equation}
where
\begin{equation}\label{Xi}
\Xi(\beta, B)\triangleq\left\{a+bi\Big{|}a, b\in
\mathbb{R},
\frac{1}{2}+\frac{1-\beta\rho(B)}{\beta\rho(B)}
\leq a,
i=\sqrt{-1} \right\}.
\end{equation}
\end{lemma}
\begin{proof}
The proof is lengthy and has been relegated to the
Appendix.\hfill
\end{proof}

Based on the above Lemma, we directly obtain
the following Lemma, which shows
there is at least one nonzero eigenvalue of
$\left[I_n-\beta \left(I_n+\beta \hat{B} \right)^{-1}B
\right]^{-1}$ with the same structure as $\left\{
D\left[g_{\text{\tiny$\alpha(-f)$}}
\right]^{-1}(x^*)
\right\}^T$,
whose magnitude is strictly greater than one.

\begin{lemma}\label{Key Theorem BCGD P Mod}
Assume that $B$ and $\hat{B}$ are defined by
\eqref{Def B} and \eqref{hat B}, respectively.
Furthermore suppose $\lambda_{\max}(B)>0$ and
$I_n-\beta \left(I_n+\beta \hat{B} \right)^{-1}B
$ is invertible,
then, for an arbitrary
$\beta\in\left(0, \frac{1}{\rho\left(B\right)}\right)$,
there is at least
one nonzero eigenvalue of
$\left[I_n-\beta \left(I_n+\beta \hat{B} \right)^{-1}B
\right]^{-1}$,
whose magnitude is strictly greater than one.
\end{lemma}
\begin{proof}
According to assumptions,
$I_n-\beta \left(I_n+\beta \hat{B} \right)^{-1}B
$ is invertible. Then
\begin{equation}
\lambda\in \mbox{eig}
\left(
\beta \left(I_n+\beta \hat{B} \right)^{-1}B
\right)\Leftrightarrow
\frac{1}{1-\lambda}\in\mbox{eig}\left(
\left[I_n-\beta \left(I_n+\beta \hat{B} \right)^{-1}B
\right]^{-1}\right).
 \end{equation}
Hence, the statement of the lemma is equivalent to that
there is at least
one  eigenvalue $\lambda$ of
$\beta \left(I_n+\beta \hat{B} \right)^{-1}B
$ such that
\begin{equation}
\left|\frac{1}{1-\lambda}\right|>1
\Leftrightarrow
\left|\frac{
\frac{1}{\lambda}
}{
\frac{1}{\lambda}
-1
}\right|>1.
\end{equation}
The above inequality is clearly equivalent to
\begin{equation}
\left|
\frac{1}{\lambda}
\right|>
\left|
\frac{1}{\lambda}
-1
\right|\Leftrightarrow
\frac{1}{2}<\mbox{Re}\left(\frac{1}{\lambda}\right).
\end{equation}
In addition, under the same assumptions,
Lemma \ref{Key Theorem BCGD P} shows that
there exists at least one eigenvalue $\lambda$ of
$\beta \left(I_n+\beta \hat{B} \right)^{-1}B
$ satisfying
\begin{equation}
\frac{1}{2}<
\left(\frac{1}{2}+\frac{1-\beta\rho(B)}{\beta\rho(B)}\right)
\leq\mbox{Re}\left(\frac{1}{\lambda}\right).
\end{equation}
Thus, the proof is finished.
 \hfill\end{proof}

\section{Conclusion}
In this paper,
given a non-convex twice continuously
differentiable cost
function with Lipschitz continuous gradient,
we prove that all of BCGD, BMD
and PBCD
converge to a local minimizer, almost surely with
random initialization.

As a by-product, it affirmatively answers the
open questions whether mirror descent or block coordinate
descent does not converge to saddle points in
 \cite{JLeeandMSimchowitz2016}.
More importantly, our results also hold true even
for the cost functions with non-isolated critical points,
which generalizes the results in
\cite{Panageas2016Gradient} as well.

By using the similar arguments, it is interesting to further research
 whether other
methods, such as ADMM, BCGD, BMD, PBCD and their variations for the problems with some specially structured constraints,
admit similar results or not.
\section{Acknowledgements}
The authors would like to thank Prof. Zhi-Quan (Tom) Luo
from 
The Chinese University of Hong Kong, Shenzhen,
for the helpful discussions on this work.
\section{Appendix}
\subsection{Proof of Lemma \ref{g12Diff}}
\begin{proof}
Firstly, we prove that $g_{\text{\tiny$\alpha f$}}^{1}$ with step size
$\alpha<\frac{1}{L}$ is a diffeomorphism.
The proof is given by the following four steps.\\
\\
{\bf (a)}
We prove that $g_{\text{\tiny$\alpha f$}}^{1}$ is injective
from $\mathbb{R}^n\rightarrow \mathbb{R}^n$ for
$\alpha<\frac{1}{L}$. Suppose that there
exist $x$ and $y$ such that $g_{\text{\tiny$\alpha f$}}^{1}(x)
=g_{\text{\tiny$\alpha f$}}^{1}(y)$.
Then
$x-y=\alpha U_1\left(\nabla_1 f(x)
-\nabla_1 f(y)\right)$
and \begin{equation}\label{x-y0}
\begin{split}
\|x-y\|&=
\alpha\|
U_1\left(\nabla_1 f(x)
-\nabla_1 f(y)\right)\|
\\
&=
\alpha\|
\nabla_{1}f(x)-\nabla_{1}f(y)\|
\leq\alpha\| \nabla f(x)-\nabla f(y)\|
\leq\alpha L\|x-y\|,
\end{split}
\end{equation}
where the second equality is due to
$U_1^TU_1=I_{n_1}$.
Since $\alpha L<1$, \eqref{x-y0} means $x=y$.\\
\\
{\bf (b)} To show $g_{\text{\tiny$\alpha f$}}^{1}$
is surjective, we construct an explicit inverse function.
Given a point $y$ in $\mathbb{R}^n$, suppose
it has the following partition,
\begin{equation*}
y=\left(\begin{array}{c}
y\text{\scriptsize$(1)$}\\
y\text{\scriptsize$(2)$}\\
\vdots\\
y\text{\scriptsize$(p)$}
\end{array}
\right).
\end{equation*}
Then we define $n-n_1$ dimensional vector
\begin{equation}\label{y^-(1)}
y_{_-}\text{\scriptsize$(1)$}\triangleq\left(\begin{array}{c}
y\text{\scriptsize$(2)$}\\
\vdots\\
y\text{\scriptsize$(p)$}
\end{array}
\right)
\end{equation}
and define a function
$\hat{f}\left(\cdot;{y_{_-}\text{\scriptsize$(1)$}}\right)$ $:
\,\mathbb{R}^{n_1}\rightarrow \mathbb{R}$,
\begin{equation*}
\hat{f}\left(x\text{\scriptsize$(1)$};{y_{_-}\text{\scriptsize$(1)$}}\right)\triangleq
f\left(
\left(\begin{array}{c}
x\text{\scriptsize$(1)$}\\
y_{_-}\text{\scriptsize$(1)$}\\
\end{array}
\right)
\right),
\end{equation*}
which is determined by function $f$ and the remained  block
coordinate vector $y_{_-}\text{\scriptsize$(1)$}$ of $y$.
Then, the proximal point mapping of
$-\hat{f}\left(\cdot;{y_{_-}\text{\scriptsize$(1)$}}\right)$ centered
at $y\text{\scriptsize$(1)$}$ is given by
\begin{equation}\label{x1y1}
x_y\text{\scriptsize$(1)$}
=\arg\min\limits_{x\text{\scriptsize$(1)$}}
\frac{1}{2}\|x\text{\scriptsize$(1)$}-y\text{\scriptsize$(1)$}
\|^2-
\alpha \hat{f}\left(x\text{\scriptsize$(1)$};{y_{_-}\text{\scriptsize$(1)$}}\right)
\end{equation}
For $\alpha<\frac{1}{L}$, the function above
is strongly convex with respect to $x\text{\scriptsize$(1)$}$,
so there is a unique minimizer. Let
$x_y\text{\scriptsize$(1)$}$
be the unique minimizer, then by the KKT condition,
\begin{equation}\label{KKTy(1)}
y\text{\scriptsize$(1)$}
=x_y\text{\scriptsize$(1)$}
-\alpha\nabla
\hat{f}\left(x_y\text{\scriptsize$(1)$}
;{y_{_-}\text{\scriptsize$(1)$}}\right)
=x_y\text{\scriptsize$(1)$}
-\alpha\nabla_1
f\left(\left(\begin{array}{c}
x_y\text{\scriptsize$(1)$}\\
y_{_-}\text{\scriptsize$(1)$}\\
\end{array}
\right)\right),
\end{equation}
where the second equality is due to the definition
of function
$\hat{f}\left(\cdot;y_{_-}\text{\scriptsize$(1)$}\right)$.
Let  $x_{y}$ be defined as
\begin{equation}\label{Defx_y}
x_y\triangleq\left(\begin{array}{c}
x_y\text{\scriptsize$(1)$}\\
y_{_-}\text{\scriptsize$(1)$}\\
\end{array}
\right),
\end{equation}
where $x_y\text{\scriptsize$(1)$}$ is defined by \eqref{x1y1}.
Accordingly,
\begin{equation*}
\begin{split}
y&=
\left(\begin{array}{c}
y\text{\scriptsize$(1)$}\\
y_{_-}\text{\scriptsize$(1)$}\\
\end{array}
\right)
=
\left(\begin{array}{c}
x_y\text{\scriptsize$(1)$}
-\alpha\nabla_1
f\left(\left(\begin{array}{c}
x_y\text{\scriptsize$(1)$}\\
y_{_-}\text{\scriptsize$(1)$}\\
\end{array}
\right)\right)
\\
y_{_-}\text{\scriptsize$(1)$}\\
\end{array}
\right)
=x_y-\alpha U_1\nabla_1
f\left(x_y\right)
=g_{\text{\tiny$\alpha f$}}^{1}(x_y),
\end{split}
\end{equation*}
where the first equality is due to the
definition of $y_{_-}\text{\scriptsize$(1)$}$
 (see \eqref{y^-(1)});
the second equality thanks to \eqref{KKTy(1)};
and the third equality holds because of the
definition of $x_y$ (see \eqref{Defx_y}).

Hence, $x_{y}$ is mapped to $y$ by
the mapping $g_{\text{\tiny$\alpha f$}}^{1}$.\\
\\
{\bf (c)} In addition, combined with
$\nabla^2 f(x)U_1U_1^T\in\mathbb{R}^{n\times n}$
and  $U_1^T\nabla^2 f(x)U_1
\in\mathbb{R}^{n_1\times n_1}$,
Theorem 1.3.20 in \cite{Horn:1985:MA:5509} means
\begin{equation}\label{Conta Zero}
\mbox{eig}\left(\nabla^2 f(x)U_1U_1^T\right)
\subseteq
\mbox{eig}\left(U_1^T\nabla^2 f(x)U_1\right)
\cup{\left\{\bf 0\right\}}.
\end{equation}
Moreover, since
\begin{equation}\label{Submatrix}
U_1^T\nabla^2 f(x)U_1=
\frac{\partial^2 f(x)}{\partial x\text{\scriptsize$(1)$}
\partial x\text{\scriptsize$(1)$}}
\end{equation}
which
 is the $n_1$-by-$n_1$ principal
 submatrix of $\nabla^2 f(x)$, it follows from
 Theorem 4.3.15 in \cite{Horn:1985:MA:5509} that
 \begin{equation}\label{Contains}
 \mbox{eig}\left(\frac{\partial^2 f(x)}
 {\partial x\text{\scriptsize$(1)$}^2}
 \right)\subseteq
 \left[\lambda_{\min}
 \left(\nabla^2 f(x)\right)
 ,
 \lambda_{\max}
 \left(\nabla^2 f(x)\right)\right]
 \subseteq\left[-L, L\right],
 \end{equation}
where the last relation holds because of
Eq. p\eqref{Lipschitz} and
Lemma 7 in \cite{Panageas2016Gradient}. Since $\alpha<\frac{1}{L}$,
Eqs. \eqref{Conta Zero}, \eqref{Submatrix} and
\eqref{Contains} imply that
 \begin{equation}\label{rangleLL}
 \mbox{eig}\left(
 \alpha \nabla^2 f(x)U_1U_1^T
 \right)
\subseteq\left(-1, 1\right).
\end{equation}
Hence, $Dg_{\text{\tiny$\alpha f$}}^{1}(x)
=I-\alpha \nabla^2 f(x)U_1U_1^T$
is invertible for $\alpha<\frac{1}{L}$. \\
\\
{\bf (d)} Note that we have shown that  $g_{\text{\tiny$\alpha f$}}^{1}$
is bijection, and continuously differentiable.
Since $Dg_{\text{\tiny$\alpha f$}}^{1}(x)$
is invertible  for $\alpha<\frac{1}{L}$,
the inverse function theorem guarantees $\left[g_{\text{\tiny$\alpha f$}}^{1}\right]^{-1}$
is continuously differentiable. Thus,
$g_{\text{\tiny$\alpha f$}}^{1}$
is a diffeomorphism.

Secondly, it is obvious that similar arguments
can be applied to verify that, $g_{\text{\tiny$\alpha f$}}^{s},
s=2, \ldots, p,$ are also diffeomorphisms.
\hfill\end{proof}
\subsection{Proof of Lemma \ref{Property G}}
\begin{proof}
First, we define
recursively
 \begin{equation}\label{Def Gs}
G[s]\triangleq \frac{1}{\alpha}
\left[
I_n-
\prod\limits_{t=s}^{1}
\left(I_n-\alpha U_tA_t\right)
\right],
\hspace*{3mm}
 1\leq s \leq p,
\end{equation}
which, combined with
  \eqref{Dgx*T1} and \eqref{Def G}, implies that
  $G=G[p]$.
In addition, it is easily seen that
\begin{equation}\label{PropUiT}
U_s^T\left(I_n-\alpha U_tA_t\right)=
U_s^T-\alpha  U_s^TU_tA_t=U^T_s
\end{equation}
  when $ s\neq t $. If $k<s$, then
  the above \eqref{PropUiT}
 follows that
  \begin{equation}\label{PropUiTGssmal}
  \begin{split}
U_s^TG[k]&=
\frac{1}{\alpha}U_s^T
\left[
I_n-
\prod\limits_{t=k}^{1}
\left(I_n-\alpha U_tA_t\right)
\right]
=
\frac{1}{\alpha}
\left[
U_s^T-
U_s^T\prod\limits_{t=k}^{1}
\left(I_n-\alpha U_tA_t\right)
\right]
={\bf 0}.
\end{split}
\end{equation}
Consequently,   if $1\leq s< q\leq p$, then
 \begin{equation}\label{UTG=UTGs}
 \begin{split}
U^T_{s}\left(\alpha G[q]\right)
&=U_s^T\left[
I_n-
\prod\limits_{t=q}^{1}
\left(I_n-\alpha U_tA_t\right)
\right]\\
&=U_s^T
I_n-
U_s^T\prod\limits_{t=q}^{s+1}
\left(I_n-\alpha U_tA_t\right)
\prod\limits_{t=s}^{1}
\left(I_n-\alpha U_tA_t\right)
\\
&=U_s^T
I_n-
U_s^T
\prod\limits_{t=s}^{1}
\left(I_n-\alpha U_tA_t\right)\\
&=U_s^T
\left[
I_n-
\prod\limits_{t=s}^{1}
\left(I_n-\alpha U_tA_t\right)
\right]\\
&=U^T_s\alpha G[s],
\end{split}
\end{equation}
where the fourth equality is due to
\eqref{PropUiT} and the last equality uses the
definition \eqref{Def Gs} of $G[s]$. Particularly if $q=p$, the above equation becomes
\begin{equation}\label{UTG=UTGs11}
U^T_s\alpha G[s]=U^T_s\alpha G[p]=U^T_s\alpha G.
\end{equation}
From equality \eqref{UTG=UTGs11}, we further have
\begin{align*}
U^T_{s}\left(\alpha G\right)
&=
U^T_{s}\left[I_n-
\left(I_n-\alpha G[s]\right)\right]\\
&\overset{\mbox{(a)}}{=}U^T_{s}\left[I_n-
\left(I_n-\alpha U_{s}A_s\right)
\left(I_n-\alpha G[s-1]\right)\right]\\
&=\alpha U^T_{s} G[{s-1}]
+\alpha A_s-\alpha^2 A_s G[{s-1}]\\
&\overset{\mbox{(b)}}{=}
\alpha A_s-\alpha^2 A_s G[{s-1}]\\
&\overset{\mbox{(c)}}{=}
\alpha A_s-\alpha^2 A_s
\sum\limits_{t=1}^{p}U_tU_t^T G[{s-1}]\\
&\overset{\mbox{(d)}}{=}
\alpha A_s-\alpha^2 A_s
\sum\limits_{t=1}^{s-1}U_tU_t^T G[{s-1}]\\
&\overset{\mbox{(e)}}{=}
\alpha A_s-\alpha^2 A_s
\sum\limits_{t=1}^{s-1}U_tU_t^T G[p]\\
&\overset{\mbox{(f)}}{=}
\alpha A_s-\alpha^2 A_s
\sum\limits_{t=1}^{s-1}U_tU_t^T G\\
&=\alpha A_s-\alpha^2
\sum\limits_{t=1}^{s-1}A_sU_tU_t^T G\\
&=\alpha A_s-\alpha^2
\sum\limits_{t=1}^{s-1}A_{st}U_t^T G
\end{align*}
where (a) uses the definitions of $G[s]$
in \eqref{Def Gs}; (b) is due to \eqref{PropUiTGssmal};
(c) thanks to the definition \eqref{DefUi} of $U_t$; (d) uses
\eqref{PropUiTGssmal} again; (e) holds because of
\eqref{UTG=UTGs}; (f) follows from $G[p]=G$.
Dividing both sides of the above equation by
$\alpha$, we have \eqref{Porperty UiG}.
\hfill\end{proof}
\subsection{Proof of Lemma \ref{Key Lemma BCGD}}
\begin{proof}
We divide the proof into  two cases.

{\bf Case 1: $B$ is an invertible matrix.} 
In this case, we clearly have,
\begin{equation}
\left(
\left(I_n+\beta \check{B} \right)^{-1}B
\right)^{-1}
=B^{-1}\left(I_n+\beta \check{B} \right).
 \end{equation}
 In what follows, we will prove that
\eqref{Key} is true by using
Lemma \ref{Zero Lemma} in Appendix.
Fristly, 
we define an analytic function with $t$ as a parameter:
  \begin{equation}\label{Def Xzt}
\begin{split}
\mathcal{X}_t(z)
 &\triangleq
 \det
 \left\{
 zI_n-\left[(1-t)
 B^{-1}+tB^{-1}
 \left(I_n+\beta \check{B} \right)\right]
 \right\},
  \hspace*{3mm}
  0\leq t\leq1
  \\
  &=
 \det
 \left\{
 zI_n-B^{-1}
 \left(I_n+t \beta \check{B} \right)
 \right\},
  \hspace*{3mm}
  0\leq t\leq1.
  \end{split}
  \end{equation}
Moreover, define a closed rectangle in the complex plane
as
\begin{equation}\label{Def D g}
 \mathscr{D}\triangleq\left\{a+bi|~-2\nu\leq
a\leq 0,~-2\nu\leq b\leq2\nu\right\},
\end{equation}
with $\nu$ being defined below:
\begin{equation}\label{Def nu}
\begin{split}
\nu&\triangleq
\left\|B^{-1}\right\|+\frac{1}{\rho\left(B\right)}
\left\|B^{-1}\check{B}\right\|
\geq \left\|B^{-1}\|
+t\beta\|B^{-1}\check{B}
\right\|
\geq\left\|B^{-1}\left(I_n+t\beta \check{B}\right)
\right\|, \hspace*{3mm}\forall\, t\in[0,1],
\end{split}
\end{equation}
where the first inequality holds because of
$\beta\in\left(0, \frac{1}{\rho\left(B\right)}\right)$
and $t\in [0, 1]$.
Note  that the boundary $\partial\mathscr{D}$ of
$\mathscr{D}$ consists of a finite number
 of smooth curves. Specifically, define
\begin{equation}
\begin{split}
\gamma_1&\triangleq
\left\{a+bi|~a=0,~-2\nu\leq b\leq2\nu\right\},\\
\gamma_2&\triangleq
\left\{a+bi|~a=-2\nu,~-2\nu\leq b\leq2\nu\right\},\\
\gamma_3&\triangleq
\left\{a+bi|~-2\nu\leq
a\leq 0,~b=2\nu\right\},\\
\gamma_4&\triangleq
\left\{a+bi|~-2\nu\leq
a\leq 0,~b=-2\nu\right\},
\end{split}
\end{equation}
then
\begin{equation}\label{Boundary D}
\begin{split}
\partial\mathscr{D}=
\gamma_1\cup\gamma_2\cup\gamma_3\cup\gamma_4.
\end{split}
\end{equation}
In order to apply Lemma \ref{Zero Lemma},
we will show that
\begin{equation}\label{Neq BoundD}
 \mathcal{X}_t(z)\neq0,~\forall\, t\in[0,1],
~\forall\, z\in\partial\mathscr{D}.
\end{equation}
On the one hand, since the spectral norm of a
matrix is lager than or equal to its spectral
radius, the above inequality \eqref{Def nu} yields that,
for any $t\in [0, 1]$,
every eigenvalue of $B^{-1}\left(I+t\beta \check{B}\right)$
has a magnitude less than $\nu$.
Note that for an arbitrary
$z\in \gamma_2\cup\gamma_3\cup\gamma_4$, then
$|z|\geq2\nu$.  Consequently,
\begin{equation}\label{Neq Gamma234}
 \mathcal{X}_t(z)\neq0,~\forall\, t\in[0,1],
~\forall\, z\in \gamma_2\cup\gamma_3\cup\gamma_4.
\end{equation}
On the other hand, it follows from Lemma \ref{Re neq 0 Lem} that,
for any $t\in [0, 1]$,
 if $\lambda$ is an eigenvalue of
$B^{-1}\left(I_n+t\beta \check{B} \right)$,
then $\mbox{Re}\left(\lambda\right)\neq 0$.
Hence, $\lambda\notin\gamma_1$.
As mentioned at the beginning of the proof,
 $B^{-1}\left(I_n+t\beta \check{B} \right)$ is
invertible. Hence, there are no
zero eigenvalues of
$B^{-1}\left(I_n+t\beta \check{B} \right)$
in $\gamma_1$, i.e.,
\begin{equation}\label{Neq Gamma1}
 \mathcal{X}_t(z)\neq0,~\forall\, t\in[0,1],
~\forall\, z\in \gamma_1.
\end{equation}
Combining \eqref{Neq Gamma234} and \eqref{Neq Gamma1}, we obtain
\eqref{Neq BoundD}.

As a result, it follows from Lemma \ref{Zero Lemma}
in Appendix, Eqs. \eqref{Def Xzt}, \eqref{Def D g} and \eqref{Neq BoundD}
that
$\mathcal{X}_{0}(z)
=\det
 \left\{
 zI_n-B^{-1}
 \right\}$
 and
 $\mathcal{X}_{1}(z)
 =\det
 \left\{
 zI_n-B^{-1}
 \left(I_n+\beta \check{B} \right)
 \right\}$
have the  same number of zeros in $\mathscr{D}$.
Note that $\lambda_{\min}\left(B\right)<0$
implies that there is at least one  negative eigenvalue
$\frac{1}{\lambda_{\min}\left(B\right)}$
of $B^{-1}$. Recalling the definition
\eqref{Def nu} of $\nu$, we know
$\left|\frac{1}{\lambda_{\min}\left(B\right)}\right|\leq\nu$.
Thus $\frac{1}{\lambda_{\min}\left(B\right)}$  must
lie inside $\mathscr{D}$.
In other words, the number of zeros of
$\mathcal{X}_{0}(z)$ inside $\mathscr{D}$ is at least one,
which in turn shows the number of zeros of
$\mathcal{X}_{1}(z)$ inside $\mathscr{D}$ is at least
one as well.
Thus, there must exist at least one eigenvalue
of $B^{-1}
\left(I_n+\beta \check{B}\right) $ lying
inside $\mathscr{D}$.
We denote it as $x+yi$, then $-2\nu<x<0$ and
$-2\nu<y<2\nu$.
Consequently,
$\frac{1}{x+yi}=\frac{x-yi}{x^2+y^2}$
is an eigenvalue of
$\left(I_n+\beta\check{B}\right)^{-1}B$ with
real part $\frac{x}{x^2+y^2}<0$. Hence,
$\frac{1}{x+yi}$ lies in $\Omega$ defined by \eqref{Omega}
and the proof is finished in this case.
\\

{\bf Case 2: $B$ is a singular matrix.} In this
case, we will apply  perturbation  theorem based
on the results in {\bf Case 1} to
prove \eqref{Key}.

Suppose the multiplicity of zero eigenvalue of $B$ is
$m$ and $B$ has an eigenvalue decomposition in the form of
\begin{equation}\label{B eig decom}
B=V
\left(
  \begin{array}{cc}
    \Theta & 0 \\
    0 & 0 \\
  \end{array}
\right)V^T=V_1\Theta V_1^T,
\end{equation}
where $\Theta=\mbox{Diag}\left(\theta_1,
\theta_2, \ldots, \theta_{n-m}\right)$,
$\theta_s$, $s=1, \ldots, n-m$,
 are the nonzero eigenvalues of
$B$ and
\begin{equation}\label{Def U}
V=\left(
\begin{array}{cc}
    V_1 & V_2
  \end{array}
\right)
\end{equation}
is an orthogonal matrix and $V_1$ consists of
the first $(n-m)$ columns of $V$.

Denote
\begin{equation}
\delta
\triangleq\min
\left\{\left|\theta_1\right|,
\left|\theta_2\right|, \ldots,
\left|\theta_{n-m}\right|\right\}.
\end{equation}
For any $\epsilon\in\left(0, \,\delta\right)$, we define
\begin{equation}\label{Def Bepsilon0}
B\left(\epsilon\right)
\triangleq B+\epsilon I_n,
\end{equation}
then,
\begin{equation}
\mbox{eig}\left(B\left(\epsilon\right)\right)
=\left\{\theta_1+\epsilon,
\theta_2+\epsilon, \ldots,
\theta_{n-m}+\epsilon, \epsilon\right\}\not\ni0,
\hspace*{3mm}
\forall\, \epsilon\in\left(0, \,\delta\right),
\end{equation}
and
\begin{equation}\label{laminleq0}
\lambda_{\min}\left(B\left(\epsilon\right)\right)=
\lambda_{\min}\left(B\right)+\epsilon
\leq-\delta+\epsilon<0,\hspace*{3mm}
\forall\, \epsilon\in\left(0, \,\delta\right),
\end{equation}
where the first inequality  is due to the definition
of $\delta$ and
$\min
\left\{\theta_1,
\theta_2, \ldots,
\theta_{n-m}\right\}=\lambda_{\min}\left(B\right)<0$.

Since $B$ is defined by \eqref{Def B},
 $B\left(\epsilon\right)$ has
$p\times p$ blocks form as well.
Specifically,
\begin{equation}\label{Def Bepsilon}
B\left(\epsilon\right)
=
\left(B\left(\epsilon\right)_{st}\right)_{1\leq s, \,t\leq p},
\end{equation}
and  its $(s, t)$-th block is given by
\begin{equation}\label{Def Bepsilon_ij}
B\left(\epsilon\right)_{st}
=
\left\{
\begin{array}{lr}
B_{st}+\epsilon I_{n_s},   & s=t,\\
B_{st},                    & s\neq t,
  \end{array}
  \right.
\end{equation}
where  $n_1$, $n_2$, $\ldots$,
$n_p$ are $p$
positive integer numbers satisfying
$\sum\limits_{s=1}^{p}n_s=n$.
Similar to the definition \eqref{Def bar B}
of $\check{B}$, we denote the strictly block lower triangular matrix
based on $B\left(\epsilon\right)$ as
 \begin{equation}\label{Def checkBepsilon}
\check{B}\left(\epsilon\right)\triangleq
 \left(\check{B}\left(\epsilon\right)_{st}\right)_{1\leq s,\,t\leq p}
 \end{equation}
 with $p\times p$ blocks and its $(s, t)$-th block is given by
 \begin{equation}\label{LsigmabarBepsilon}
 \begin{split}
\check{B}\left(\epsilon\right)_{st}
&=\left\{
\begin{array}{lr}
    B\left(\epsilon\right)_{st},      & s>t,\\
    \mathbf{0},  & s\leq t,
  \end{array}
  \right.\\
  &=\left\{
\begin{array}{lr}
    B_{st},      & s>t,\\
    \mathbf{0},  & s\leq t,
  \end{array}
  \right.\\
  &=\check{B}_{st},
  \end{split}
\end{equation}
where the second equality holds because of
\eqref{Def Bepsilon_ij}; the last equality is due to
\eqref{LsigmabarB}.
 It follows easily from Eqs. \eqref{Def bar B}, \eqref{LsigmabarB}
 \eqref{Def checkBepsilon} and
 \eqref{LsigmabarBepsilon} that
\begin{equation}\label{CheckBepsilon=B}
\check{B}\left(\epsilon\right)=
\check{B}.
\end{equation}
Consequently,
\begin{equation}
\begin{split}
\left(I_n+\beta \check{B}\left(\epsilon\right)
 \right)^{-1}B\left(\epsilon\right)
 &=\left(I_n+\beta \check{B}
 \right)^{-1}B\left(\epsilon\right)
  =\left(I_n+\beta \check{B}
 \right)^{-1}\left(B+\epsilon I_n\right),
 \end{split}
 \end{equation}
where the first equality  is due to
\eqref{CheckBepsilon=B} and the second equality holds
because of \eqref{Def Bepsilon0}.
For simplicity, let
\begin{equation}
\lambda^\beta_s\hspace*{-1.5mm}
\left(\epsilon\right),\hspace*{3mm}
s=1, \ldots, n,
\end{equation}
be the eigenvalues of $\left(I_n+\beta \check{B}
 \right)^{-1}\left(B+\epsilon I_n\right)$.

Note that for any $\epsilon\in \left(0, \, \delta\right)$,
$B\left(\epsilon\right)$ is invertible and
$\lambda_{\min}\left(B\left(\epsilon\right)\right)<0$ (see
\eqref{laminleq0}).
According to the definitions of  $B\left(\epsilon\right)$
and $\check{B}\left(\epsilon\right)$,
a similar argument in {\bf Case 1}
can be applied  with the identifications
$B\left(\epsilon\right)\sim B$,
$\check{B}\left(\epsilon\right)\sim \check{B}$,
$\beta \sim \beta$ and
$\rho\left(B(\epsilon)\right)
\sim
\rho(B)$,
to prove that,
for any $\beta\in\left(0,
\frac{1}{\rho(B(\epsilon))}\right)$,
there must exist at least one eigenvalue of
$\left(I_n+\beta \check{B}
\left(\epsilon\right)\right)^{-1}B\left(\epsilon\right)$
which lies in $\Omega$ defined by \eqref{Omega}.
Taking  into account definition \eqref{Def Bepsilon0}, we
have
$\rho\left(B(\epsilon)\right)\leq
\rho\left(B\right)+\epsilon$.
Hence, for any
$\epsilon\in \left(0, \, \delta\right)$
and
$\beta\in \left(0, \,
\frac{1}{\rho\left(B\right)+\epsilon}\right)
\subseteq
\left(0, \,
\frac{1}{\rho\left(B(\epsilon)\right)}\right)$,
 there exists at least one index denoted as
 $s(\epsilon)\in\left\{1, 2, \ldots, n\right\}$
 such that
 \begin{equation}\label{Epsi Belo Ome}
 \lambda^\beta_{s\text{\tiny$(\epsilon)$}}
 \hspace*{-1.5mm}\left(\epsilon\right)\in \Omega.
 \end{equation}

Furthermore, it is well known that
the eigenvalues of a matrix $M$ are continuous functions
of the entries of $M$.
Therefore, for any $\beta\in\left(0, \,
\frac{1}{\rho\left(B\right)}\right)$,
$\lambda^\beta_s\hspace*{-1.5mm}\left(\epsilon\right)$ is a
continuous function of $\epsilon$ and
 \begin{equation}\label{Alllimit=0}
 \lim\limits_{\epsilon\rightarrow 0^+}
 \lambda^\beta_s\hspace*{-1.5mm}\left(\epsilon\right)=
 \lambda^\beta_s\hspace*{-1.5mm}\left(0\right)
 , \hspace*{3mm}
s=1, \ldots, n,
 \end{equation}
 where $\lambda^\beta_s\hspace*{-1.5mm}\left(0\right)$
is the eigenvalue of
$\left(I_n+\beta \check{B}\right)^{-1}B$.

In what follows, we will prove that \eqref{Key} holds true by contradiction.

Suppose for sake of contradiction that,
 there exists a $\beta^*\in\left(0, \,
\frac{1}{\rho\left(B\right)}\right)$
such that, for any $s\in \left\{1, \ldots, n\right\}$,
we have
\begin{equation}\label{Lambda0 notIn Omega}
\lambda^{\beta^*}_s\hspace*{-2mm}\left(0\right)\notin\Omega,
\end{equation}
 where $\lambda^{\beta^*}_s\hspace*{-2mm}\left(0\right)$
is the eigenvalue of
$\left(I_n+\beta^* \check{B}\right)^{-1}B$.

According to Lemma \ref{Multi Zero Lem} in Appendix and the
assumption that the multiplicity of zero eigenvalue of
$B$ is $m$, we know
that the multiplicity of eigenvalue $0$ of
$\left(I_n+\beta^* \check{B}\right)^{-1}B$ is
$m$ as well. Then there are exactly $m$ eigenvalue functions of
$\epsilon$ whose limits are 0 as $\epsilon$
approaches zero from above. Without loss of generality,
we assume
\begin{equation}\label{LimiLambda=0}
 \lim\limits_{\epsilon\rightarrow0^+}
 \lambda^{\beta^*}_s\hspace*{-2mm}\left(\epsilon\right)=
 \lambda^{\beta^*}_s\hspace*{-2mm}\left(0\right)=0, \hspace*{3mm}
 s=1, \ldots, m,
 \end{equation}
and
\begin{equation}\label{LimiLambdaneq0}
 \lim\limits_{\epsilon\rightarrow0^+}
\lambda^{\beta^*}_s\hspace*{-2mm}\left(\epsilon\right)=
\lambda^{\beta^*}_s\hspace*{-2mm}\left(0\right)\neq0, \hspace*{3mm}
 s=m+1, \ldots, n.
\end{equation}

Subsequently, 
 under the assumption \eqref{Lambda0 notIn Omega},
we will first prove that there exists a
$\delta_1^*>0$ such that, for any
$\epsilon\in \left(-\delta_1^*,\, 0\right)$, then
$\beta^*\in\left(0, \frac{1}{\rho(B)+\epsilon}\right)
\subseteq\left(0, \frac{1}{\rho(B(\epsilon))}\right)$
and
there exists no $s\in \left\{1, \ldots, n\right\}$
such that
$\lambda^{\beta^*}_s\hspace*{-2mm}
\left(\epsilon\right)$
belongs to
$\Omega$. This would contradict
\eqref{Epsi Belo Ome}. The proof is given by the following
four steps.\\

{\bf Step (a):} Under the assumption \eqref{Lambda0 notIn Omega},
we first prove that there exists a
$\delta_1^*>0$ such that, for any
$\epsilon\in \left(-\delta_1^*,\, 0\right)$,
$\beta^*\in\left(0, \frac{1}{\rho(B)+\epsilon}\right)
\subseteq\left(0, \frac{1}{\rho(B(\epsilon))}\right)$
and
there does not exist any $s\in \left\{m+1, \ldots, n\right\}$
such that
$\lambda^{\beta^*}_s\hspace*{-2mm}
\left(\epsilon\right)$
belongs to
$\Omega$.

Taking into account the definition of
$\Omega$,  Eq. \eqref{Lambda0 notIn Omega}
 and $\beta^*\in\left(0, \,
\frac{1}{\rho\left(B\right)}\right)$, we
imply that, there exists a $\bar{\delta}$ such that,
for any $\epsilon\in(0, \bar{\delta})$,
\begin{equation}\label{inbeta}
\beta^*\in\left(0, \frac{1}{\rho(B)+\epsilon}\right)
\subseteq\left(0, \frac{1}{\rho(B(\epsilon))}\right)
\end{equation}
and
 \begin{equation}
\mbox{Re}\left(\lambda^{\beta^*}_s\hspace*{-2mm}\left(0\right)\right)>0,
\hspace*{3mm} \forall \, s\in\left\{m+1, \ldots, n\right\}.
\end{equation}
Moreover, note that $\lambda^{\beta^*}_s\hspace*{-2mm}\left(\epsilon\right)$ is a
continuous function of $\epsilon$
and \eqref{LimiLambdaneq0} holds. Combining with the above
inequalities, we know that there exists a
$\delta_1^*>0$ with $\delta_1^*\leq \bar{\delta}$
such that
\begin{equation}\label{Neigh Re Lambda m+1}
\left|
\lambda_s^{\beta^*}\hspace*{-2mm}\left(\epsilon\right)-
\lambda_s^{\beta^*}\hspace*{-2mm}\left(0\right)
\right|<\frac{1}{3}
\mbox{Re}\left(\lambda_s^{\beta^*}\hspace*{-2mm}\left(0\right)\right),
\hspace*{3mm}
 \forall\, s\in \left\{m+1, \ldots, n\right\},
 \, \forall\,
  \epsilon\in \left[0, \, \delta_1^*\right],
\end{equation}
which further means that
\begin{equation}\label{Neigh Re Lambda m+1}
0<\frac{2}{3}
\mbox{Re}\left(\lambda_s^{\beta^*}\hspace*{-2mm}\left(0\right)\right)
<
\mbox{Re}\left(\lambda_s^{\beta^*}\hspace*{-2mm}\left(\epsilon\right)\right),
\hspace*{3mm}
 \forall\, s\in \left\{m+1, \ldots, n\right\},
 \, \forall\,
  \epsilon\in \left[0, \, \delta_1^*\right].
\end{equation}
Hence, we arrive at
\begin{equation}\label{m+1notinOmega}
\lambda_s^{\beta^*}\hspace*{-2mm}\left(\epsilon\right)\notin\Omega,
\hspace*{3mm}
 \forall\, s\in \left\{m+1, \ldots, n\right\},
 \, \forall\,
  \epsilon\in \left[0, \, \delta_1^*\right].
\end{equation}

{\bf Step (b):} In this step,  we will prove that
there exists a
$\delta_2^*>0$ such that, for any
$\epsilon\in \left(0, \, \delta_2^*\right]$ and
$s\in \left\{1, \ldots, m\right\}$,
\begin{equation}
\mbox{Re}\left(\lambda_s^{\beta^*}\hspace*{-2mm}\left(\epsilon\right)
\right)>0,
\end{equation} which immediately implies that
\begin{equation}\label{Lambdas notIn Omega}
\lambda_s^{\beta^*}\hspace*{-2mm}\left(\epsilon\right)\notin\Omega,
\hspace*{3mm}
 \forall\, s\in \left\{1, \ldots, m\right\},
 \, \forall\,
  \epsilon\in \left(0, \, \delta_2^*\right].
\end{equation}
For simplicity, let
\begin{equation}\label{Cij}
\check{C}_{ij}\triangleq V_i^T\check{B}V_j,
\hspace*{2mm} 1\leq i, j\leq2,
\end{equation}
where $V_1$ and $V_2$ are given by \eqref{Def U}.

In what follows, we take an arbitrary
$s\in \left\{1, \ldots, m\right\}$.
 Since we assume that $\lambda_s^{\beta^*}\hspace*{-2mm}\left(\epsilon\right)$ is
 the eigenvalue of $\left(I_n+\beta^* \check{B}\right)^{-1}
 \left(B+\epsilon I_n\right)$,
 then, for any $\epsilon\in \left(0, \, \delta\right)$,
\begin{equation*}
\begin{split}
&\det\left\{
\left(I_n+\beta^* \check{B}\right)^{-1}
 \left(B+\epsilon I_n\right)
 -\lambda_s^{\beta^*}\hspace*{-2mm}\left(\epsilon\right) I_n\right\}
 =0,
\end{split}
\end{equation*}
which is clearly equivalent to
\begin{equation}\label{Multiply Det}
 \det\left\{
\left(I_n+\beta^* \check{B}\right)
\right\}
\det\left\{
\left(I_n+\beta^* \check{B}\right)^{-1}
 \left(B+\epsilon I_n\right)
 -\lambda_s^{\beta^*}\hspace*{-2mm}\left(\epsilon\right) I_n\right\}
 =0, 
 \forall \, \epsilon\in \left(0, \, \delta\right).
\end{equation}
It is easily seen from the above Eq.
\eqref{Multiply Det}
that, for any $\epsilon\in \left(0, \, \delta\right)$,
\begin{equation}\label{Zero Det}
\begin{split}
&0=\det\left\{
 \left(B+\epsilon I_n\right)
 -\lambda_s^{\beta^*}\hspace*{-2mm}\left(\epsilon\right)
 \left(I_n+\beta^* \check{B}\right)\right\}
 \\
 &=
\det\left\{
 \left(
 V
\left(
  \begin{array}{cc}
    \Theta & 0 \\
    0 & 0 \\
  \end{array}
\right)V^T
+\epsilon I_n\right)
 -\lambda_s^{\beta^*}\hspace*{-2mm}\left(\epsilon\right)
 \left(I_n+\beta^* \check{B}\right)\right\}
\\
 &=
\det\left\{
 \left(
  \begin{array}{cc}
    \Theta+\epsilon I_{n-m} & 0 \\
    0 & \epsilon I_m \\
  \end{array}
\right)
 -\lambda_s^{\beta^*}\hspace*{-2mm}\left(\epsilon\right)
 \left(I_n+\beta^* V^T\check{B}V\right)\right\}
 \\
 &=
\det\left\{
 \left(
  \begin{array}{cc}
\Theta+\epsilon I_{n-m}
-\lambda_s^{\beta^*}\hspace*{-2mm}\left(\epsilon\right)
\left(I_{n-m}+\beta^* \check{C}_{11}\right)
& -\lambda_s^{\beta^*}\hspace*{-2mm}\left(\epsilon\right)
\beta^* \check{C}_{12} \\
  -\lambda_s^{\beta^*}\hspace*{-2mm}\left(\epsilon\right)
    \beta^* \check{C}_{21} &
  \epsilon I_m  -\lambda_s^{\beta^*}\hspace*{-2mm}\left(\epsilon\right)
    \left(I_{m}+\beta^* \check{C}_{22}\right) \\
  \end{array}
\right)
 \right\},
\end{split}
\end{equation}
where the second equality is due to
\eqref{B eig decom} and the last equality holds
because of Eqs.  \eqref{Def U} and \eqref{Cij}.

Besides, recalling
\begin{equation}\label{ReLimiLambda=0}
 \lim\limits_{\epsilon\rightarrow0^+}
 \lambda_s^{\beta^*}\hspace*{-2mm}\left(\epsilon\right)=
 \lambda_s^{\beta^*}\hspace*{-2mm}\left(0\right) =0,
 \end{equation}
 we have
 \begin{equation}\label{LimTheta}
 \lim\limits_{\epsilon\rightarrow0^+}
 \left[
\Theta+\epsilon I_{n-m}
-\lambda_s^{\beta^*}\hspace*{-2mm}\left(\epsilon\right)
\left(I_{n-m}+\beta^* \check{C}_{11}\right)
\right]=\Theta,
\end{equation}
which is an invertible matrix because $\Theta$ is
given by \eqref{B eig decom}.
Clearly, the above \eqref{LimTheta}
further implies that,
 there exists a $\delta_1>0$ such that, for any
$\epsilon\in\left[0, \, \delta_1\right]$,
\begin{equation}\label{Inverse}
\Theta+\epsilon I_{n-m}
-\lambda_s^{\beta^*}\hspace*{-2mm}\left(\epsilon\right)
\left(I_{n-m}+\beta^* \check{C}_{11}\right)
\end{equation}
is an invertible matrix as well,
and
\begin{equation}\label{Inv Limit}
\lim\limits_{\epsilon\rightarrow0^+}
 \left[
\Theta+\epsilon I_{n-m}
-\lambda_s^{\beta^*}\hspace*{-2mm}\left(\epsilon\right)
\left(I_{n-m}+\beta^* \check{C}_{11}\right)
\right]^{-1}=\Theta^{-1}.
\end{equation}
Since the inverse of a matrix $M$, $M^{-1}$, is a continuous
function of the elements of $M$, there exists
a $\delta_2>0$ with $\delta_2\leq\delta_1$,
such that, for any
$\epsilon\in\left[0, \, \delta_2\right]$,
\begin{equation}\label{Bound Inv}
 \left\|\left[
\Theta+\epsilon I_{n-m}
-\lambda_s^{\beta^*}\hspace*{-2mm}\left(\epsilon\right)
\left(I_{n-m}+\beta^* \check{C}_{11}\right)
\right]^{-1}
\right\|\leq 2 \left\|\Theta^{-1}\right\|.
\end{equation}
Consequently, for any
$\epsilon\in\left[0, \, \delta_2\right]$, it
follows easily from \eqref{Zero Det} and
\eqref{Inverse} that
\begin{equation}
\begin{split}
0&=\det
\left\{
\Theta+\epsilon I_{n-m}
-\lambda_s^{\beta^*}\hspace*{-2mm}\left(\epsilon\right)
\left(I_{n-m}+\beta^* \check{C}_{11}\right)
\right\}
\\
&\hspace*{4mm}\times\det
\Big{\{}
\epsilon I_m  -\lambda_s^{\beta^*}\hspace*{-2mm}\left(\epsilon\right)
    \left(I_{m}+\beta^* \check{C}_{22}\right)
\\
&
\hspace*{17mm}
-\left(\beta^*\right)^2
\left(\lambda_s^{\beta^*}\hspace*{-2mm}\left(\epsilon\right)\right)^2
\check{C}_{21}
\left[\Theta+\epsilon I_{n-m}
-\lambda_s^{\beta^*}\hspace*{-2mm}\left(\epsilon\right)
\left(I_{n-m}+\beta^* \check{C}_{11}\right)
\right]^{-1}
\check{C}_{12}
\Big{\}},
\end{split}
\end{equation}
which, combined with the fact that
$\Theta+\epsilon I_{n-m}
-\lambda_s^{\beta^*}\hspace*{-2mm}\left(\epsilon\right)
\left(I_{n-m}+\beta^* V_1^T\check{B}V_1\right)
$ is an invertible matrix again
(see \eqref{Inverse}), means that
\begin{equation*}
\epsilon I_m  -\lambda_s^{\beta^*}\hspace*{-2mm}\left(\epsilon\right)
    \left(I_{m}+\beta^* \check{C}_{22}\right)
-\left(\beta^*\right)^2
\left(\lambda_s^{\beta^*}\hspace*{-2mm}\left(\epsilon\right)\right)^2
\check{C}_{21}
\left[\Theta+\epsilon I_{n-m}
-\lambda_s^{\beta^*}\hspace*{-2mm}\left(\epsilon\right)
\left(I_{n-m}+\beta^* \check{C}_{11}\right)
\right]^{-1}
\check{C}_{12}
\end{equation*}
is a singular matrix. Therefore, there exists one nonzero
 vector $\upsilon\left(\epsilon\right)\in \mathbb{C}^m$
 with
 $\left\|\upsilon\left(\epsilon\right)\right\|=1$
 satisfying
\begin{equation}
\begin{split}
&\Big{\{}
\epsilon I_m  -\lambda_s^{\beta^*}\hspace*{-2mm}\left(\epsilon\right)
    \left(I_{m}+\beta^* \check{C}_{22}\right)
\\
&\hspace*{1mm}
-\left(\beta^*\right)^2
\left(\lambda_s^{\beta^*}\hspace*{-2mm}\left(\epsilon\right)\right)^2
\check{C}_{21}
\left[\Theta+\epsilon I_{n-m}
-\lambda_s^{\beta^*}\hspace*{-2mm}\left(\epsilon\right)
\left(I_{n-m}+\beta^* \check{C}_{11}\right)
\right]^{-1}
\check{C}_{12}
\Big{\}}
\upsilon\left(\epsilon\right)={\bf 0}.
\end{split}
\end{equation}
Moreover, for any
$\epsilon\in\left[0, \, \delta_2\right]$,
premultiplying both sides of the above equality by
$\left(\upsilon\left(\epsilon\right)\right)^H$, we have
\begin{equation}
\begin{split}
\left(\upsilon\left(\epsilon\right)\right)^H
&\Big{\{}
\epsilon I_m  -\lambda_s^{\beta^*}\hspace*{-2mm}\left(\epsilon\right)
    \left(I_{m}+\beta^* \check{C}_{22}\right)
\\
&\hspace*{1mm}
-\left(\beta^*\right)^2
\left(\lambda_s^{\beta^*}\hspace*{-2mm}\left(\epsilon\right)\right)^2
\check{C}_{21}
\left[\Theta+\epsilon I_{n-m}
-\lambda_s^{\beta^*}\hspace*{-2mm}\left(\epsilon\right)
\left(I_{n-m}+\beta^* \check{C}_{11}\right)
\right]^{-1}
\check{C}_{12}
\Big{\}}
\upsilon\left(\epsilon\right)=0,
\end{split}
\end{equation}
or equivalently,
\begin{equation}
\begin{split}
&\epsilon -\lambda_s^{\beta^*}\hspace*{-2mm}\left(\epsilon\right)
    \left(1+\beta^*
    \left(\upsilon\left(\epsilon\right)\right)^H
    \check{C}_{22}
    \upsilon\left(\epsilon\right)
    \right)
\\
&=\left(\beta^*\right)^2
\left(\lambda_s^{\beta^*}\hspace*{-2mm}\left(\epsilon\right)\right)^2
\left(\upsilon\left(\epsilon\right)\right)^H
\check{C}_{21}
\left[\Theta+\epsilon I_{n-m}
-\lambda_s^{\beta^*}\hspace*{-2mm}\left(\epsilon\right)
\left(I_{n-m}+\beta^* \check{C}_{11}\right)
\right]^{-1}
\check{C}_{12}
\upsilon\left(\epsilon\right).
\end{split}
\end{equation}
Dividing both sides of the above equation
by $\lambda_s^{\beta^*}\hspace*{-2mm}\left(\epsilon\right)$,
we have, for any
$\epsilon\in\left(0, \, \delta_2\right]$,
\begin{equation}\label{Lim Bound 0}
\begin{split}
0&\leq \left|
\frac{\epsilon}{\lambda_s^{\beta^*}\hspace*{-2mm}\left(\epsilon\right)}
    -\left(1+\beta^*
    \left(\upsilon\left(\epsilon\right)\right)^H
    \check{C}_{22}
    \upsilon\left(\epsilon\right)
    \right)
\right|
\\
&=\left|
\left(\beta^*\right)^2
\lambda_s^{\beta^*}\hspace*{-2mm}\left(\epsilon\right)
\left(\upsilon\left(\epsilon\right)\right)^H
\check{C}_{21}
\left[\Theta+\epsilon I_{n-m}
-\lambda_s^{\beta^*}\hspace*{-2mm}\left(\epsilon\right)
\left(I_{n-m}+\beta^* \check{C}_{11}\right)
\right]^{-1}
\check{C}_{12}
\upsilon\left(\epsilon\right)
\right|\\
&\leq\left(\beta^*\right)^2
\left|
\lambda_s^{\beta^*}\hspace*{-2mm}\left(\epsilon\right)
\right|
\left\|
\check{C}_{21}
\right\|
\left\|
\left[\Theta+\epsilon I_{n-m}
-\lambda_s^{\beta^*}\hspace*{-2mm}\left(\epsilon\right)
\left(I_{n-m}+\beta^* \check{C}_{11}\right)
\right]^{-1}
\right\|
\left\|
\check{C}_{12}
\right\|\\
&\leq 2 \left(\beta^*\right)^2
\left|
\lambda_s^{\beta^*}\hspace*{-2mm}\left(\epsilon\right)
\right|
\left\|
\check{C}_{21}
\right\|
\left\|
\Theta^{-1}
\right\|
\left\|
\check{C}_{12}
\right\|,
\end{split}
\end{equation}
where the second inequality is due to
$\left\|\upsilon\left(\epsilon\right)\right\|=1$
and Cauchy--Schwartz inequality; the last
inequality follows from
\eqref{Bound Inv}.
As $\lim\limits_{\epsilon\rightarrow0^+}
\lambda_s^{\beta^*}\hspace*{-2mm}\left(\epsilon\right)=
 \lambda_s^{\beta^*}\hspace*{-2mm}\left(0\right) =0$, the right-hand
 limit in the above expression  is equal to zero.
Hence, we have
 \begin{equation}\label{Lim}
\begin{split}
&\lim\limits_{\epsilon\rightarrow0^+}
\left|
\frac{\epsilon}{\lambda_s^{\beta^*}\hspace*{-2mm}\left(\epsilon\right)}
    -\left(1+\beta^*
    \left(\upsilon\left(\epsilon\right)\right)^H
    \check{C}_{22}
    \upsilon\left(\epsilon\right)
    \right)
\right|=0.
\end{split}
\end{equation}
Recall that
\begin{equation}\label{ReLowBound}
\begin{split}
&\mbox{Re}
\left(1+\beta^*
    \left(\upsilon\left(\epsilon\right)\right)^H
    \check{C}_{22}
    \upsilon\left(\epsilon\right)
    \right)\\
    &=\mbox{Re}
\left(
1+\beta^*
    \left(\upsilon\left(\epsilon\right)\right)^H
    V_2^T\check{B}V_2
    \upsilon\left(\epsilon\right)
    \right)\\
    &=
    \mbox{Re}\left(
    1+\beta^*
    \left(\upsilon\left(\epsilon\right)\right)^H
    V_2^H\check{B}V_2
    \upsilon\left(\epsilon\right)
    \right)\\
    &=\mbox{Re}\left(
    1+\beta^*
    \left(V_2\upsilon\left(\epsilon\right)\right)^H
   \check{B}V_2\upsilon\left(\epsilon\right)
    \right)\\
    &\geq 1-\beta^*\rho\left(B\right)
    >0,
\end{split}
\end{equation}
where the first equality follows from \eqref{Cij};
the second equality holds because $V_2$ is a real
matrix (see eq. \eqref{B eig decom});
the first inequality follows easily from Lemma \ref{Re neq 0 Lem} and
$\left\|V_2\upsilon\left(\epsilon\right)\right\|=
\left(\upsilon\left(\epsilon\right)\right)^H
V_2^HV_2\upsilon\left(\epsilon\right)=
\left(\upsilon\left(\epsilon\right)\right)^H
\upsilon\left(\epsilon\right)=1
$;  and the last inequality thanks to $\beta^*\in
\left(0,\, \frac{1}{\rho\left(B\right)}\right)$.
Consequently, there exists a $\delta_3$ with $\delta_3\leq
\delta_2$ such that,
for any $\epsilon\in\left(0, \, \delta_3\right]$,
\begin{equation}\label{Lim Bound}
\begin{split}
\frac{1}{3}\left(1-\beta^*\rho\left(B\right)\right)
&\geq\left|
\frac{\epsilon}{\lambda_s^{\beta^*}\hspace*{-2mm}\left(\epsilon\right)}
    -\left(1+\beta^*
    \left(\upsilon\left(\epsilon\right)\right)^H
    \check{C}_{22}
    \upsilon\left(\epsilon\right)
    \right)
\right|\\
&\geq
\left|
\mbox{Re}
\left(
\frac{\epsilon}{\lambda_s^{\beta^*}\hspace*{-2mm}\left(\epsilon\right)}
    -\left(1+\beta^*
    \left(\upsilon\left(\epsilon\right)\right)^H
    \check{C}_{22}
    \upsilon\left(\epsilon\right)
    \right)
    \right)
\right|\\
&=
\left|
\mbox{Re}
\left(
\frac{\epsilon}{\lambda_s^{\beta^*}
\hspace*{-2mm}\left(\epsilon\right)}
\right)
    -\mbox{Re}\left(1+\beta^*
    \left(\upsilon\left(\epsilon\right)\right)^H
    \check{C}_{22}
    \upsilon\left(\epsilon\right)
    \right)
\right|.
\end{split}
\end{equation}
The above inequalities \eqref{ReLowBound}
and
\eqref{Lim Bound} imply that,
for any
$\epsilon\in\left(0, \, \delta_3\right]$,
\begin{equation}
\begin{split}
0&<\frac{2}{3}\left(1-\beta^*\rho\left(B\right)\right)
\leq \mbox{Re}
\left(
\frac{\epsilon}
{\lambda_s^{\beta^*}
\hspace*{-2mm}\left(\epsilon\right)}
\right)
=
\frac{\epsilon}
{\left|
\lambda_s^{\beta^*}\hspace*{-2mm}\left(\epsilon\right)
\right|^2}
\mbox{Re}
\left(
\lambda_s^{\beta^*}\hspace*{-2mm}\left(\epsilon\right)
\right).
\end{split}
\end{equation}
Since the above argument is applied to any
$s\in \left\{1, \ldots, m\right\}$, there exists
a $\delta_2^*>0$ such that,
\begin{equation}
\begin{split}
0<\mbox{Re}
\left(
\lambda_s^{\beta^*}\hspace*{-2mm}\left(\epsilon\right)
\right),
\hspace*{3mm}
 \forall\, s\in \left\{1, \ldots, m\right\},
 \, \forall\,
  \epsilon\in \left(0, \, \delta_2^*\right],
\end{split}
\end{equation}
which further implies that
\begin{equation}\label{mnotinOmega}
\begin{split}
\lambda_s^{\beta^*}\hspace*{-2mm}\left(\epsilon\right)
\notin\Omega,
\hspace*{3mm}
 \forall\, s\in \left\{1, \ldots, m\right\},
 \, \forall\,
  \epsilon\in \left(0, \, \delta_2^*\right].
\end{split}
\end{equation}

{\bf Step (c):}
Combining \eqref{m+1notinOmega} and
\eqref{mnotinOmega},
we arrive at,
\begin{equation}\label{mnotinOmega1}
\begin{split}
\lambda_s^{\beta^*}\hspace*{-2mm}\left(\epsilon\right)
\notin\Omega,
\hspace*{3mm}
 \forall\, s\in \left\{1, \ldots, n\right\},
 \, \forall\,
  \epsilon\in \left(0, \, \delta_3^*\right],
\end{split}
\end{equation}
where $\delta_3^*=\min\left\{\delta_1^*,
\, \delta_2^*\right\}$.

{\bf Step (d):} 
Let
\begin{equation}\label{Def deltastar}
\delta^*=\min\left\{\frac{1}{2}\delta,
 \,\delta_3^*\right\}.
\end{equation}
Then, for any $\epsilon\in\left(0, \,
\delta^*\right)$, we have
\begin{equation}\label{delta}
\epsilon\in\left(0, \,
\delta\right),
\end{equation}
\begin{equation}\label{beta in 0 epsilonstar}
\beta^*\in\left(0, \,
\frac{1}{\rho\left(B\right)
+\epsilon}\right)\subseteq
\left(0, \,
\frac{1}{\rho\left(B(\epsilon)\right)}\right)
\end{equation}
and
\begin{equation}\label{mnotinOmega12}
\begin{split}
\lambda_s^{\beta^*}\hspace*{-2mm}\left(\epsilon\right)
\notin\Omega,
\hspace*{3mm}
 \forall\, s\in \left\{1, \ldots, n\right\},
\end{split}
\end{equation}
where \eqref{delta} uses the definition
 \eqref{Def deltastar}
of $\delta^*$; \eqref{beta in 0 epsilonstar}
 is due to
the definition \eqref{Def deltastar}
of $\delta^*$ (i.e.,
$\delta^*\leq\delta^*_3
\leq\delta_1^*\leq\bar{\delta}$)
and \eqref{inbeta}; and
\eqref{mnotinOmega12} thanks to
the definition \eqref{Def deltastar}
of $\delta^*$ and
 \eqref{mnotinOmega1}.
Clearly, this contradicts \eqref{Epsi Belo Ome}.

Hence, we conclude that \eqref{Key} holds true.
\hfill\end{proof}
\subsection{Proof of Lemma \ref{LpsiDiff}}
\begin{proof}
We first prove that $g_{1}$ with step size
$\alpha<\frac{\mu}{L}$ is a diffeomorphism.
The proof is given by the following four steps.\\
{\bf (a)}
We first prove that $\psi_{1}$ is injective
from $\mathbb{R}^{n_1}\rightarrow \mathbb{R}^{n_1}$ for
$\alpha<\frac{\mu}{L}$. Suppose that there
exist $x$ and $y$ such that $\psi_{1}(x)=\psi_{1}(y)$,
which implies that
\begin{equation}\label{Cases}
\begin{cases}
\left[\nabla\varphi_1\right]^{-1}
\left(
\nabla\varphi_1\left(x\text{\scriptsize$(t)$}\right)
-\alpha\nabla_1f\left(x\right)\right)
=\left[\nabla\varphi_1\right]^{-1}
\left(\nabla\varphi_1\left(y\text{\scriptsize$(t)$}\right)
-\alpha\nabla_1f\left(y\right)\right), & t=1,\\
x\text{\scriptsize$(t)$}=
y\text{\scriptsize$(t)$}, & t=2, 3, \ldots, p.
\end{cases}
\end{equation}
Since Lemma \ref{Phi Diff} in Appendix asserts
that $\nabla\varphi_1$  is a diffeomorphism,
then $\left[\nabla\varphi_1\right]^{-1}$ is a
diffeomorphism as well. Hence, the above  equality \eqref{Cases}
is equivalent to
\begin{equation}\label{Cases1}
\begin{cases}
\nabla\varphi_1\left(x\text{\scriptsize$(t)$}\right)
-\alpha\nabla_1f\left(x\right)
=\nabla\varphi_1\left(y\text{\scriptsize$(t)$}\right)
-\alpha\nabla_1f\left(y\right), & t=1,\\
x\text{\scriptsize$(t)$}=
y\text{\scriptsize$(t)$}, & t=2, 3, \ldots, p.
\end{cases}
\end{equation}
In particular, $
\nabla\varphi_1\left(x\text{\scriptsize$(1)$}\right)
-\alpha\nabla_1f\left(x\right)
=\nabla\varphi_1\left(y\text{\scriptsize$(1)$}\right)
-\alpha\nabla_1f\left(y\right)
$ further implies that
\begin{equation}\label{x-y}
\begin{split}
\left\|x\text{\scriptsize$(1)$}
-y\text{\scriptsize$(1)$}\right\|
&\leq\frac{1}{\mu}\left\|
\nabla\varphi_1\left(x\text{\scriptsize$(1)$}\right)
-\nabla\varphi_1\left(y\text{\scriptsize$(1)$}\right)
\right\|\\
&=\frac{\alpha}{\mu}\left\|
\nabla_1 f\left(x\right)-\nabla_1 f\left(y\right)\right\|
\\
&\leq\frac{\alpha}{\mu}\|\nabla f(x)-\nabla f(y)\|
\\
&\leq\frac{\alpha L}{\mu}\|x-y\|
\\
&=\frac{\alpha L}{\mu}\|x\text{\scriptsize$(1)$}
-y\text{\scriptsize$(1)$}\|,
\end{split}
\end{equation}
where the first inequality is due to strong
convexity (see \eqref{Strongly Convex}); the
third inequality thanks to \eqref{Lipschitz};
the last equality holds because of \eqref{Cases1}.
Since $\alpha L<1$, \eqref{x-y} means
$x\text{\scriptsize$(1)$}
=y\text{\scriptsize$(1)$}$.
Combining with \eqref{Cases1}, we have $x=y$.

{\bf (b)} To show $\psi_1$ is surjective,
we construct an explicit inverse function.
Given a point $y$ in $\mathbb{R}^n$,
suppose it has the following partition,
\begin{equation}
y=\left(\begin{array}{c}
y\text{\scriptsize$(1)$}\\
y\text{\scriptsize$(2)$}\\
\vdots\\
y\text{\scriptsize$(p)$}
\end{array}
\right).
\end{equation}
Then we define $n-n_1$ dimensional vector
\begin{equation}\label{y^-(1)BMD}
y_{_-}\text{\scriptsize$(1)$}\triangleq\left(\begin{array}{c}
y\text{\scriptsize$(2)$}\\
\vdots\\
y\text{\scriptsize$(p)$}
\end{array}
\right)
\end{equation}
and define a function
$\bar{f}\left(\cdot;{y_{_-}\text{\scriptsize$(1)$}}\right)$ $:
\,\mathbb{R}^{n_1}\rightarrow \mathbb{R}$,
\begin{equation}
\bar{f}\left(x\text{\scriptsize$(1)$};{y_{_-}\text{\scriptsize$(1)$}}\right)\triangleq
f\left(
\left(\begin{array}{c}
x\text{\scriptsize$(1)$}\\
y_{_-}\text{\scriptsize$(1)$}\\
\end{array}
\right)
\right),
\end{equation}
which is determined by function $f$ and the
remained  block coordinate vector
$y_{_-}\text{\scriptsize$(1)$}$ of $y$.
Consider the following problem,
\begin{equation}\label{x1y1BMD}
\min\limits_{x\text{\scriptsize$(1)$}}
B_{\varphi_1}\left(x\text{\scriptsize$(1)$},
y\text{\scriptsize$(1)$}\right)-
\alpha \bar{f}\left(x\text{\scriptsize$(1)$};{y_{_-}\text{\scriptsize$(1)$}}\right)
\end{equation}

For $\alpha<\frac{\mu}{L}$, the function above
is strongly convex with respect to
$x\text{\scriptsize$(1)$}$, so there is a unique
minimizer of the problem \eqref{x1y1BMD}.
Let $x_y\text{\scriptsize$(1)$}$
be the unique minimizer ,
then by the KKT condition,
\begin{equation}
\begin{split}
\nabla \varphi_1 \left(y\text{\scriptsize$(1)$}\right)
&=\nabla \varphi_1 \left(
x_y\text{\scriptsize$(1)$}
\right)-\alpha\nabla
\bar{f}\left(x_y\text{\scriptsize$(1)$}
;{y_{_-}\text{\scriptsize$(1)$}}\right),
\end{split}
\end{equation}
which is equivalent to
\begin{equation}
\begin{split}\label{KKTy(1)BMD}
y\text{\scriptsize$(1)$}=
\left[\nabla \varphi_1\right]^{-1} \left(
\nabla \varphi_1 \left(
x_y\text{\scriptsize$(1)$}
\right)
-\alpha\nabla_1
f\left(\left(\begin{array}{c}
x_y\text{\scriptsize$(1)$}\\
y_{_-}\text{\scriptsize$(1)$}\\
\end{array}
\right)\right)
\right).
\end{split}
\end{equation}
Let  $x_{y}$ be defined as
\begin{equation}\label{Defx_y BMD}
x_y\triangleq\left(\begin{array}{c}
x_y\text{\scriptsize$(1)$}\\
y_{_-}\text{\scriptsize$(1)$}\\
\end{array}
\right),
\end{equation}
where $x_y\text{\scriptsize$(1)$}$
is determined by \eqref{KKTy(1)BMD}.
Accordingly,
\begin{equation}
\begin{split}
y&=
\left(\begin{array}{c}
y\text{\scriptsize$(1)$}\\
y_{_-}\text{\scriptsize$(1)$}\\
\end{array}
\right)\\
&=
\left(\begin{array}{c}
\left[\nabla\varphi_1\right]^{-1}
\left(
\nabla \varphi_1 \left(
x_y\text{\scriptsize$(1)$}
\right)
-\alpha\nabla_1
f\left(\left(\begin{array}{c}
x_y\text{\scriptsize$(1)$}\\
y_{_-}\text{\scriptsize$(1)$}\\
\end{array}
\right)\right)
\right)\\
y_{_-}\text{\scriptsize$(1)$}\\
\end{array}
\right)\\
&=\left(I_n-U_1U_1^T\right)x_y
+U_1\left[\nabla\varphi_1\right]^{-1}
\left(\nabla\varphi_1\left(x_y\text{\scriptsize$(1)$}\right)
-\alpha\nabla_1f\left(x_y\right)\right)\\
&=\psi_{1}(x_y),
\end{split}
\end{equation}
where the first equality is due to the definition
of $y_{_-}\text{\scriptsize$(1)$}$
(see \eqref{y^-(1)BMD}); the second equality
thanks to \eqref{KKTy(1)BMD}; and the third
equality holds because of definition of $U_1$
(see \eqref{DefUi}); since $\psi_{1}(x_y)$
is defined by \eqref{MBDA psi1}, the last equality
holds true.\\
\\
Hence, $x_{y}$ is mapped to $y$ by the
mapping $\psi_{1}$.

{\bf (c)} In addition, recalling
\eqref{D MBDA psis}, we have
\begin{equation}\label{D MBDA psi1}
\begin{split}\nonumber
&D\psi_1\left(x\right)
\\
&=\left(I_n-U_1U_1^T\right)+
\\
&\left\{U_1\nabla^2
\varphi_1\left(x\text{\scriptsize$(1)$}\right)
-\alpha\nabla^2f\left(x\right)U_1\right\}
\left\{
\nabla^2\varphi_1
\left\{\left[\nabla \varphi_1\right]^{-1}
\left(\nabla\varphi_1\left(x\text{\scriptsize$(1)$}\right)
-\alpha\nabla_1f\left(x\right)\right)\right\}
\right\}^{-1}U_1^T\\
&=
\left(\begin{array}{cccc}
{\bf 0} & {\bf 0} &  \cdots & {\bf 0} \\
{\bf 0}  & I_{n_2} &  \cdots & {\bf 0} \\
\vdots   & \vdots  &  \ddots & \vdots  \\
{\bf 0}  & {\bf 0} &  \cdots & I_{n_p}
\end{array}
\right)+\\
&
\left(\begin{array}{cccc}
\left\{\nabla^2
\varphi_1\left(x\text{\scriptsize$(1)$}\right)
-\alpha A_{11}\right\}
\left\{
\nabla^2\varphi_1
\left\{\left[\nabla \varphi_1\right]^{-1}
\left(\nabla\varphi_1\left(x\text{\scriptsize$(1)$}\right)
-\alpha\nabla_1f\left(x\right)\right)\right\}
\right\}^{-1}
  & {\bf 0} &  \cdots & {\bf 0} \\
A_{21}
\left\{
\nabla^2\varphi_1
\left\{\left[\nabla \varphi_1\right]^{-1}
\left(\nabla\varphi_1\left(x\text{\scriptsize$(1)$}\right)
-\alpha\nabla_1f\left(x\right)\right)\right\}
\right\}^{-1}
  & {\bf 0} &  \cdots & {\bf 0} \\
\vdots   & \vdots  &  \ddots & \vdots  \\
A_{p1}
\left\{
\nabla^2\varphi_1
\left\{\left[\nabla \varphi_1\right]^{-1}
\left(\nabla\varphi_1\left(x\text{\scriptsize$(1)$}\right)
-\alpha\nabla_1f\left(x\right)\right)\right\}
\right\}^{-1}
  & {\bf 0} &  \cdots & {\bf 0}
\end{array}
\right)
\\
&=
\left(\begin{array}{cccc}
\left\{\nabla^2
\varphi_1\left(x\text{\scriptsize$(1)$}\right)
-\alpha A_{11}\right\}
\left\{
\nabla^2\varphi_1
\left\{\left[\nabla \varphi_1\right]^{-1}
\left(\nabla\varphi_1\left(x\text{\scriptsize$(1)$}\right)
-\alpha\nabla_1f\left(x\right)\right)\right\}
\right\}^{-1}
  & {\bf 0} &  \cdots & {\bf 0} \\
A_{21}
\left\{
\nabla^2\varphi_1
\left\{\left[\nabla \varphi_1\right]^{-1}
\left(\nabla\varphi_1\left(x\text{\scriptsize$(1)$}\right)
-\alpha\nabla_1f\left(x\right)\right)\right\}
\right\}^{-1}
  & I_{n_2} &  \cdots & {\bf 0} \\
\vdots   & \vdots  &  \ddots & \vdots  \\
A_{p1}
\left\{
\nabla^2\varphi_1
\left\{\left[\nabla \varphi_1\right]^{-1}
\left(\nabla\varphi_1\left(x\text{\scriptsize$(1)$}\right)
-\alpha\nabla_1f\left(x\right)\right)\right\}
\right\}^{-1}
  & {\bf 0} &  \cdots & I_{n_p}
\end{array}
\right),
\end{split}
\end{equation}
where the second equality is due to the definitions
of $U_1$ and $A_{s1}$ which are given by
 \eqref{DefUi} and \eqref{Def A_ij}, respectively.
The above equality means that
\begin{equation}
\begin{split}\nonumber
&\mbox{eig}\left(D\psi_1\left(x\right)\right)
\\
&=\left\{1\right\}
\bigcup\mbox{eig}\left(
\left\{\nabla^2
\varphi_1\left(x\text{\scriptsize$(1)$}\right)
-\alpha A_{11}\right\}
\left\{
\nabla^2\varphi_1
\left\{\left[\nabla \varphi_1\right]^{-1}
\left(\nabla\varphi_1\left(x\text{\scriptsize$(1)$}\right)
-\alpha\nabla_1f\left(x\right)\right)\right\}
\right\}^{-1}
\right).
\end{split}
\end{equation}
Moreover,
\begin{equation}
\begin{split}
\nabla^2
\varphi_1\left(x\text{\scriptsize$(1)$}\right)
-\alpha A_{11}
&\succeq
\nabla^2
\varphi_1\left(x\text{\scriptsize$(1)$}\right)
-\alpha LI_{n_1}
\\
&\succ
\nabla^2
\varphi_1\left(x\text{\scriptsize$(1)$}\right)
-\frac{\mu}{L}LI_{n_1}
\\
&=\nabla^2
\varphi_1\left(x\text{\scriptsize$(1)$}\right)
-\mu I_{n_1}
\succeq{\bf 0},
\end{split}
\end{equation}
where the first inequality  holds because of
$A_{11}=\nabla_{11}^2 f(x)$,
\eqref{Lipschitz} and  Lemma 7 in
\cite{Panageas2016Gradient}; the second
 inequality is due to $\alpha < \frac{\mu}{L}$;
 the last inequality thanks to
 \eqref{Strongly Convex}.
Hence, $\nabla^2
\varphi_1\left(x\text{\scriptsize$(1)$}\right)
-\alpha A_{11}$ is an invertible matrix.
Consequently, $D\psi_1\left(x\right)$ is an
invertible matrix as well.

 {\bf (d)}  Note that we have shown $\psi_1$
is bijection, and continuously differentiable.
Since $D\psi_1\left(x\right)$
is invertible  for $\alpha<\frac{\mu}{L}$,
the inverse function theorem guarantees
that $\left[\psi_1\right]^{-1}$
is continuously differentiable. Thus,
$\psi_1$ is a diffeomorphism.

Secondly, it is obvious that similar arguments
can be applied to verify that, $\psi_{s},
s=2, \ldots, p,$ are also diffeomorphisms.
Thus, the proof is completed.
\hfill\end{proof}

\subsection{Proof of Lemma \ref{Re neq 0 Lem P}}
\begin{proof}
Let $\lambda$ be
an eigenvalue of $\left(\beta B\right)^{-1}
\left(I+t\beta \hat{B}\right)$
and $\xi$ be the corresponding eigenvector of
unit length, then $\lambda\neq0$ and
\begin{equation}\label{Eigenvaluevector0 P}
\left(\beta B\right)^{-1}
\left(I_n+t\beta \hat{B}\right)\xi=\lambda\xi,
\end{equation}
which is clearly equivalent to equation:
\begin{equation}\label{Eigenvaluevector1 P}
\left(I_n+t\beta \hat{B}\right)\xi=\lambda
\left(\beta B\right)\xi.
\end{equation}
Premultiplying both sides of the above
equality by $\xi^H$, we arrive at
\begin{equation}\label{RealneqReal P}
1+t\beta\xi^H\hat{B}\xi
=\lambda\xi^H\left(\beta B\right)\xi,
\end{equation}
or equivalently,
\begin{equation}\label{RealneqReal P1}
\lambda=
\frac{1+t\beta\xi^H\hat{B}\xi}
{\beta\xi^HB\xi}.
\end{equation}
Recalling that $0<\beta
<\frac{1}{\rho\left(B\right)}$ and $t\in\left[0,\, 1\right]$,
then Lemma \ref{Bound etaecheckBeta} implies that
$0<\mbox{Re}\left(1+\beta\xi^H\hat{B}\xi\right)<2$.
Combining with the assumptions that
$\mbox{Re}\left(\lambda\right)>0$
and $B$ is a symmetric matrix,
we have
\begin{equation}\label{xiDxigeq0}
\beta\xi^HB\xi>0.
\end{equation}
We rewrite $\tilde{B}$ defined by \eqref{DiagB} below:
 \begin{equation}\label{DiagD}
\tilde{B}=
\mbox{Diag}
\left(B_{11},B_{22},\ldots, B_{pp}\right),
\end{equation}
whose main diagonal blocks are  the same as those of $B$.
Therefore, $B$ has the following decomposition:
\begin{equation}\label{Decomp D}
B=\hat{B}+\tilde{B}+\hat{B}^T.
\end{equation}
In addition, Theorem 4.3.15 in \cite{Horn:1985:MA:5509}
means that
\begin{equation}\label{Decomp D1}
-\rho\left(B\right)\preceq
\tilde{B}\preceq\rho\left(B\right).
\end{equation}
 Hence, if $t\in \left[\frac{1}{2},\, 1\right]$, then
\begin{equation}
\begin{split}
\mbox{Re}\left(\lambda\right)-\frac{1}{2}
&=
\frac{\mbox{Re}\left(1+t\beta\xi^H\hat{B}\xi\right)}
{\beta\xi^HB\xi}-\frac{1}{2}\\
&=\frac{
2\mbox{Re}\left(1+t\beta\xi^H\hat{B}\xi\right)
-\beta\xi^HB\xi
}
{2\beta\xi^HB\xi}\\
&=\frac{
2\mbox{Re}\left(1+t\beta\xi^H\hat{B}\xi\right)
-\beta\xi^H
\left(\hat{B}+\tilde{B}+\hat{B}^T\right)\xi
}
{2\beta\xi^HB\xi}\\
&=\frac{
2+2t\beta\mbox{Re}\left(\xi^H\hat{B}\xi\right)
-2\beta
\mbox{Re}\left(\xi^H\hat{B}\xi\right)
-
\beta\xi^H\tilde{B}\xi
}
{2\beta\xi^HB\xi}\\
&=\frac{
2+2(t-1)\beta\mbox{Re}\left(\xi^H\hat{B}\xi\right)
-
\beta\xi^H\tilde{B}\xi
}
{2\beta\xi^HB\xi}\\
&\geq\frac{
2+2(t-1)\beta\rho\left(B\right)
-
\beta\rho\left(B\right)
}
{2\beta\xi^HB\xi}\\
&\geq\frac{
2-2\beta\rho\left(B\right)
}
{2\rho\left(B\right)}\\
&=\frac{
1-\beta\rho\left(B\right)
}
{\beta\rho\left(B\right)}
\\
&
>0,
\end{split}
\end{equation}
where the third equality is due to \eqref{Decomp D};
the first inequality thanks to Eqs. \eqref{leq Re ch D geq},
 \eqref{xiDxigeq0} and \eqref{Decomp D1};
the second inequality holds because
of $t\in \left[\frac{1}{2},\, 1\right]$; the last
inequality holds because of $\beta\in
\left(0, \frac{1}{\rho\left(B\right)}\right)$.
If $t\in \left[0,\, \frac{1}{2}\right]$, then
 \begin{align}
\begin{split}
\mbox{Re}\left(\lambda\right)-\frac{1}{2}
&=
\frac{\mbox{Re}\left(1+t\beta\xi^H\hat{B}\xi\right)}
{\beta\xi^HB\xi}-\frac{1}{2}\\
&=\frac{
2\mbox{Re}\left(1+t\beta\xi^H\hat{B}\xi\right)
-\beta\xi^HB\xi
}
{2\beta\xi^HB\xi}\\
&=\frac{
2+2t\beta\mbox{Re}\left(\xi^H\hat{B}\xi\right)
-\beta\xi^HB\xi
}
{2\beta\xi^HB\xi}\\
&\geq\frac{
2-2t\beta\rho\left(B\right)
-\beta\rho\left(B\right)}
{2\beta\rho\left(B\right)}
\\
&\geq\frac{
2-2\beta\rho\left(B\right)
}
{2\beta\rho\left(B\right)}
\\
&=\frac{
1-\beta\rho\left(B\right)
}
{\beta\rho\left(B\right)}
\\
&>0,
\end{split}
\end{align}
where the first inequality is due to Eqs. \eqref{leq Re ch D geq} and
\eqref{xiDxigeq0}; the second inequality holds because
of $t\in \left[0, \, \frac{1}{2}\right]$;
$\beta\in
\left(0, \frac{1}{\rho\left(B\right)}\right)$
implies the last inequality.

Thus, the proof is finished.
\hfill\end{proof}
\subsection{Proof of Lemma \ref{Key Theorem BCGD P}}
\begin{proof}
We divide the proof into  two cases.

{\bf Case 1: $B$ is an invertible matrix.} 
Therefore, we have
\begin{equation}
\left(\beta
\left(I_n+\beta \hat{B} \right)^{-1}B
\right)^{-1}
=\left(\beta B\right)^{-1}
\left(I_n+\beta \hat{B} \right),
 \end{equation}
 which implies that
\begin{equation}\label{EiginvBP}
\lambda
 \in\mbox{eig}
 \left(\left(I_n+\beta \hat{B} \right)^{-1}B
 \right)\Leftrightarrow
 \frac{1}{\lambda} \in
  \mbox{eig}\left(
  \left(\beta B\right)^{-1}\left(I_n+\beta \hat{B}
 \right)\right).
\end{equation}
For clarity of notation, we use $\sigma$ to
denote the eigenvalue of $ \left(\beta B\right)^{-1}
\left(I_n+\beta \hat{B}\right)$.
Hence, it is sufficient for us to prove that,
 for an arbitrary
$\beta\in\left(0, \frac{1}{\rho\left(B\right)}\right)$,
there is at least
one nonzero eigenvalue $\sigma$ of
$\left(\beta B\right)^{-1}
\left(I_n+\beta \hat{B}\right)$
such that
\begin{equation}\label{Contain}
\sigma\in\Xi(\beta, B).
\end{equation}
Subsequently, we will prove that relation
\eqref{Contain} is true by using
Lemma \ref{Zero Lemma} in Appendix.

We first define an analytic function with $t$ as parameter:
\begin{equation}\label{Def Xzt P}
\begin{split}
\mathcal{X}_t(z)
 &\triangleq
 \det
 \left\{
 zI_n-\left[(1-t)
 \left(\beta B\right)^{-1}+t\left(\beta B\right)^{-1}
 \left(I_n+\beta \hat{B} \right)\right]
 \right\},
  \hspace*{3mm}
  0\leq t\leq1
  \\
  &=
 \det
 \left\{
 zI_n-\left(\beta B\right)^{-1}
 \left(I_n+t \beta \hat{B} \right)
 \right\},
  \hspace*{3mm}
  0\leq t\leq1.
  \end{split}
  \end{equation}

In order to construct a closed region, we define
\begin{equation}\label{Def nu P}
\begin{split}
\nu&\triangleq
\left\|\left(\beta B\right)^{-1}\right\|+\frac{1}{\rho\left(B\right)}
\left\|\left(\beta B\right)^{-1}\hat{B}\right\|\\
&\geq \left\|\left(\beta B\right)^{-1}\|
+t\beta\|\left(\beta B\right)^{-1}\hat{B}
\right\|, \hspace*{4mm}\forall \,t \in[0, 1]\\
&\geq\left\|\left(\beta B\right)^{-1}\left(I_n+t\beta \hat{B}\right)
\right\|, \hspace*{4mm}\forall\, t \in[0, 1]
\end{split}
\end{equation}
where the first inequality holds because of
$\beta\in\left(0, \frac{1}{\rho\left(B\right)}\right)$
and $t\in [0, 1]$.
In addition, let $t=0$,
the above equation also means that
\begin{equation}\label{Def nu P1}
\nu\geq\left\|\left(\beta B\right)^{-1}\right\|
\geq \rho\left(\left(\beta B\right)^{-1}\right)
\geq \frac{1}{\beta \rho\left(B\right)}
>
\frac{1}{\beta\rho(B)}
-\frac{1}{2}
=\frac{1}{2}+\frac{1-\beta\rho(B)}{\beta\rho(B)},
\end{equation}
where the second equality is due to the definitions
of spectral norm and spectral radius; the third inequality
thanks to property of spectral radius; the last inequality
holds because of
$\beta\in\left(0, \frac{1}{\rho\left(B\right)}\right)$.

Thus given the above $\nu$ satisfying \eqref{Def nu P}
and \eqref{Def nu P1},
we can define a closed rectangle as
 \begin{equation}\label{Def D P}
 \mathscr{D}
 \triangleq
 \left\{a+bi|
 ~\frac{1}{2}\leq a \leq 2\nu,~-2\nu\leq b\leq2\nu\right\},
\end{equation}
which is a closed region in the complex plane.
Note  its boundary $\partial\mathscr{D}$ consists of a finite number
 of smooth curves. Specifically, define
\begin{equation}
\begin{split}
\gamma_1&\triangleq
\left\{a+bi|
~a=\frac{1}{2},~-2\nu\leq b\leq 2\nu\right\},\\
\gamma_2&\triangleq
\left\{a+bi|
~a=2\nu,~-2\nu\leq b\leq2 \nu\right\},\\
\gamma_3&\triangleq
\left\{a+bi|
~\frac{1}{2}\leq a\leq 2\nu,~b=2\nu\right\},\\
\gamma_4&\triangleq
\left\{a+bi|
~\frac{1}{2}\leq a\leq 2\nu,~b=-2\nu\right\},
\end{split}
\end{equation}
then
\begin{equation}\label{Boundary D P}
\begin{split}
\partial\mathscr{D}=
\gamma_1\cup\gamma_2\cup\gamma_3\cup\gamma_4.
\end{split}
\end{equation}

In order to apply Lemma \ref{Zero Lemma},
we will show that
\begin{equation}\label{Neq BoundD P}
 \mathcal{X}_t(z)\neq0,~\forall\, t\in[0,1],
~\forall\, z\in\partial\mathscr{D}.
\end{equation}
On the one hand, since the spectral norm of a
matrix is lager than or equal to its spectral
radius, the above inequality \eqref{Def nu P} yields that,
for any $t\in [0, 1]$,
every eigenvalue of $B^{-1}\left(I+t\beta \check{B}\right)$
has a magnitude less than $\nu$. Note that
for an arbitrary
$z\in \gamma_2\cup\gamma_3\cup\gamma_4$, then
$|z|\geq2\nu$.
 Consequently,
\begin{equation}\label{Neq Gamma234 P}
 \mathcal{X}_t(z)\neq0,~\forall\, t\in[0,1],
~\forall\, z\in \gamma_2\cup\gamma_3\cup\gamma_4.
\end{equation}
On the other hand, if $\sigma$ is an eigenvalue of
$\left(\beta B\right)^{-1}\left(I_n+t\beta \hat{B}
\right)$ with any $t\in\left[0,\,1\right]$,
and $\mbox{Re}\left(\sigma\right)>0$,
then  Lemma \ref{Re neq 0 Lem P} implies that,
\begin{equation}
\mbox{Re}\left(\sigma\right)>\frac{1}{2},
\end{equation}
which immediately implies that
\begin{equation}
\sigma\not\in \gamma_1.
\end{equation}
 Hence, we have
\begin{equation}\label{Neq Gamma1 P}
 \mathcal{X}_t(z)\neq0,~\forall\, t\in[0,1],
~\forall\, z\in \gamma_1.
\end{equation}
Combining \eqref{Neq Gamma234 P} and \eqref{Neq Gamma1 P},
we obtain \eqref{Neq BoundD P}.

As a result, it follows from Lemma \ref{Zero Lemma}
in Appendix,
\eqref{Def Xzt P}, \eqref{Def D P} and \eqref{Neq BoundD P}
that
$\mathcal{X}_{0}(z)
=\det
 \left\{
 zI_n-\left(\beta B\right)^{-1}
 \right\}$
 and
 $\mathcal{X}_{1}(z)
 =\det
 \left\{
 zI_n-\left(\beta B\right)^{-1}
 \left(I_n+\beta \hat{B} \right)
 \right\}$
have the  same number of zeros in $\mathscr{D}$.
Note that $\lambda_{\max}\left(B\right)>0$
implies that there is at least one  positive eigenvalue
$\frac{1}{\beta\lambda_{\max}\left(B\right)}$
of $\left(\beta B\right)^{-1}$.
Note that $\frac{1}{\beta\lambda_{\max}\left(B\right)}
\geq \frac{1}{\beta\rho\left(B\right)}>1$.
Recalling the definition
\eqref{Def nu P} of $\nu$, we know
$\left|\frac{1}
{\beta\lambda_{\max}\left(B\right)}
\right|\leq\nu$.
Thus $\frac{1}{\beta \lambda_{\max}
\left(B\right)}$  must
lie inside $\mathscr{D}$.
In other words, the number of zeros of
$\mathcal{X}_{0}(z)$ inside $\mathscr{D}$ is at least one,
which in turn shows the number of zeros of
$\mathcal{X}_{1}(z)$ is at least one as well.
Thus, there must exist at least one eigenvalue
of $\left(\beta B\right)^{-1}
\left(I_n+\beta \hat{B}\right) $ lying
inside $\mathscr{D}$.
We denote it as $\sigma$, then
$\mbox{Re}\left(\sigma\right)>\frac{1}{2}$.
Moreover, Lemma \ref{Re neq 0 Lem P} means that
$\mbox{Re}\left(\sigma\right)
\geq \frac{1}{2}
+\frac{1-\beta\rho(B)}{\beta\rho(B)}$.
Hence, $\sigma$ lies in $\Xi(\beta, B)$ defined by \eqref{Xi}
and the proof is finished in this case.
\\

{\bf Case 2: $B$ is a singular matrix.} In this
case, we will apply  perturbation  theorem based
on the results in {\bf Case 1} to
prove \eqref{Key P}.\\

Suppose the multiplicity of zero eigenvalue of $B$ is
$m$. For clarity of notation, we rewrite the
eigen decomposition \eqref{B eig decom} of $B$ below:
\begin{equation}\label{D eig decom}
B=V
\left(
  \begin{array}{cc}
    \Theta  & 0 \\
    0 & 0 \\
  \end{array}
\right)V^T=V_1\Theta V_1^T,
\end{equation}
where $\Theta=\mbox{Diag}\left(\theta_1,
\theta_2, \ldots, \theta_{n-m}\right)$,
$\theta_s$, $s=1, \ldots, n-m$,
 are the nonzero eigenvalues of
$B$ and
\begin{equation}
V=\left(
\begin{array}{cc}
    V_1 & V_2
  \end{array}
\right)
\end{equation}
is an orthogonal matrix and $V_1$ consists of
the first $(n-m)$ columns of $V$.

Denote
\begin{equation}
\delta
\triangleq\min
\left\{\left|\theta_1\right|,
\left|\theta_2\right|, \ldots,
\left|\theta_{n-m}\right|\right\}.
\end{equation}
For any $\epsilon\in\left(-\delta, \, 0,\right)$, we define
\begin{equation}\label{Def Depsilon0}
B\left(\epsilon\right)\triangleq B+\epsilon I_n,
\end{equation}
then,
\begin{equation}
\mbox{eig}\left(B\left(\epsilon\right)\right)
=\left\{\theta_1+\epsilon,
\theta_2+\epsilon, \ldots,
\theta_{n-m}+\epsilon, \epsilon\right\}\not\ni0,
\hspace*{3mm}
\forall\, \epsilon\in\left(-\delta,\, 0\right),
\end{equation}
and
\begin{equation}\label{laminleq0 P}
\lambda_{\max}\left(B\left(\epsilon\right)\right)=
\lambda_{\max}\left(B\right)+\epsilon
\geq\delta+\epsilon>0,\hspace*{3mm}
\forall\, \epsilon \in \left(-\delta,\, 0\right),
\end{equation}
where the first inequality  is due to the definition
of $\delta$ and
$\max
\left\{\theta_1,
\theta_2, \ldots,
\theta_{n-m}\right\}=\lambda_{\max}\left(B\right)>0$.

Since $B$ is defined by \eqref{Def B},
 $B\left(\epsilon\right)$  has
$p\times p$ blocks form as well.
Specifically,
\begin{equation}\label{Def Depsilon}
B\left(\epsilon\right)
=
\left(B\left(\epsilon\right)_{st}\right)_{1\leq s, \,t\leq p},
\end{equation}
and  its $(s, t)$-th block is given
\begin{equation}\label{Def Depsilon_ij}
B\left(\epsilon\right)_{st}
=
\left\{
\begin{array}{lr}
B_{st}+\epsilon I_{n_s},   & s=t,\\
B_{st},                    & s\neq t,
  \end{array}
  \right.
\end{equation}
where $n_1$, $n_2$, $\ldots$,
$n_p$ are $p$
positive integer numbers satisfying
$\sum\limits_{s=1}^{p}n_s=n$.
Similar to definition \eqref{hat B}, we denote the
strictly block upper triangular matrix
based on $B\left(\epsilon\right)$ as
 \begin{equation}\label{Def checkDepsilon}
\hat{B}\left(\epsilon\right)\triangleq
 \left(\hat{B}\left(\epsilon\right)_{st}\right)_{1\leq s,\,t\leq p}
 \end{equation}
 with $p\times p$ blocks and its $(s, t)$-th block is given by
 \begin{equation}\label{LsigmabarDepsilon}
 \begin{split}
\hat{B}\left(\epsilon\right)_{st}
&=\left\{
\begin{array}{lr}
    B\left(\epsilon\right)_{st},      & s<t,\\
    \mathbf{0},  & s\geq t,
  \end{array}
  \right.\\
  &=\left\{
\begin{array}{lr}
    B_{st},      & s<t,\\
    \mathbf{0},  & s\geq t,
  \end{array}
  \right.\\
  &=\hat{B}_{st},
  \end{split}
\end{equation}
where the second equality holds because of
\eqref{Def Depsilon_ij}; the last equality is due to
\eqref{hat Bst}.

It follows easily from Eqs. \eqref{hat B},
\eqref{hat Bst}
 \eqref{Def checkDepsilon} and
 \eqref{LsigmabarDepsilon} that
\begin{equation}\label{CheckDepsilon=D}
\hat{B}\left(\epsilon\right)=
\hat{B}.
\end{equation}
Consequently,
\begin{equation*}
\begin{split}
\beta\left(I_n+\beta \hat{B}\left(\epsilon\right)
 \right)^{-1}B\left(\epsilon\right)
 &=\beta\left(I_n+\beta \hat{B}
 \right)^{-1}B\left(\epsilon\right)
  =\beta\left(I_n+\beta \hat{B}
 \right)^{-1}\left(B+\epsilon I_n\right),
 \end{split}
 \end{equation*}
where the first equality  is due to
\eqref{CheckDepsilon=D} and the second equality holds
because of \eqref{Def Depsilon0}.
For simplicity, let
\begin{equation}
\lambda^\beta_s\hspace*{-1.5mm}
\left(\epsilon\right),\hspace*{3mm}
s=1, \ldots, n,
\end{equation}
be the eigenvalues of
$\beta\left(I_n+\beta \hat{B}\right)^{-1}
\left(B+\epsilon I_n\right)$.

Note that for any
$\epsilon\in \left(-\delta, \, 0\right)$,
$B\left(\epsilon\right)$ is invertible and
$\lambda_{\max}\left(B\left(\epsilon\right)
\right)>0$ (see \eqref{laminleq0 P}).
According to the definitions of  $B\left(\epsilon\right)$
and $\check{B}\left(\epsilon\right)$,
a similar argument in {\bf Case 1}
can be applied  with the identifications
$B\left(\epsilon\right)\sim B$,
$\hat{B}\left(\epsilon\right)\sim \hat{B}$,
$\beta \sim \beta$ and
$\rho\left(B(\epsilon)\right)\sim\rho(B)$,
to prove that, for any $\beta\in\left(0,
\frac{1}{\rho(B(\epsilon))}\right)$,
there must exist at least one eigenvalue of
$\left(I_n+\beta \hat{B}
\left(\epsilon\right)\right)^{-1}B
\left(\epsilon\right)$
which lies in $\Xi(\beta, B(\epsilon))$
defined by the following \eqref{Xiepsilon}.
Taking into account definition \eqref{Def Depsilon0}, we
have
$\rho\left(B(\epsilon)\right)\leq
\rho\left(B\right)+\epsilon$.
Hence, for any
$\epsilon\in \left(-\delta, \, 0\right)$
and
$\beta\in \left(0, \,
\frac{1}{\rho\left(B\right)+\epsilon}\right)
\subseteq
\left(0, \,
\frac{1}{\rho\left(B(\epsilon)\right)}\right)$,
 there exists at least one index denoted as
 $s(\epsilon)\in\left\{1, 2, \ldots, n\right\}$
 such that
 \begin{equation}\label{Epsi Belo Xi}
 \frac{1}
 {\lambda^\beta_{s\text{\tiny$(\epsilon)$}}
 \hspace*{-1.5mm}\left(\epsilon\right)
 }\in \Xi(\beta, B(\epsilon)),
 \end{equation}
where
\begin{equation}\label{Xiepsilon}
\Xi(\beta, B(\epsilon))
\triangleq\left\{a+bi\Big{|}a, b\in
\mathbb{R},
\frac{1}{2}+\frac{1-\beta\rho(B(\epsilon))}
{\beta\rho(B(\epsilon))}
\leq a,
i=\sqrt{-1} \right\}.
\end{equation}
Furthermore, it is well known that
the eigenvalues of a matrix $M$ are continuous functions
of the entries of $M$.
Therefore, for any $\beta\in\left(0, \,
\frac{1}{\rho\left(B\right)}\right)$,
$\lambda^\beta_s\hspace*{-1.5mm}\left(\epsilon\right)$ is a
continuous function of $\epsilon$ and
 \begin{equation}\label{Alllimit=0 P}
 \lim\limits_{\epsilon\rightarrow 0^-}
 \lambda^\beta_s\hspace*{-1.5mm}\left(\epsilon\right)=
 \lambda^\beta_s\hspace*{-1.5mm}\left(0\right)
 , \hspace*{3mm}
s=1, \ldots, n,
 \end{equation}
 where $\lambda^\beta_s\hspace*{-1.5mm}\left(0\right)$
is the eigenvalue of
$\beta\left(I_n+\beta \check{B}\right)^{-1}B$.

In what follows, we will prove that
\eqref{Key P} holds true by contradiction.

Suppose for sake of contradiction that,
 there exists a $\beta^*\in\left(0, \,
\frac{1}{\rho\left(B\right)}\right)$
such that, for any $s\in \left\{1, \ldots, n\right\}$,
if $\lambda^{\beta^*}_s\hspace*{-2mm}
\left(0\right)\neq0$, then
\begin{equation}\label{Lambda0 notIn Xi}
\frac{1}
{\lambda^{\beta^*}_s\hspace*{-2mm}
\left(0\right)
}
\notin\Xi\left(\beta^*, B(0)\right)
=\Xi\left(\beta^*, B\right),
\end{equation}
 where $\lambda^{\beta^*}_s\hspace*{-2mm}\left(0\right)$
is the eigenvalue of
$\beta^*\left(I_n+\beta^* \check{B}\right)^{-1}B$.

According to Lemma \ref{Multi Zero Lem} in Appendix and the
assumption that the multiplicity of zero eigenvalue of
$B$ is $m$, we know
that the multiplicity of eigenvalue $0$ of
$\beta\left(I_n+\beta^* \check{B}\right)^{-1}B$ is
$m$ as well. Then there are exactly $m$ eigenvalue functions of
$\epsilon$ whose limits are 0 as $\epsilon$
approaches zero from above. Without loss of generality,
we assume
\begin{equation}\label{LimiLambda=0 P}
 \lim\limits_{\epsilon\rightarrow0^+}
 \lambda^{\beta^*}_s\hspace*{-2mm}\left(\epsilon\right)=
 \lambda^{\beta^*}_s\hspace*{-2mm}\left(0\right)=0, \hspace*{3mm}
 s=1, \ldots, m,
 \end{equation}
and
\begin{equation}\label{LimiLambdaneq0 P}
 \lim\limits_{\epsilon\rightarrow0^+}
\lambda^{\beta^*}_s\hspace*{-2mm}\left(\epsilon\right)=
\lambda^{\beta^*}_s\hspace*{-2mm}\left(0\right)\neq0, \hspace*{3mm}
 s=m+1, \ldots, n.
\end{equation}

Subsequently, 
 under the assumption \eqref{Lambda0 notIn Xi},
we will prove that there
exists a
$\delta^*>0$ with $\delta^* \leq \delta$ such that, for any
$\epsilon\in \left(-\delta^*, 0\right)$, then
$\beta^*\in\left(0, \frac{1}{\rho(B)+\epsilon}\right)
\subseteq\left(0, \frac{1}{\rho(B(\epsilon))}\right)$. This would contradict
\eqref{Epsi Belo Xi} and
there does not exist any $s\in \left\{1, \ldots, n\right\}$
such that
$\frac{1}
{\lambda^{\beta^*}_s\hspace*{-2mm}
\left(\epsilon\right)}$
belongs to
$\Xi\left(\beta^*, B(\epsilon)\right)$
. The proof is given by the following
four steps.\\

{\bf Step (a):} Under the assumption
\eqref{Lambda0 notIn Xi},
we first prove that there exists a
$\delta_1^*>0$ such that, for any
$\epsilon\in \left(-\delta_1^*,\, 0\right)$,
$\beta^*\in\left(0, \frac{1}{\rho(B)+\epsilon}\right)
\subseteq\left(0, \frac{1}{\rho(B(\epsilon))}\right)$
and
there does not exist any $s\in \left\{m+1, \ldots, n\right\}$
such that
$\frac{1}
{\lambda^{\beta^*}_s\hspace*{-2mm}
\left(\epsilon\right)}$
belongs to $\Xi\left(\beta^*, B(\epsilon)\right)$.

Taking into account the definition of
$\Xi\left(\beta^*, B(\epsilon)\right)$,
 Eq. \eqref{Lambda0 notIn Xi} and $\beta^*\in\left(0, \,
\frac{1}{\rho\left(B\right)}\right)$, we
imply that, there exists a $\bar{\delta}$ such that,
for any $\epsilon\in(-\bar{\delta}, 0)$,
\begin{equation}\label{beta*in}
\beta^*\in\left(0, \frac{1}{\rho(B)+\epsilon}\right)
\subseteq\left(0, \frac{1}{\rho(B(\epsilon))}\right)
\end{equation}
and
 \begin{equation}
\mbox{Re}\left(
\frac{1}
{\lambda^{\beta^*}_s\hspace*{-2mm}
\left(0\right)}
\right)
<
\frac{1}{2}+\frac{1-\beta^*\rho(B(\epsilon))}
{\beta^*\rho(B(\epsilon))}
,
\hspace*{3mm} \forall \, s\in\left\{m+1, \ldots, n\right\}.
\end{equation}
Moreover, note that
$\frac{1}
{\lambda^{\beta^*}_s\hspace*{-2mm}
\left(\epsilon\right)}$ is a
continuous function of $\epsilon$
and \eqref{LimiLambdaneq0 P} holds. Combining with the above
inequalities,
we know that there exists a $\delta_1^*>0$ with
$\delta_1^*\leq \bar{\delta}$
such that
\begin{equation*}\label{Neigh Re Lambda m+1 P}
\begin{split}
&\left|
\frac{1}
{\lambda^{\beta^*}_s\hspace*{-2mm}
\left(\epsilon\right)}-
\frac{1}
{\lambda^{\beta^*}_s\hspace*{-2mm}
\left(0\right)}
\right|\\
&
<\frac{1}{3}
\left[
\frac{1}{2}+\frac{1-\beta^*\rho(B(\epsilon))}
{\beta^*\rho(B(\epsilon))}
-\mbox{Re}
\left(
\frac{1}
{\lambda^{\beta^*}_s\hspace*{-2mm}
\left(0\right)}
\right)\right],
\hspace*{3mm}
 \forall\, s\in \left\{m+1, \ldots, n\right\},
 \, \forall\,
  \epsilon\in \left(-\delta_1^*\, 0\right),
\end{split}
\end{equation*}
which further means that, for any
$s\in \left\{m+1, \ldots, n\right\}$ and $
  \epsilon\in \left(-\delta_1^*\, 0\right)$,
\begin{equation*}\label{Neigh Re Lambda m+1 P}
\begin{split}
\mbox{Re}\left(
\frac{1}
{\lambda^{\beta^*}_s\hspace*{-2mm}
\left(\epsilon\right)}
\right)
&<
\mbox{Re}\left(
\frac{1}
{\lambda^{\beta^*}_s\hspace*{-2mm}
\left(0\right)}
\right)
+
\frac{1}{3}
\left[\frac{1}{2}+\frac{1-\beta^*\rho(B(\epsilon))}
{\beta^*\rho(B(\epsilon))}
-\mbox{Re}
\left(
\frac{1}
{\lambda^{\beta^*}_s\hspace*{-2mm}
\left(0\right)}
\right)\right],
\\
&
=\frac{1}{2}+\frac{1-\beta^*\rho(B(\epsilon))}
{\beta^*\rho(B(\epsilon))}
-\frac{2}{3}
\left[
\frac{1}{2}+\frac{1-\beta^*\rho(B(\epsilon))}
{\beta^*\rho(B(\epsilon))}
-\mbox{Re}
\left(
\frac{1}
{\lambda^{\beta^*}_s\hspace*{-2mm}
\left(0\right)}
\right)\right]
\\
&<\frac{1}{2}+\frac{1-\beta^*\rho(B(\epsilon))}
{\beta^*\rho(B(\epsilon))}.
\end{split}
\end{equation*}
Hence, we arrive at
\begin{equation}\label{m+1notinXi}
\frac{1}
{\lambda^{\beta^*}_s\hspace*{-2mm}
\left(\epsilon\right)}\notin\Xi\left(\beta^*, B(\epsilon)\right),
\hspace*{3mm}
 \forall\, s\in \left\{m+1, \ldots, n\right\},
 \, \forall\,
  \epsilon\in \left(-\delta_1^*,\, 0\right).
\end{equation}

{\bf Step (b):} In this step,  by the same arguments
as in {\bf Step (b)} in the proof of Lemma
\ref{Key Lemma BCGD} (see \eqref{Cij}$-$\eqref{mnotinOmega}),
we can prove that there exists a
$\delta_2^*>0$ such that, for any
$\epsilon\in \left(-\delta_2^*,\, 0\right)$
and $s\in \left\{1, \ldots, m\right\}$,
\begin{equation}
\mbox{Re}\left(
\lambda^{\beta^*}_s\hspace*{-2mm}
\left(\epsilon\right)\right)
<0,
\end{equation}
or equivalently,
\begin{equation}
\mbox{Re}\left(
\frac{1}
{\lambda^{\beta^*}_s\hspace*{-2mm}
\left(\epsilon\right)}
\right)<0,
\end{equation}
 which immediately implies that
\begin{equation}\label{Lambdas notIn Xi}
\frac{1}
{\lambda^{\beta^*}_s\hspace*{-2mm}
\left(\epsilon\right)}\notin\Xi\left(\beta^*, B(\epsilon)\right),
\hspace*{3mm}
 \forall\, s\in \left\{m+1, \ldots, n\right\},
 \, \forall\,
  \epsilon\in \left(-\delta_2^*,\, 0\right).
\end{equation}

{\bf Step (c):}
Combining \eqref{m+1notinXi} and
\eqref{Lambdas notIn Xi},
we have the following conclusion
\begin{equation}\label{mnotinXi1}
\begin{split}
\frac{1}
{\lambda^{\beta^*}_s\hspace*{-2mm}
\left(\epsilon\right)}
\notin\Xi\left(\beta^*, B(\epsilon)\right),
\hspace*{3mm}
 \forall\, s\in \left\{1, \ldots, n\right\},
 \, \forall\,
  \epsilon\in \left(-\delta_3^*,\, 0\right),
\end{split}
\end{equation}
where $\delta_3^*=\min\left\{\delta_1^*,
\, \delta_2^*\right\}$.

{\bf Step (d):} 
Let
\begin{equation}\label{Def deltastar P}
\delta^*=\min\left\{\frac{1}{2}\delta,
 \,\delta_3^*\right\}.
\end{equation}
Then, for any $\epsilon\in\left(
-\delta^*, \, 0\right)$, we have
\begin{equation}\label{delta P}
\epsilon\in\left(-
\delta,\, 0\right),
\end{equation}
\begin{equation}\label{beta in 0 epsilonstar P}
\beta^*\in\left(0, \,
\frac{1}{\rho\left(B\right)
+\epsilon}\right)\subseteq
\left(0, \,
\frac{1}{\rho\left(B(\epsilon)\right)}\right)
\end{equation}
and
\begin{equation}\label{mnotinOmega12 P}
\begin{split}
\frac{1}
{\lambda_s^{\beta^*}\hspace*{-2mm}\left(\epsilon\right)
}\notin\Xi\left(\beta^*, B(\epsilon)\right),
\hspace*{3mm}
 \forall\, s\in \left\{1, \ldots, n\right\},
\end{split}
\end{equation}
where \eqref{delta P} uses the definition
 \eqref{Def deltastar P}
of $\delta^*$; \eqref{beta in 0 epsilonstar P} is due to
the definition \eqref{Def deltastar P}
of $\delta^*$ (i.e., $\delta^*\leq \delta^*_3\leq\delta_1^*
\leq\bar{\delta}$)
and \eqref{beta*in}; and
\eqref{mnotinOmega12 P} thanks to
the definition \eqref{Def deltastar P}
of $\delta^*$ and
 \eqref{mnotinXi1}.
Clearly, this contradicts \eqref{Epsi Belo Xi}.

Hence, we conclude that Eq. \eqref{Key P} holds true.
\hfill\end{proof}

\begin{lemma}\label{Rouche's Theorem}
 (Rouche's Theorem: \cite{Conway1973Functions}) Suppose $f$ and $g$ are meromorphic in a neighborhood of $\bar{B}(a; R)$ with no zeros or poles on the circle $\gamma = \{z: |z - a|= R\}$. If $Z_f$ and $Z_g$ ($P_f$ and $P_g$ ) are the number of zeros (poles) of $f$ and $g$ inside $\gamma$ counted according to their multiplicities and if
 $$|f(z)+g(z)|<|f(z)|+|g(z)|$$
 on $\gamma$, then
 $$Z_f-P_f=Z_g-P_g.$$
\end{lemma}

\begin{lemma}\label{Zero Lemma}
 (\cite{Ostrowski1958จน}) Assume
 $M,N\in\mathbb{C}^{n\times n}$ and
define
\begin{equation}
\mathcal{X}_t(z)
 \triangleq\det\left\{zI_n-\left[(1-t)M+tN\right]\right\},
  \hspace*{3mm}
  0\leq t\leq1.
  \end{equation}
Moreover, suppose that  $\mathscr{D}$ is a closed region
 and $\partial\mathscr{D}$ consists of a finite number
 of smooth curves.
 If
 \begin{equation}
 \mathcal{X}_t(z)\neq0,~\forall\, t\in[0,1],
~\forall z\in\partial\mathscr{D},
\end{equation}
 then, for any $t_1, \, t_2\in[0,1]$,
 $\mathcal{X}_{t_1}(z)$ and $\mathcal{X}_{t_2}(z)$
 have the  same number of zeros
  in $\mathscr{D}$ counted according to their multiplicities.
\end{lemma}
\begin{proof}For any $t_0 \in [0,1]$, define
\begin{equation}
u(z) \triangleq \mathcal{X}_{t_{0}}(z)
\end{equation}
and
\begin{equation}
p \triangleq \min_{x \in \partial\mathscr{D}}|\mathcal{X}_{t_{0}}(z)|>0.
\end{equation}
Moreover, for any $t_{0}+\epsilon \in [0,1]$, let
\begin{equation}
v(z) \triangleq \mathcal{X}_{t_{0}+\epsilon}(z)-\mathcal{X}_{t_{0}}(z).
\end{equation}
Since $\mathcal{X}_{t_{0}}(z)$ is a
continuous function on the close set
$[0,1] \times \partial\mathscr{D}$,
the following inequality holds true
\begin{equation}
\max_{z \in \partial\mathscr{D}}|v(z)|<p.
\end{equation}
Hence $|\epsilon|$ is sufficiently small.
Therefore, two single-valued analytic
functions $u(z)$ and $v(z)$
on region $\mathscr{D}$ satisfy
\begin{equation}
|u(z)|>|v(z)|, \,\forall\, z \in \partial\mathscr{D}.
\end{equation}
According to Rouche's theorem (see Lemma
\ref{Rouche's Theorem}),
$u(z)+v(z)=\mathcal{X}_{t_{0}+\epsilon}(z)$
and $u(z)=\mathcal{X}_{t_{0}}(z)$
have the same number of zeros in $\mathscr{D}$.
Hence, given any $t$ in the $|\epsilon|$
neighborhood of $t_0$, the number of zeros
of $\mathcal{X}_{t}(z)$ in the region
$\mathscr{D}$ denoted as $N(t)$,
is a constant. Thereby, $N(t)$ is
a continuous function defined on $[0,1]$.
However, the function $N(t)$ can
only be nonnegative integers.
Therefore, $N(t)$ is a constant function
on $[0,1]$.
\hfill\end{proof}

\begin{lemma}\label{Multi Zero Lem}
Assume that $B$ and $\hat{B}$ are defined by
\eqref{Def B} and \eqref{Def bar B}, respectively,
then, for any
$\beta\in\left(0,\, \frac{1}{\rho\left(B\right)}\right)$, the multiplicity of zero eigenvalue of $\left(I_n+\beta \hat{B} \right)^{-1}B$
is the same as that of zero eigenvalue of $B$.
\end{lemma}
\begin{proof}
Without a loss of generality, suppose the
multiplicity of zero eigenvalue of $B$ is
$m$ and $B$ has an eigen decomposition in the form of
\eqref{B eig decom}. Therefore, in order to
investigate the eigenvalues of
$\left(I_n+\beta \hat{B}\right)^{-1}B$,
its characteristic polynomial is given below:
\begin{equation}\label{Multi of Zero}
\begin{split}
&\det\left\{
\left(I_n+\beta \hat{B}\right)^{-1}
 B
 -\lambda I_n\right\}
 \\
&=
\det\left\{
 B
 -\lambda
 \left(I_n+\beta \hat{B}\right)\right\}
 \\
 &=
\det\left\{
 V
\left(
  \begin{array}{cc}
    \Theta & {\bf 0} \\
    {\bf 0} & {\bf 0} \\
  \end{array}
\right)V^T
 -\lambda
 \left(I_n+\beta \hat{B}\right)\right\}
 \\
 &=
\det\left\{
 \left(
  \begin{array}{cc}
    \Theta & {\bf 0} \\
    {\bf 0} & {\bf 0} \\
  \end{array}
\right)
 -\lambda
 \left(I_n+\beta V^T\hat{B}V\right)\right\}
 \\
 &=
\det\left\{
 \left(
  \begin{array}{cc}
\Theta
-\lambda
\left(I_{n-m}+\beta \check{C}_{11}\right)
& -\lambda
\beta \check{C}_{12} \\
  -\lambda
    \beta \check{C}_{21} &
   -\lambda
    \left(I_{m}+\beta \check{C}_{22}\right) \\
  \end{array}
\right)
 \right\}
\\
 &=\det\left\{\Theta
-\lambda
\left(I_{n-m}+\beta \check{C}_{11}\right)\right\}\\
&\hspace*{4mm}\times\det\left\{
 -\lambda
    \left(I_{m}+\beta \check{C}_{22}\right) -
    \lambda^2\beta^2\check{C}_{21}
    \left[\Theta
-\lambda
\left(I_{n-m}+\beta \check{C}_{11}\right)\right]^{-1}
    \check{C}_{12}
\right\}\\
&=\lambda^m\det\left\{\Theta
-\lambda
\left(I_{n-m}+\beta \check{C}_{11}\right)\right\}\\
&\hspace*{4mm}\times\det\left\{
 -
    \left(I_{m}+\beta \check{C}_{22}\right) -
    \lambda\beta^2\check{C}_{21}
    \left[\Theta
-\lambda
\left(I_{n-m}+\beta \check{C}_{11}\right)\right]^{-1}
    \check{C}_{12}
\right\},
\end{split}
\end{equation}
where the second equality holds because of
the eigen decomposition form \eqref{B eig decom}
of $B$ and the fourth equality is due to the
definitions \eqref{Cij} of $C_{ij}$, $i, j=1, 2$.
Note that $\Theta$ is invertible, we know zero
is not a root of polynomial
$\det\left\{\Theta
-\lambda
\left(I_{n-m}+\beta \check{C}_{11}\right)\right\}$.
In addition, assume that $\lambda$ is an eigenvalue of
 $I_{m}+\beta \check{C}_{22}$ and
 $\upsilon\in\mathbb{C}^m$ is  its eigenvector of length
 one, then
\begin{equation}\label{ReLowBoC22}
\begin{split}
\mbox{Re}
\left(\lambda\right)
&=\mbox{Re}
\left(\upsilon^H\left(I_m+\beta
    \check{C}_{22}\right)
    \upsilon
    \right)\\
&=\mbox{Re}
\left(1+\beta
    \upsilon^H
    \check{C}_{22}
    \upsilon
    \right)\\
    &=\mbox{Re}
\left(
1+\beta
    \upsilon^H
    V_2^T\hat{B}V_2
    \upsilon
    \right)\\
&=\mbox{Re}
\left(
    1+\beta
   \upsilon^H
    V_2^H\hat{B}V_2
    \upsilon
    \right)\\
    &=\mbox{Re}
    \left(
    1+\beta
    \left(V_2\upsilon\right)^H
   \hat{B}V_2\upsilon
    \right)\\
    &\geq 1-\frac{\beta}{\rho\left(B\right)}
    =\frac{\rho\left(B\right)-\beta}{\rho\left(B\right)}>0,
    \end{split}
    \end{equation}
where the first equality follows from \eqref{Cij};
the second equality is because $V_1$ is a real
matrix in equation \eqref{B eig decom}; it follows
the last equality from Lemma \ref{Re neq 0 Lem} and
$\left\|V_1\upsilon\right\|=\left(V_1\upsilon\right)^H
V_1\upsilon=\upsilon^H\upsilon=1
$. The above inequality also shows that
the real part of the eigenvalues of
$I_{m}+\beta \check{C}_{22}$ is not zero. Hence, zero is
not one of its eigenvalues.

The above analysis shows that zero is not one
of the eigenvalues of $\Theta -\lambda
\big{(}I_{n-m}+\beta \check{C}_{11}\big{)}$
nor $I_m+\beta\check{C}_{22}$. Combined with the
expression \eqref{Multi of Zero}, the multiplicity of
zero eigenvalue of $\left(I_n+\beta
\hat{B} \right)^{-1}B$ is exactly $m$.
\hfill\end{proof}
\begin{lemma}\label{Phi Diff}
Suppose that $\phi(x)$ is a strongly convex twice
continuously differentiable function with parameter
$\sigma>0$, i.e., for any $y$ and $x\in\mathbb{R}^{n}$,
\begin{equation}\label{Lem Strongly Convex}
\phi\left(x\right)\geq
\phi\left(y\right)
+\left\langle x-y,\nabla \phi\left(y\right)
\right\rangle
+\frac{\sigma}{2}\left\|x-y\right\|^2,
\end{equation}
then its gradient mapping
$\nabla\phi: \mathbb{R}^{n} \to\mathbb{R}^{n}$
is a diffeomorphism.
\end{lemma}
\begin{proof}
Since $\phi$ is a $\sigma$-strongly convex function,
it  follows easily from \eqref{Lem Strongly Convex} that
\begin{eqnarray}\label{Norm Ine}
\left\langle\nabla\phi(y)
-\nabla\phi(x),y-x\right\rangle\geq\sigma\|y-x\|^{2}.
\end{eqnarray}
We first check that $\nabla\phi$ is injective
from $\mathbb{R}^n\rightarrow \mathbb{R}^n$.
Suppose that there  exist $x$ and $y$ such that $\nabla\phi(x)
=\nabla\phi(y)$. Then we would have
\begin{eqnarray}
\|y-x\|^{2}
\leq\frac{1}{\sigma}\langle\nabla\phi(y)
-\nabla\phi(x),y-x\rangle=0,\nonumber
\end{eqnarray}
which means $x=y$.

To show the gradient map $\nabla\phi$ is surjective,
we construct  an explicit inverse function $(\nabla\phi)^{-1}$
given by
$$x_{y}=\arg\min\limits_{x}\phi(x)-\langle y,x\rangle.$$
Since $\phi$ is a $\sigma$-strongly convex function,
the function above is strongly convex with respect
to $x$. Hence, there is a unique minimizer and we denote it
 as $x_{y}$. Then by the KKT condition,
\begin{eqnarray}
\nabla\phi(x_{y})=y\nonumber.
\end{eqnarray}
Thus, $x_{y}$ is mapped to $y$ by the mapping
$\nabla\phi$. Consequently, $\nabla\phi$
is bijection. The assumptions also mean it is
continuously differentiable. \\

In addition, for any $x\in\mathbb{R}^n$,
 \begin{align}
 \begin{split}\label{DVarphiGeq0}
 D\nabla\phi(x)=\nabla^{2}\phi(x)\succeq\mu I_n\succ{\bf 0},
 \end{split}
 \end{align}
 where the last but one inequality holds
 because of $\sigma$ strong convexity of $\phi$
\citep{Nesterov2004Introductory}.
Therefore, $D\nabla\phi(x)$ is invertible, and
the inverse function theorem guarantees
$(\nabla\phi)^{-1}$ is continuously differentiable.
Thus $\nabla\phi$ is
a diffeomorphism.
\hfill\end{proof}

\bibliographystyle{abbrvnat}
\bibliography{strings,stats}

\def\noopsort#1{}\def\l{\char32l}\def\v#1{{\accent20 #1}}
  \let\^^_=\v\def\hbk{hardback}\def\pbk{paperback}
\begin{thebibliography}{25}
\providecommand{\natexlab}[1]{#1}
\providecommand{\url}[1]{\texttt{#1}}
\expandafter\ifx\csname urlstyle\endcsname\relax
  \providecommand{\doi}[1]{doi: #1}\else
  \providecommand{\doi}{doi: \begingroup \urlstyle{rm}\Url}\fi

\bibitem[Auslender and Teboulle(2006)]{Auslender2006Interior}
A.~Auslender and M.~Teboulle.
\newblock Interior gradient and proximal methods for convex and conic
  optimization.
\newblock \emph{SIAM Journal on Optimization}, 16:\penalty0 697--725, 2006.

\bibitem[Bauschke et~al.(2006)Bauschke, Borwein, and
  Combettes]{Bauschke2006Bregman}
H.~H. Bauschke, J.~M. Borwein, and P.~L. Combettes.
\newblock Bregman monotone optimization algorithms.
\newblock \emph{SIAM Journal on Control and Optimization}, 42:\penalty0
  596--636, 2006.

\bibitem[Beck and Teboulle(2003)]{Beck2003Mirror}
A.~Beck and M.~Teboulle.
\newblock Mirror descent and nonlinear projected subgradient methods for convex
  optimization.
\newblock \emph{Operations Research Letters}, 31\penalty0 (3):\penalty0
  167--175, 2003.

\bibitem[Beck and Tetruashvili(2013)]{Beck2013On}
A.~Beck and L.~Tetruashvili.
\newblock On the convergence of block coordinate descent type methods.
\newblock \emph{Siam Journal on Optimization}, 23\penalty0 (4):\penalty0
  2037--2060, 2013.

\bibitem[Bregman(1967)]{Bregman1967The}
L.~M. Bregman.
\newblock The relaxation method of finding the common point of convex sets and
  its application to the solution of problems in convex programming.
\newblock \emph{Ussr Computational Mathematics $\&$ Mathematical Physics},
  7\penalty0 (3):\penalty0 200--217, 1967.

\bibitem[Conway(1973)]{Conway1973Functions}
J.~B. Conway.
\newblock \emph{Functions of One Complex Variable I}.
\newblock Springer-Verlag, 1973.

\bibitem[Fercoq and Richt¨¢rik(2013)]{Fercoq2013Accelerated}
O.~Fercoq and P.~Richt¨¢rik.
\newblock Accelerated, parallel and proximal coordinate descent.
\newblock \emph{Siam Journal on Optimization}, 25\penalty0 (4), 2013.

\bibitem[Ge et~al.(2015)Ge, Huang, Jin, and Yuan]{Ge2015Escaping}
R.~Ge, F.~Huang, C.~Jin, and Y.~Yuan.
\newblock Escaping from saddle points --- online stochastic gradient for tensor
  decomposition.
\newblock \emph{Mathematics}, 2015.

\bibitem[Hirsch et~al.(1977)Hirsch, Pugh, and Shub]{hirsch1977invariant}
M.~W. Hirsch, C.~C. Pugh, and M.~Shub.
\newblock \emph{Invariant Manifolds}.
\newblock Springer-Verlag, 1977.

\bibitem[Hong et~al.(2015)Hong, Razaviyayn, Luo, and Pang]{Hong2015A}
M.~Hong, M.~Razaviyayn, Z.~Q. Luo, and J.~S. Pang.
\newblock A unified algorithmic framework for block-structured optimization
  involving big data.
\newblock \emph{Mathematics}, 2015.

\bibitem[Hong et~al.(2017)Hong, Wang, Razaviyayn, and Luo]{Hong2017Iteration}
M.~Hong, X.~Wang, M.~Razaviyayn, and Z.~Q. Luo.
\newblock Iteration complexity analysis of block coordinate descent methods.
\newblock \emph{Mathematical Programming}, 163\penalty0 (1-2):\penalty0
  85--114, 2017.

\bibitem[Horn and Johnson(1986)]{Horn:1985:MA:5509}
R.~A. Horn and C.~R. Johnson.
\newblock \emph{Matrix analysis}.
\newblock Cambridge University Press, New York, NY, USA, 1986.
\newblock ISBN 0-521-30586-1.

\bibitem[Juditsky and Nemirovski(2014)]{Juditsky2014Deterministic}
A.~Juditsky and A.~S. Nemirovski.
\newblock Deterministic and stochastic first order algorithms of large-scale
  convex optimization.
\newblock \emph{Yes Workshop Stochastic Optimization and Online Learning},
  2014.

\bibitem[Kleinberg et~al.(2009)Kleinberg, Piliouras, and
  Tardos]{Kleinberg2009Multiplicative}
R.~Kleinberg, G.~Piliouras, and E.~Tardos.
\newblock Multiplicative updates outperform generic no-regret learning in
  congestion games: extended abstract.
\newblock \emph{Proceedings of the Annual Acm Symposium on Theory of
  Computing}, pages 533--542, 2009.

\bibitem[Lee et~al.(2016)Lee, Simchowitz, Jordan, and
  Recht.]{JLeeandMSimchowitz2016}
J.~Lee, M.~Simchowitz, M.~I. Jordan, and B.~Recht.
\newblock Gradient descent converges to minimizers.
\newblock In \emph{Proceedings of the Conference on Computational Learning
  Theory (COLT)}, pages 1246--1257, Columbia University, New York, USA, 2016.

\bibitem[Lee(2013)]{Lee2013}
J.~M. Lee.
\newblock \emph{Introductory Lectures to smooth manifolds}.
\newblock Springer Science and Business Media, 2013.

\bibitem[Nesterov(2004)]{Nesterov2004Introductory}
Y.~Nesterov.
\newblock \emph{Introductory Lectures on Convex Optimization}, volume~87.
\newblock Springer Science and Business Media, 2004.

\bibitem[Nesterov(2012)]{Nesterov2012Effici}
Y.~Nesterov.
\newblock Efficiency of coordinate descent methods on huge-scale optimization
  problems.
\newblock \emph{SIAM Journal on Optimization}, 22\penalty0 (2):\penalty0
  341--362, 2012.

\bibitem[Ostrowski(1958)]{Ostrowski1958¨¹}
A.~Ostrowski.
\newblock \"uber die stetigkeit von charakteristischen wurzeln in
  abh\"angigkeit von den matrizenelementen.
\newblock \emph{Jahresbericht der Deutschen Mathematiker-Vereinigung},
  60:\penalty0 40--42, 1958.

\bibitem[Panageas and Piliouras(2016)]{Panageas2016Gradient}
I.~Panageas and G.~Piliouras.
\newblock Gradient descent converges to minimizers: The case of non-isolated
  critical points.
\newblock \emph{arXiv preprint arXiv}, page 1605.00405, 2016.

\bibitem[Pemantle(1990)]{Pemantle1990Nonconvergence}
R.~Pemantle.
\newblock Nonconvergence to unstable points in urn models and stochastic
  approximations.
\newblock \emph{The Annals of Probability}, pages 698--712, 1990.

\bibitem[Shub(1987)]{Michael1987}
M.~Shub.
\newblock \emph{Global stability of dynamical systems}.
\newblock Springer Science and Business Media, 1987.

\bibitem[Smale(1967)]{Smale1967Differentiable}
S.~Smale.
\newblock Differentiable dynamical systems.
\newblock \emph{Bulletin of the American Mathematical Society}, 73\penalty0
  (6):\penalty0 747--817, 1967.

\bibitem[Spivak(1965)]{Spivak1965Calculus}
M.~Spivak.
\newblock \emph{Calculus on manifolds: a modern approach to classical theorems
  of advanced calculus}.
\newblock Benjamin, 1965.

\bibitem[Teboulle(1997)]{Teboulle1997Convergence}
M.~Teboulle.
\newblock \emph{Convergence of Proximal-Like Algorithms}, volume~7.
\newblock SIAM Journal on Optimization, 1997.

\end{thebibliography}

\end{document}